\documentclass[12pt,reqno]{amsart}
\usepackage[utf8x]{inputenc}
\usepackage{amssymb,amsmath,latexsym}
\usepackage{epsfig}

\usepackage{color}

\usepackage{epic,eepic,wrapfig,color, ifthen}
\usepackage{url}
\newtheorem{thm}{Theorem}[section]

\newtheorem{lemma}[thm]{Lemma}
\newtheorem{prop}[thm]{Proposition}
\newtheorem{defn}[thm]{Definition}

\newtheorem{remark}[thm]{Remark}

\numberwithin{equation}{section}

\newcommand{\formula}[2][nolabel]
{\ifthenelse{\equal{#1}{nolabel}}
 {\begin{align*} #2 \end{align*}}
 {\ifthenelse{\equal{#1}{}}
  {\begin{align} #2 \end{align}}
  {\begin{align} \label{#1} #2 \end{align}}
 }
}
   
 \def\({\left(} \def\){\right)} 
   \def\<{\langle} \def\>{\rangle} 
   
\def\pf{{\medskip\noindent {\bf Proof. }}}
\def\qed{{\hfill $\Box$ \bigskip}}
\def\R{{\mathbb R}}
\def\P{{\mathbb P}}
\def\E{{\mathbb E}}
\def\1{{\bf 1}}
\newcommand{\gener}{{\cal L}}

\renewcommand{\H}{\mathbb{H}}

\newcommand{\cal}[1]{\mathcal{#1}}
 \def\sB {{\cal B}}

  \def\sL {{\cal L}}

 \def\bB {{\mathbb B}}

 \def\bN {{\mathbb N}} 
\def\bP {{\mathbb P}}  \def\bR {{\mathbb R}}

\def\R {{\mathbb R}}

\def\nn{\nonumber}

\def\wt{\widetilde}
\def\wh{\widehat}

\def\E{{\mathbb E}}
\def\P{{\mathbb P}}

\def\bea{\begin{align*}}
\def\eea{\end{align*}}
\def\bee{\begin{equation}}
\def\eee{\end{equation}}

\def\eps{\varepsilon}

\def\wh{\widehat}

 \def\({\left(} \def\){\right)}  \def\<{\langle} \def\>{\rangle}

\makeatletter
\@addtoreset{equation}{section}

\makeatother

\begin{document}
\bibliographystyle{plain}

\title[
Estimates of Dirichlet heat kernels  for SBMs]{ \bf
Estimates of Dirichlet heat kernels  for subordinate Brownian motions}

\author[P.\ Kim]{Panki Kim}
\address[Kim]{Department of Mathematical Sciences,
Seoul National University,
Building 27, 1 Gwanak-ro, Gwanak-gu
Seoul 08826, Republic of Korea}
\curraddr{}
\thanks{
This work was supported by the National Research Foundation of Korea(NRF) grant funded by the Korea government(MSIP) (No. 2016R1E1A1A01941893).
}
\email{pkim@snu.ac.kr}

\author[A.\ Mimica]{Ante Mimica\,$(\dagger)$}
\address[Mimica]{Ante Mimica, *20-Jan-1981\,-- 
	$\dagger$\,9-Jun-2016 
	\\
\url{https://web.math.pmf.unizg.hr/~amimica/}}

 \date{}

\maketitle

\begin{abstract}
In this paper, we discuss estimates of transition densities of subordinate Brownian motions in open subsets of Euclidean space.
When $D$ is  a $C^{1,1}$ domain, we establish sharp two-sided estimates
 for the transition densities of a large class of  subordinate Brownian motions in  $D$ whose scaling order is not necessarily strictly below $2$. Our estimates are explicit and written in terms of the dimension,  the Euclidean distance between two points,  the distance to the boundary and the Laplace exponent of the corresponding subordinator only. 
\end{abstract}



\bigskip

\section{Introduction}

Transition densities of  L\'evy processes killed upon leaving an open set $D$ are
Dirichlet heat kernels of
the generators of such  processes on $D$.
For example, the classical Dirichlet heat kernel, which is  the fundamental solution of the heat equation
in $D$ with zero boundary values, is the transition density of Brownian motion killed upon leaving $D$. 
Since, except in some special cases, explicit forms of the Dirichlet heat kernels are impossible to obtain, 
obtaining sharp estimates of the Dirichlet heat kernels
 has been a  fundamental problem both in probability theory and in analysis. 
 
After the fundamental work in \cite{CKS}, sharp two-sided
estimates for the Dirichlet heat kernel $p_D(t, x, y)$ of non-local operators
in
open sets have been studied a lot (see \cite{BG, BGR, BGR2, BGR3, CK, CKSc, CKSi,  CKS1, CKS4, CKS5, CKS50, CKS7, CKS9, CKS6, CT, K, KK, KSo}). In particular, very recently in \cite{BGR3, CKS9}, sharp two-sided
estimates of $p_D(t,x,y)$ were obtained for a large class of rotationally symmetric L\'evy processes when the radial parts of their characteristic exponents satisfy weak scaling conditions whose upper scaling exponent is strictly less than 2. A still remaining open question in this direction is that,  when the upper scaling exponent is not strictly less than 2, for 
how general discontinuous L\'evy processes one can prove sharp two-sided estimates for their Dirichlet heat kernels. 
In this paper we investigate this question for subordinate Brownian motions, which form a very large class of L\'evy processes.

A subordinate Brownian motion in ${\mathbb R}^d$ is a L\'evy process which
can be obtained by replacing the time of Brownian motion in ${\mathbb R}^d$ by
an independent subordinator (i.e., an increasing L\'evy process starting from 0).
The subordinator used to define the subordinate Brownian motion
$X$ can be interpreted as ``operational'' time or ``intrinsic''
time. For this reason, subordinate Brownian motions have been used
in applied fields a lot.

To obtain the sharp  Dirichlet heat kernel estimates, it is necessary to know the sharp heat kernel estimates in $\R^d$.
Recently heat kernel estimates for  discontinuous Markov processes  have been a very active research area and, 
for a large class of purely  discontinuous Markov processes, the sharp heat kernel estimates were obtained in \cite{BGR1, CKK, CKK3, CK8, CK2, CKW0,  KS15, Szt11, Szt17}. 
But except \cite{M, Szt17},  for the estimates of the heat kernel, a common assumption on the purely  discontinuous Markov processes in $\R^d$ considered so far  
is that their weak scaling orders were always strictly between 0 and 2. 
Very recently in \cite{M}, the second-named author considered a large class of
purely  discontinuous subordinate Brownian motions whose  weak scaling order  is between 0 and 2   {\it including 2}, and succeeded in obtaining sharp heat kernel estimates of such processes. In this sense, 
the results in \cite{M} extend earlier works in \cite{BGR1}. 

Motivated by \cite{M}, the main purpose of this paper is to establish sharp  two-sided  estimates  of $p_D(t,x,y)$  for a large class of  subordinate Brownian motions in $C^{1,1}$ open set whose weak scaling order is not necessarily strictly below $2$. Our estimates are explicit and written in terms of  the dimension $d$, the Euclidian distance $|x-y|$ for $x,y \in D$,  the distance to the boundary of $D$ for $x,y \in D$ and the Laplace exponent of the corresponding subordinator only.  See Section \ref{s:ex} for examples, in particular, \eqref{e:example1}--\eqref{e:example2} for estimates of the Dirichlet heat kernels.

This paper is also motivated by  \cite{BGR2, CK}, and, 
several  results and ideas in  \cite{CK, M} will be used here. It is shown in \cite{BGR2} that, when 
weak scaling orders of characteristic exponents of unimodal  L\'evy processes in $\R^d$ are  strictly below $2$,
 sharp estimates on the survival probabilities 
 for the unimodal  L\'evy processes  can be obtained  without the information on sharp two-sided estimates
 for the Dirichlet heat kernels. Such estimates  in \cite{BGR2} can not be used in the setting of this paper.

We will use the symbol ``$:=$,'' which is read as ``is defined to be.''
In this paper,  
for $a,b\in \R$ we denote $a\wedge b:=\min\{a,b\}$ and $a\vee b:=\max\{a,b\}$.  By $B(x,r)=\{y\in\R^d: |x-y|<r\}$ we denote the open ball around $x\in \R^d$ with radius $r>0$\,. We also use convention $0^{-1}=+\infty$. 
For any open set $V$, we denote by $\delta_V (x)$
the distance of a point $x$ to $V^c$. We sometimes write point $z=(z_1, \dots, z_d)\in \bR^d$ as
$(\wt z, z_d)$ with $\wt z \in \bR^{d-1}$.

Let $B=(B_t,\, t\ge 0)$ be a Brownian motion in
$\R^d$ whose infinitesimal generator is $\Delta$  
and
let $S=(S_t,\, t\ge 0)$ be a subordinator
which is independent of $B$.
The process $X=(X_t:\, t\ge 0)$ defined by  $X_t=B_{S_t}$ is a
rotationally  invariant (unimodal) L\'evy process in $\R^d$ and is called  a  subordinate Brownian motion.
Let {$\phi$} be the {Laplace  exponent of $S$}. 
That is, 
$$\E[\exp\{-\lambda S_t\}]=\exp\{-t\phi(\lambda)\}, \qquad \lambda >0.$$
Then the characteristic exponent of $X$ is   $\Psi(\xi)=\phi(|\xi|^2)$ and the infinitesimal generator $X$ is $\phi(\Delta)=-\phi(-\Delta).$
It is known that the Laplace exponent $\phi$ is a Bernstein function with $\phi(0+)=0$, that is $(-1)^n \phi^{(n)} \le 0$,  for all  $n \ge 1$. 
Thus it 
has a representation 
\begin{align}
\label{e:phi}
 \phi (\lambda) = b \lambda + \int_0^{\infty} (1- e^{-\lambda t})\, \mu (d t) ,
\end{align}
where $b \ge 0$, and $\mu$ is a  measure  
satisfying 
$ \int_0^{\infty} (1 \wedge t ) \mu(d t)< \infty\, , $
which is called the L\'evy measure of $S$ (or $\phi$).
In this paper, we will always assume that $b=0$ and $\mu(0, \infty)=\infty$.
Note that 
$\phi'(\lambda)=\lambda \int_0^{\infty} e^{-\lambda t} \mu (d t)>0$.
Due to the independence of $B$ and $S$, 
the L\'evy measure $\Pi(d x)$ of $X$ has a density $j(|x|)$, 
given by
\begin{equation}\label{eq:jump}
j(r)=\int^{\infty}_{0}(4\pi s)^{-d/2}e^{-\frac{|x|^2}{4s}} \mu(d s), \quad r>0.
 \end{equation}
 It is well known that  there exists $c_0=c_0(d)$ depending only on $d$ such that 
\begin{align}
\label{e:jupper}
j(r) \le  c_0\frac{\phi(r^{-2})}{r^{d}}, \quad r>0 
\end{align}
(see \cite[(15)]{BGR}).
Moreover, since $\mu(0, \infty)=\infty$,
$X$ has transition density $p(t,x,y)=p(t,y-x)=p(t, |y-x|)$ and it is of the form
\begin{equation}\label{eq:heat_kernel}
p(t,x)=\int_{(0,\infty)}(4\pi s)^{-d/2}e^{-\frac{|x|^2}{4s}}\P(S_t\in ds)
\end{equation}
for $x\in \R^d$ and $t>0$\,.

We now introduce the following scaling conditions.

\begin{defn}\label{d:UL}
Suppose  $f$ is a function from $ (0,\infty)$ into $(0,\infty)$.
\begin{itemize}
\item[(1)]
We say that $f$ satisfies the lower scaling condition $ L_a(\gamma, C_L)$ if  there exist $a\ge 0$, $\gamma>0$ and $C_L \in (0, 1]$ such that 
\begin{equation}\label{eq:lowerscaling}
    \frac{f(\lambda t)}{f(\lambda)}\geq C_Lt^{\gamma}\qquad \text{ for all }\quad \lambda>a \text{ and }\ t\geq 1\,.
\end{equation}
We say that $f$ satisfies the lower scaling condition near  infinity if the above constant $a$ is strictly positive  and we say $f$ satisfies the lower scaling condition globally if $a=0$. 
\item[(2) ]
We say $f$ satisfies the upper scaling condition $U_a(\delta, C_U)$  if there exist  $a\ge 0$, $\delta>0$ and $C_{U} \in [1, \infty)$ such that 
\begin{equation}\label{eq:upperscaling}
    \frac{f(\lambda t)}{f(\lambda)}\leq C_{U}t^{\delta}\qquad \text{ for all }\quad \lambda> a \text{ and }\ t\geq 1\,.
\end{equation}
We say $f$ satisfies the upper scaling condition near  infinity  if  the above constant $a$ is strictly positive and we say $f$ satisfies the upper  scaling condition globally if $a=0$. 
\end{itemize}
\end{defn}

For any open set $D\subset \R^d$, the {\it first exit time of $D$} by the process $X$ is defined by the formula 
$\tau_D:= \inf\{t>0: X_t\notin D\}$
and we use $X^D$ to denote the process obtained by killing the process $X$ upon exiting $D$.
By the strong Markov property, it can easily be verified that
\begin{equation}\label{e:hkos1}
p_D(t,x,y)  :=  p(t,x,y) - \E_x [ p(t - \tau_D,
X_{\tau_D},y) : \tau_D < t],\quad 
t>0, x,y \in D,
\end{equation}
is the transition density of $X^D$.
Note that from \eqref{eq:heat_kernel} we see that $ \sup_{|x| \geq \beta, t>0}  p(t,x) < \infty$  for  all  $\beta > 0.$
Using this estimate and the continuity of $p$, it is routine to show that $p_D(t, x, y)$ is symmetric and continuous (see \cite{CZ}).

We say that $D\subset \R^d$ (when $d\ge 2$) is  a $C^{1,1}$ open set with $C^{1, 1}$ characteristics $(R_0, \Lambda)$ if there exist a localization radius $ R_0>0 $ and a constant $\Lambda>0$ such that for every $z\in\partial D$ there exist a $C^{1,1}$-function $\varphi=\varphi_z: \R^{d-1}\to \R$ satisfying $\varphi(0)=0$, $\nabla\varphi (0)=(0, \dots, 0)$, $\| \nabla\varphi \|_\infty \leq \Lambda$, $| \nabla \varphi(x)-\nabla \varphi(w)| \leq \Lambda |x-w|$ and an orthonormal coordinate system $CS_z$ of  $z=(z_1, \cdots, z_{d-1}, z_d):=(\wt z, \, z_d)$ with origin at $z$ such that $ D\cap B(z, R_0 )= \{y=({\tilde y}, y_d) \in B(0, R_0) \mbox{ in } CS_z: y_d > \varphi (\wt y) \}$. 
The pair $( R_0, \Lambda)$ will be called the $C^{1,1}$ characteristics of the open set $D$.
Note that a $C^{1,1}$ open set $D$ with characteristics $(R_0, \Lambda)$ can be unbounded and disconnected, and the distance between two distinct components of $D$ is at least $R_0$.
By a $C^{1,1}$ open set  in $\R$ with a characteristic $R_0>0$, we mean an open set that can be written as the union of disjoint intervals so that the {infimum} of the lengths of all these intervals is {at least $R_0$} and the {infimum} of the distances between these intervals is  {at least $R_0$}.

It is well-known that $C^{1,1}$ open set $D$ with the characteristic $(R_0, \Lambda)$ satisfies the interior and exterior ball conditions with the characteristic $R_1>0$, that is, there exists $R_1>0$ such that the following holds:
for all $x \in D$ with $\delta_D(x) \le R_1$ there exist balls $B_1 \subset D$ and $B_2\subset D^c$ whose radii are $R_1$ such that 
$x\in B_1$ and $\delta_{B_1}(x)=\delta_D(x)=\delta_{\overline{B_2}^c}(x)$.
Without loss of generality 
whenever we consider a $C^{1,1}$ open set $D$ with the characteristic $(R_0, \Lambda)$, 
we will take $R_0$ as the characteristic of the interior and exterior ball conditions of $D$,  that is,  $R_1=R_0$. 

We say that {\it the path distance in a connected open set $U$ is comparable to the Euclidean distance with characteristic $\lambda_1$} if for every $x$ and $y$ in  $U$ there is a rectifiable curve $l$ in $U$ which connects $x$ to $y$ such that the length of $l$ is less than or equal to  $\lambda_1|x-y|$.
Clearly, such a property holds for all bounded $C^{1,1}$ domains (connected open sets), $C^{1,1}$ domains with compact complements, and  a domain consisting of all the points above the graph of a 
bounded globally $C^{1,1}$ function. 

In this paper, for the Laplace exponent $\phi$ of a subordinator, we define the function $H:(0,\infty)\rightarrow [0,\infty)$ by $H(\lambda):=\phi(\lambda)-\lambda\phi'(\lambda)$. The function $H$, which appeared earlier in the work of Jain and Pruitt \cite{JP}, took a central role in \cite{M} in  obtaining the sharp heat kernel estimates of the transition density of the corresponding  subordinate Brownian motion $X$ in $\R^d$. 

Obviously, this function $H$ will also naturally appear in this paper in the estimates of the transition density of $X$ in open subsets. 
Under the weak scaling assumptions on $H$ we will obtain the sharp two-sided estimates of $p_D(t,x,y)$.
Recall that  $\delta_D(x)$ is  the distance between $x$ and the boundary of $D$.

In the main results of this paper, we will impose the following assumption: there exists a positive constant $c>0$ such that 
\begin{align}
\label{e:jdouble}
j(r) \le c j(r+1), \quad r>1.\end{align}

\begin{remark}\label{r:UL0}  {\rm
A Bernstein function $\phi$ is called a complete Bernstein function
if the L\'evy measure $\mu$ has a completely monotone density
$\mu(t)$, i.e., $(-1)^n D^n \mu\ge 0$ for every non-negative integer
$n$. 
Note that, if $\phi$ is a complete Bernstein function then by \cite[Lemma 2.1]{KSV12b}, there exists $c_1
>1$ such that
\begin{equation}\label{ksv-nu infty}
\mu(r)\le c_1\, \mu(r+1), \qquad \forall r>1.
\end{equation}
If $H$ satisfies $ L_a(\gamma, C_L)$ and $ U_a(\delta, C_U)$ with $\delta<2$
then by \cite[Lemma 2.6]{M} and Remark \ref{r:UL}, 
$c^{-2} H(r^{-1}) \le \mu(r, \infty) \le c_2 H(r^{-1})$ for $r<2$. Using the monotonicity of $\mu$ and $ U_a(\delta, C_U)$ of $H$,
it is easy to see that $c^{-3} r^{-1}H(r^{-1}) \le \mu(r) \le c_3 r^{-1} H(r^{-1})$ for $r<2$ (see  the proof \cite[Theorem 13.2.10]{KSV3}).
Therefore, by  \cite[Proposition 13.3.5]{KSV3}, we see that if $\phi$ is a complete Bernstein function and $H$ satisfies $ L_a(\gamma, C_L)$ and $ U_a(\delta, C_U)$ with $\delta<2$, then \eqref{e:jdouble} holds.
}

\end{remark}

We are now ready to state the main result of this paper.

\begin{thm}\label{t:main0}
  Let $S=(S_t)_{t\geq 0}$ be a subordinator with zero drift whose Laplace  exponent  is {$\phi$}  
  and let $X=(X_t)_{t\geq 0}$ be the corresponding  subordinate Brownian motion in $\R^d$. Assume that \eqref{e:jdouble} holds and that   $H$ satisfies $ L_a(\gamma, C_L)$ and $ U_a(\delta, C_U)$ with $\delta<2$ and $\gamma >2^{-1}{\bf 1}_{\delta \ge 1}$
   for some $a > 0$.
Suppose that $D$ is a  $C^{1,1}$ open set in $\R^d$ with characteristics $(R_0, \Lambda)$. 

 \noindent
(a) For every $T>0$, there exist constants
 $c_1, C_0$ and  $a_U>0$ such that 
  for every $(t,x,y) \in (0, T]
\times D \times D$,  
 \begin{align}
& p_{D}(t,x,y) \le  C_0
\left(1\wedge  \frac{1}{ \sqrt{ t\phi(1/\delta_D(x)^2) } }\right)\left(1\wedge  \frac{1}{ \sqrt{ t\phi(1/\delta_D(y)^2) } }\right)  p(t,x/3,y/3) \label{e:tu1} \\
&\le   c_1
 \left(1\wedge  \frac{1}{ \sqrt{ t\phi(1/\delta_D(x)^2) } }\right)\left(1\wedge  \frac{1}{ \sqrt{ t\phi(1/\delta_D(y)^2) } }\right) \nn\\&\qquad   \times
 \left( \phi^{-1}(t^{-1})^{d/2} \wedge
 \left(
 \frac{tH(|x-y|^{-2})}{|x-y|^{d}}+ \phi^{-1}(t^{-1})^{d/2}\exp[-a_U|x-y|^2\phi^{-1}(t^{-1})]\right)\right).
\label{e:tu2}
\end{align}

\noindent
(b) 
When $D$ is an unbounded,  we further assume that  $H$ satisfies $ L_0(\gamma_0, C_L)$ and $ U_0(\delta, C_U)$ with $\delta<2$ and that 
  the path distance in each connected component of $D$
is comparable to the Euclidean distance with characteristic $\lambda_1$. Then
 for every $T>0$ there exist constants
 $c_2,  a_L>0$ such that  for every $(t,x,y) \in (0, T]
\times D \times D$, 

 \begin{align}\label{e:pDl1}
& p_D(t, x, y) \geq c_2^{-1} \left(1\wedge  \frac{1}{ \sqrt{ t\phi(1/\delta_D(x)^2) } }\right)\left(1\wedge  \frac{1}{ \sqrt{ t\phi(1/\delta_D(y)^2) } }\right)\nn\\&  \qquad\times
 \left( \phi^{-1}(t^{-1})^{d/2} \wedge
  \left( \frac{tH(|x-y|^{-2})}{|x-y|^{d}}+ \phi^{-1}(t^{-1})^{d/2}\exp[-a_L|x-y|^2\phi^{-1}(t^{-1})]\right) \right). 
 \end{align}
 
 \noindent
 (c) If $D$ is a bounded $C^{1,1}$ open set,  then for each $T>0$ there exists $c_3\ge 1$ such that 
for every   $(t, x, y)\in [T, \infty)\times D\times D$,
\begin{align*}
c_3^{-1} \, \frac{e^{- t\, \lambda^{D}}}{ \sqrt{ \phi(1/\delta_D(x)^2) \phi(1/\delta_D(y)^2) } } \le p_D(t, x, y) \le c_3\frac{e^{- t\, \lambda^{D}}}{ \sqrt{ \phi(1/\delta_D(x)^2) \phi(1/\delta_D(y)^2) } },
\end{align*}
where $-\lambda^D<0$ is the largest eigenvalue of the generator of $X^D$.
\end{thm}

We emphasize that we put the assumption $\gamma >2^{-1}{\bf 1}_{\delta \ge 1}$ on
lower scaling condition near  infinity, not globally, i.e.,  we don't assume that $\gamma_0 >2^{-1}{\bf 1}_{\delta \ge 1}$ in Theorem \ref{t:main0}(b). 

When $D$ is a {\it half space-like} domain, we have the global estimates  for all $t>0$ on the Dirichlet heat kernel. 

\begin{thm}\label{t:main1}
  Let $S=(S_t)_{t\geq 0}$ be a subordinator with zero drift whose Laplace  exponent  is {$\phi$}  
  and let $X=(X_t)_{t\geq 0}$ be the corresponding subordinate Brownian motion in $\R^d$.  
  Suppose that $D$ is a domain consisting of all the points above the graph of a 
bounded globally $C^{1,1}$ function and $H$ satisfies $ L_0(\gamma, C_L)$ and $ U_0(\delta, C_U)$ with $\delta<2$.  Then there exist  $c \ge 1$ and  $a_L,a_U>0$  such that both \eqref{e:tu2} and \eqref{e:pDl1} hold for all $(t, x, y)\in (0, \infty) \times D\times D$.
 \end{thm}

 The assumption that  $H$ satisfies $ L_a(\gamma, C_L)$ and $ U_a(\delta, C_U)$ with $\delta<2$ in Theorems \ref{t:main0} and \ref{t:main1} allows us to cover several interesting cases where the 
scaling order of the characteristic exponent $\Psi(\xi)=\phi(|\xi|^2)$  of $X$ is  $2$.

The rest of the paper is organized as follows. In Section 2, we 
revisit \cite{M} and improve one of the main results of \cite{M} in Theorem \ref{tm:main}. This result will be used
in Sections 5--7 to show the sharp two-sided estimates of the Dirichlet heat kernel when $\phi$ satisfies the lower scaling condition near  infinity
or $H(\lambda)=\phi(\lambda)-\lambda\phi'(\lambda)$ satisfies the lower and upper scaling conditions near  infinity.
In Section 3 we first 
show that the scale-invariant parabolic Harnack inequality holds with explicit scaling in terms of Laplace exponent. Then using this we give some preliminary interior lower bound of the Dirichlet heat kernel. 
Using such lower bound of the Dirichlet heat kernel, Theorem \ref{tm:main}, \eqref{prop:off-lower2}, and the estimates on exit probabilities in Section 4 
we prove the estimates of the survival probabilities and the sharp two-sided estimates of the transition density $p_D(t, x, y)$ for the killed process $X^D$. This is
done in Sections 5--6. As an application of Theorem \ref{t:main0}, in Section 7
we establish the estimates on the Green functions in bounded $C^{1,1}$ domain. 
Section 8 contains some examples of subordinate Brownian motions and the sharp two-sided estimates of transition density and Green function of them.

In this paper, we use the following notations.
For a Borel set $W$ in $\R^d$, $\partial W$, $\overline{W}$ and $|W|$ denote the boundary,  the closure and the Lebesgue measure of $W$ in $\R^d$,  respectively. For $s \in R$, $s_+:=s \vee 0$ Throughout the rest of this paper, the positive constants
 $a_0, a_1, T_1, M_0, M_1, \wt R, R_*,R_0, R_1, C, C_i$,
$i=0,1,2,\dots $, can be regarded as fixed, 
while the constants $c_i=c_i(a,b,c,\ldots)$, $i=0,1,2,  \dots$, denote generic constants depending on $a, b, c, \ldots$, whose exact values are unimportant.
They start
anew in each statement and each proof.
The dependence of the constants on $\phi, \gamma, \delta, C_L, C_U$ and the dimension $d \ge 1$, 
may not be mentioned explicitly.

\section{Preliminary Heat kernel estimates in $\R^d$}

Throughout this paper we assume that $\phi$ is the Laplace exponent of a subordinator $S$. Without loss of generality we assume that $\phi(1)=1$. 
In this section we revisit \cite{M} and improve the main result of \cite{M}
for the case that $\phi$ satisfies the lower scaling condition near  infinity.

The Laplace exponent $\phi$ belongs to the class of Bernstein functions 
$$\mathcal{BF}=\{f\in C^\infty(0,\infty)\colon f\geq 0,\,  (-1)^{n-1}f^{(n)}\geq 0,\ n\in \bN\}\,$$ with $\phi(0+)=0$. Thus 
$\phi$ has a unique representation 
\begin{equation}\label{eq:bf-repr}
	\phi(\lambda)=b\lambda+\int_{(0,\infty)} (1-e^{-\lambda y})\mu(dy),
\end{equation}
where $b\geq 0$ and $\mu$ is a L\' evy measure satisfying 
$ \int_0^{\infty} (1 \wedge t ) \mu(d t)< \infty $.
Let $\Phi$ be denote the increasing  function
\bee \label{e:1.9}
\Phi(r):=\frac{1}{\phi(1/r^2)}, \qquad r>0.
\eee

The next Proposition is a particular case of \cite[Proposition 2.4]{M}. Note that there is  a typo in \cite[Proposition 2.4]{M}:
$\alpha\phi^{-1}(\beta^{-1})$  in the display there should be $\alpha\phi^{-1}(\beta t^{-1})$. 

\begin{prop}[{\cite[Proposition 2.4]{M}}]\label{prop:sub-low}
There exist constants $\rho\in (0,1)$ and $\tau>0$ such that for every subordinator $S$,
\[
    \P\left(\frac{1}{2\phi^{-1}( t^{-1})}\leq S_t\leq \frac{1}{\phi^{-1}(\rho t^{-1})}\right)\geq \tau\quad \text{ for all }t>0\,.
\]
\end{prop}

We recall the conditions $ L_a(\gamma, C_L)$ and $ U_a(\delta, C_U)$ from Definition \ref{d:UL}.

\begin{remark}\label{r:UL} \rm
Suppose that $f$ is non-decreasing. 

\noindent
(1)
If $f$ satisfies {$ L_b(\gamma, C_L)$} then  $f$ satisfies    { \bf$ L_a(\gamma, (a/b)^{\gamma} C_L)$} for all $a \in (0,  b]$;
\begin{equation}\label{e:lower-scaling}
\frac{f(x \lambda)}{f(\lambda)} \ge C_L(a/b)^{\gamma}x^{\gamma}  \, ,\quad x \ge 1, \lambda\ge a\, .
\end{equation}
In fact, suppose $a \le  \lambda <b$ and $x \ge 1$. Then, 
  $ f(x \lambda)   \ge C_Lx^{\gamma}(\lambda/b)^{\gamma}    f(b) \ge C_Lx^{\gamma}(a/b)  ^{\gamma}f(\lambda)$    if $x \lambda > b$, 
  and  $f(x \lambda)   \ge f(\lambda)  \ge C_L x^{\gamma}(a/b)^{\gamma}f(\lambda)$  if $x \lambda \le b$. 

\noindent
(2)
If $f$ satisfies {$ U_b(\delta, C_U)$}  then  $f$ satisfies   { \bf$ U_a(\delta, C_U{f(b)}/{f(a)})$} for all $a \in (0,  b]$;
\begin{equation}\label{e:upper-scaling}
\frac{f(x \lambda)}{ f(\lambda)} \le C_U\frac{f(b)}{f(a)}x^{\delta}\, ,\quad x \ge 1, \lambda\ge a\, .
\end{equation}
In fact, suppose $a \le  \lambda <b$ and $x \ge 1$. Then, 
  $ f(x \lambda)   \le C_Ux^{\delta} (\lambda/b)^{\delta}   f(b)\le C_Ux^{\delta}    f(b) \le C_Ux^{\delta} f(b)  f(\lambda)/f(a)$    if $x \lambda > b$, 
  and  $f(x \lambda)   \le f(b)  \le C_Ux^{\delta} f(b)  f(\lambda)/f(a)$  if $x \lambda \le b$. 

\end{remark}

Recall that $H(\lambda)=\phi(\lambda)-\lambda\phi'(\lambda)$.
Note that, by the concavity of $\phi$,  $H(\lambda)=\phi(\lambda)-\lambda\phi'(\lambda)\ge 0$. 
Moreover, $H$ is non-decreasing since $H'(\lambda)=-\lambda \phi''(\lambda) \ge 0$.

Using Remark \ref{r:UL}, we have the following. c.f., \cite[Lemma 2.1]{M}. 
\begin{lemma}\label{lem:bf}
\begin{itemize}
\item[(a)] For any $\lambda>0$ and $x\geq 1$,
\[
    \phi(\lambda x)\leq x\phi(\lambda)\qquad \text{ and }\qquad H(\lambda x)\leq x^2H(\lambda)\,.
\]
\item[(b)] 
Assume that the drift $b$ of $\phi$ in the representation (\ref{eq:bf-repr}) is zero. If $H$ satisfies $ L_a(\gamma, C_L)$ (resp. $ U_a(\delta, C_U)$), then $\phi$ satisfies $ L_a(\gamma, C_L)$(resp. $ U_a(\delta \wedge 1, C_U)$).  
Thus if 
either $H$ or $\phi$ satisfies $ L_a(\gamma, C_L)$ and $ U_a(\delta, C_U)$ 
then   for every $M>0$ there exist $c_1, c_2 >0$ such that  
\begin{equation}\label{eqn:poly}
c_1 \Big(\frac Rr\Big)^{2\gamma} \,\leq\,
 \frac{\Phi (R)}{\Phi (r)}  \ \leq \ c_2
\Big(\frac Rr\Big)^{2(\delta \wedge 1)}
\qquad \hbox{for every } 0<r<R<a^{-1}M.
\end{equation}
\end{itemize}
\end{lemma}

By Remark \ref{r:UL} we also  have  
\begin{lemma}\label{lem:inverse}
If $\phi$ satisfies   $ L_a(\gamma, C_L)$ for some $a>0$, then for every $b \in (0, a]$,
\[
	\frac{\phi^{-1}(\lambda x)}{\phi^{-1}(\lambda)}\leq (a/b) C_L^{-1/\gamma} x^{1/\gamma}\qquad \text{ for all }\quad \lambda>\phi(b),\ x\geq 1\,.
\]
\end{lemma}

Throughout this paper, the process $X=(X_t:\, t\ge 0)$ is a subordinate Brownian motion whose characteristic exponent is 
$\phi(|x|^2)$. 
Recall that $x \to  j(|x|)$ is the L\'evy density of the subordinate Brownian motion $X$ defined in \eqref{eq:jump}, which  gives rise to a L\'evy system for $X$ describing the jumps of  $X$; 
For any $x\in \bR^d$, stopping time $\tau$ (with respect to the filtration of $X$), and nonnegative measurable function $f$ on $\bR_+ \times \bR^d\times \bR^d$ with $f(s, y, y)=0$ for all $y\in\bR^d$ and $s\ge 0$ we have
\begin{equation}\label{e:levy}
\E_x \left[\sum_{s\le \tau} f(s,X_{s-}, X_s) \right] = \E_x \left[ \int_0^\tau \left(\int_{\bR^d} f(s,X_s, y) j(|X_s-y|) dy \right) ds \right]
\end{equation}
(e.g., see \cite[Appendix A]{CK2}).

The next lemma holds for every symmetric L\'evy process and it
follows from \cite[(3.2)]{P} and \cite[Corollary 1]{G}.
Recall that $\tau_D$ is the first exit time of $D$ by the process $X$.
\begin{lemma} \label{L:3.3}
 For any  positive
 constants  $a, b$, there exists
$c=c(a,b, \phi)>0$ such that for all $z \in \bR^d$ and $t>0$,
$$
\inf_{y\in B(z, a \Phi^{-1}(t)/2)}\P_y \left(\tau_{B(z,
a \Phi^{-1}(t))}> bt  \right)\, \ge\, c.
$$
\end{lemma}

Recall that $X$ has a transition density $p(t,x,y)=p(t,y-x)=p(t, |y-x|)$ of the form \eqref{eq:heat_kernel}. 
We first consider the estimates of $p(t,x)$ under the assumption that $\phi$ satisfies   $ L_a(\gamma, C_L)$ for some $a>0$. Note that 
$ L_a(\gamma, C_L)$ implies $\lim\limits_{\lambda \to\infty}\phi(\lambda)=\infty$.

By our Remark \ref{r:UL} and \cite[Propositions 3.2 and 3.4]{M}, we have  the following two upper bounds.
\begin{prop}\label{prop:on-up}
If $\phi$ satisfies   $ L_a(\gamma, C_L)$ for some $a>0$, then for every $T>0$ there exists $c=c(T)>0$ such that for all $t \le T$ and $x \in \R^d$,
	\[
		p(t,x)\leq  c\, \phi^{-1}(t^{-1})^{d/2}.
	\]
\end{prop}

\begin{prop}\label{prop:off-up}
If $\phi$ satisfies   $ L_a(\gamma, C_L)$ for some $a>0$, then for every $T>0$ there  exist  
$c_1, c_2>0$ such that for all $t \le T$ and $x \in \R^d$ satisfying $t\phi(|x|^{-2})\leq 1$, 
	\[
		p(t,x)\leq c_1 \left(t|x|^{-d}H(|x|^{-2})+\phi^{-1}(t^{-1})^{d/2}\exp[-c_2|x|^2\phi^{-1}(t^{-1})]\right)\,.
	\]
\end{prop}

\begin{prop}\label{prop:on-low}
If $\phi$ satisfies   $ L_a(\gamma, C_L)$ for some $a>0$, then for every $T>0$ there exists $c=c(T)>0$ such that for all $t \le T$ and $x \in \R^d$,
\[
		p(t,x)\geq c\, \phi^{-1}(t^{-1})^{\frac{d}{2}}\exp[-2^{-1}|x|^{2}\phi^{-1}(t^{-1})].
	\]
	In particular, if additionally $t\phi(M |x|^{-2})\geq  1$ holds for some $M>0$, then we have 
\begin{align}
\label{e:prop:on-low}
   p(t,x)\geq c\, e^{-M/2}\phi^{-1}(t^{-1})^{\frac{d}{2}}\,.
\end{align}	
\end{prop}
\proof 
We closely follow the proof of \cite[Proposition 3.5]{M}.
Let  $\rho\in (0,1)$ be the constant in Proposition \ref{prop:sub-low} and,  without loss of generality, we assume $T \ge \rho\phi^{-1}(a)$.
Using (\ref{eq:heat_kernel}) we get
\begin{align}
p(t,x)
    &\geq (4 \pi)^{-d/2} \int_{\left[2^{-1}\phi^{-1}(t^{-1})^{-1},\phi^{-1}(\rho t^{-1})^{-1}\right]}s^{-d/2}e^{-\frac{|x|^2}{4s}}\P(S_t\in ds)\nonumber \\
  &\geq (4 \pi)^{-d/2} \phi^{-1}(\rho t^{-1})^{d/2}e^{-\frac12 |x|^2\phi^{-1}(t^{-1})}\P\left(2^{-1}\phi^{-1}(t^{-1})^{-1}\leq S_t\leq\phi^{-1}(\rho t^{-1})^{-1}\right)
\label{eq:tmp_2034}.
\end{align}

Let $b=\phi^{-1}(\rho/T)$. Note that, by Lemma \ref{lem:inverse}, we have that for $0<t<T=\rho\phi(b)^{-1}$,
\begin{align}
\label{e:prg1}
  \phi^{-1}(\rho t^{-1})=\phi^{-1}(t^{-1})\frac{\phi^{-1}(\rho t^{-1})}{\phi^{-1}(t^{-1})}\geq (b/a) C_L^{1/\gamma}\rho^{1/\gamma} \,\phi^{-1}(t^{-1}).
\end{align}
Using \eqref{e:prg1}, Proposition \ref{prop:sub-low} and  \eqref{eq:tmp_2034} we get
\[
    p(t,x)\geq c_2 e^{-\frac12 |x|^{2}\phi^{-1}(t^{-1})}\phi^{-1}(t^{-1})^{d/2}\,. 
\]
\qed

We now revisit \cite{M}.
\begin{thm}\label{tm:main}
 Let $S=(S_t)_{t\geq 0}$ be a subordinator with zero drift whose Laplace  exponent  is {$\phi$}  
  and let $X=(X_t)_{t\geq 0}$ be the corresponding  subordinate Brownian motion in $\R^d$ and $p(t,x,y)=p(t,y-x)$ be   the transition density of $X$. 

   If $\phi$ satisfies   $ L_a(\gamma, C_L)$ for some $a>0$, then for every $T>0$ there exist $c_1=c_1(T, a)>1$ and
   $c_2=c_2(T, a)>0$
   such that for all $t \le T$ and $x \in \R^d$, 
  \begin{align}\label{e:t:up1}
 p(t,x) \leq c_1 \left( \phi^{-1}(t^{-1})^{d/2} \wedge \big( t|x|^{-d}H(|x|^{-2})+ \phi^{-1}(t^{-1})^{d/2}e^{-c_2|x|^2\phi^{-1}(t^{-1})}\big)\right),
 \end{align}
   \begin{align}\label{e:t:up1j}
 j(|x|) \leq c_1 |x|^{-d}H(|x|^{-2}),
  \end{align}
and
  \begin{align}\label{e:t:uplow}
c_1^{-1} \phi^{-1}(t^{-1})^{d/2}  \le p(t,x) \leq c_1 \phi^{-1}(t^{-1})^{d/2} , \quad \text{ if } t\phi(|x|^{-2})\geq 1.
 \end{align}
\end{thm}

\pf \eqref{e:t:up1} and \eqref{e:t:up1j}  follow from 
Propositions \ref{prop:on-up} and \ref{prop:off-up}. 
The estimates  \eqref{e:t:uplow} follow from Remark \ref{r:UL}, \cite[Proposition 3.2]{M} and  Proposition \ref{prop:on-low}. 
\qed

\section{Parabolic Harnack inequality and preliminary lower bounds of $p_D(t,x,y)$}

Throughout this section, we assume that {\it $\phi$ has no drift and satisfies   $ L_a(\gamma, C_L)$ for some $a\ge 0$}. Recall that $p_D(t,x,y)$ defined in \eqref{e:hkos1} is the transition density for $X^D$,
the subprocess of $X$ killed upon leaving $D$.

 Let $Z_s:=(V_s, X_s)$ be
the time-space process of $X$, where
$V_s=V_0- s$.
The law of the time-space process $s\mapsto Z_s$ starting from
$(t, x)$ will be denoted as $\mathbb{P}^{(t, x)}$.

\begin{defn}
A non-negative Borel  measurable function
$h(t,x)$ on $\bR \times \bR^d$ is said to be {\it parabolic}
(or {\it caloric})
on $(a,b]\times B(x_0,r)$
if for every relatively compact open subset $U$ of $(a,b]\times B(x_0,r)$,
$h(t, x)=
\mathbb{E}^{(t,x)}
 [h (Z_{\tau^Z_{U}})]$
for every $(t, x)\in U\cap ([0,\infty)\times
  \bR^d)$,
where
$\tau^Z_{U}:=\inf\{s> 0: \, Z_s\notin U\}$.
\end{defn}

Recall that $\Phi(r)=\frac{1}{\phi(1/r^2)}$.
In this section, we will  first  prove that
 $X$ satisfies the scale-invariant
parabolic Harnack inequality  with  explicit scaling in terms of $\Phi$. That is, 
\begin{thm}\label{t:PHI}
Suppose that 
$\phi$ has no drift and satisfies   $ L_a(\gamma, C_L)$ for some $a\ge 0$.  For every  $M>0$, there exist
$c>0$ and $c_1, c_2 \in (0,1)$ depending on $d$, $\gamma$ and $C_L$ (also depending  on  $M$ and $a$ if $a>0$)  
such that for every $x_0\in
\R^d$, $t_0\ge 0$, $R \in (0, a^{-1} M)$ and every non-negative function $u$ on $[0,
\infty)\times \R^d$ that is
  parabolic
 on $(t_0,t_0+4c_1
\Phi(R)]\times B(x_0,
R)$,
$$
\sup_{(t_1,y_1)\in Q_-}u(t_1,y_1)\le c \, \inf_{(t_2,y_2)\in
Q_+}u(t_2,y_2),
$$
where $Q_-=(t_0+c_1 \Phi(R),t_0+2c_1  \Phi(R)]\times B(x_0,
c_2R)$
and $Q_+=[t_0+3c_1 \Phi(R),t_0+ 4c_1 \Phi(R)]\times B(x_0,
c_2R)$.
\end{thm}
Theorem  \ref{t:PHI} clearly implies the elliptic Harnack inequality. Thus this extends the main result of \cite{G}.

To prove Theorem  \ref{t:PHI}, we first observe that for each $c_1, b>0$ and  every $r, t>0$ satisfying $r\phi^{-1}(t^{-1})^{1/2} \ge c_1$ we have 
\begin{align*}
&\phi^{-1}(t^{-1})^{d/2}e^{-br^2\phi^{-1}(t^{-1})}\big/
(t r^{-d} \phi(r^{-2}))= (\phi(r^{-2})t)^{-1}(r \phi^{-1}(t^{-1})^{1/2})^d  e^{-br^2\phi^{-1}(t^{-1})}\\
& \le \sup_{a>0}[(\phi(a^2r^{-2})/\phi(r^{-2})) a^d e^{-ba^2}]\le \sup_{a>0} a^{d} (a \vee 1)^2 e^{-ba^2}=:c_2 <\infty.
\end{align*}
Using this and the fact that $ \phi \ge H$, we see that 
  for each $b>0$ there exists $c=c(b)>0$ such that for all  $t>0, x  \in \R^d$, 
\begin{align}
\label{e:uppcom} \phi^{-1}(t^{-1})^{d/2} \wedge \big( t|x|^{-d}H(|x|^{-2})+ \phi^{-1}(t^{-1})^{d/2}e^{-b|x|^2\phi^{-1}(t^{-1})}\big) \le c (\phi^{-1}(t^{-1})^{d/2}\wedge t|x|^{-d}\phi(|x|^{-2})).
\end{align} 

Thus by \cite{M} (for $a=0$) and  Proposition \ref{prop:on-low} and \eqref{e:t:up1} (for $a>0$) we have the following bounds: for $t \in (0,  T]$ if $a>0$ (for $t >0$ if $a=0$),
\begin{align}
\label{e:upp}
 p(t,x)\leq C\left((\Phi^{-1}(t))^{-d} \wedge \frac{t}{|x|^{d}\Phi(|x|)} \right), \quad x  \in \R^d 
\end{align}
 and 
\begin{align}
\label{e:low}
 p(t,x)\geq C^{-1} (\Phi^{-1}(t))^{-d} e^{-\frac{1}{2}|x|^{2}/(\Phi^{-1}(t))^{2}},
\end{align}
where the above constant $C>1$ depends on $T$ if $a>0$.

Now, using \eqref{e:upp} and \eqref{e:low}  we get the following lower bound.

\begin{prop}\label{P:3.5-550}
Suppose that  $\phi$ has no drift and satisfies   $ L_a(\gamma, C_L)$ for some $a\ge 0$. 
For  every  $M>0$, there exist  constants
$c>0$ and $\eps \in (0,1/2)$ such that for every $x_0\in \R^d$ and $r  \in (0, a^{-1}M)$,
\begin{equation}\label{e:aa11}
 p_{B(x_0,r)}(t, x,y)\ge   c \, \frac 1{(\Phi^{-1}(t))^d}
 \qquad \hbox{for   } x, y\in B(x_0, \eps  \Phi^{-1}(t)) \hbox{ and }
 t\in (0,   \Phi(\eps r)].
\end{equation}
\end{prop}

\proof 
Since the proof for  the case $a=0$ is almost identical to   the proof for  the case $a>0$, we will prove the proposition  for  the case $a>0$ only.
Fix $x_0\in \R^d$ and let   $B_r:=B(x_0,r)$. 
The constant  $\eps \in (0,1/2)$ will be chosen later.
For   $ x, y\in B_{\eps  \Phi^{-1}(t)}$, we have $|x-y| \le 2\eps  \Phi^{-1}(t)$. So, 
\begin{align}
\label{e:new12}
\frac{|x-y|^{2}}{2(\Phi^{-1}(t))^{2}} \le 2\eps^2 \le 1/2.
\end{align}
Now combining   \eqref{e:hkos1}, 
\eqref{e:upp},  \eqref{e:low} and \eqref{e:new12}
we have that for $x, y\in B_{\eps  \Phi^{-1}(t)}$ and $t\in (0,   \Phi(\eps r)]$,
\begin{eqnarray}
&&p_{B_r}(t,x,y)\,\ge\,
C^{-1}\frac{e^{-2 }    }{(
\Phi^{-1}( t))^d}\nonumber\\&&- C \, 
\E^x\left[1_{\{\tau_{B_r}\le
t\}}\left(\frac1{(\Phi^{-1}(t-\tau_{B_r}))^d}\wedge \frac
{t-\tau_{B_r}}{|X_{\tau_{B_r}}-y|^d \Phi(|X_{\tau_{B_r}}-y|)}\right)\right]\label{eq:dipco22},
\end{eqnarray}

Observe that
$$ |X_{\tau_{B_r}}-y|\ge r-  \eps  \Phi^{-1}(t) \ge   
 (\eps^{-1} -\eps) \Phi^{-1}(t) \ge \Phi^{-1}(t),\quad \text{for all } t\in (0,   \Phi(\eps r)]
$$
 and so 
\begin{align}
  \frac
{t-\tau_{B_r}}{|X_{\tau_{B_r}}-y|^d \Phi(|X_{\tau_{B_r}}-y|)} \le 
\frac{t}{( (\eps^{-1} -\eps) \Phi^{-1}(t))^d \Phi(  \Phi^{-1}(t))}
=
\frac{(\eps^{-1} -\eps)^{-d}}{(\Phi^{-1}( t))^d}.\label{eq:dipco33}
\end{align}
Consequently, we have from \eqref{eq:dipco22} and \eqref{eq:dipco33},
\begin{eqnarray*}
p_{B_r}(t,x,y)
&\ge&
\frac{e^{-2 }  C^{-1}}{(\Phi^{-1}( t))^d}-C
\frac{(\eps^{-1} -\eps)^{-d}}{(\Phi^{-1}( t))^d}\\
&\ge& \left(e^{-2 }  C^{-1} -C(\eps^{-1} -1)^{-d} \right)  \frac{1}{
(\Phi^{-1}( t))^d} .
\end{eqnarray*}
Choose $\eps:=((2 e^{2 }  C^{2})^{1/d}+1)^{-1} <1/2$ 
so that 
$e^{-2 }  C^{-1} -C(\eps^{-1} -\eps)^{-d}   \ge e^{-2 }  C^{-1}/2.
$
We now have $ p_{B_r}(t,x,y)\ge  2^{-1} e^{-2 }  C^{-1}
   (\Phi^{-1}(t))^{-d}$ for  $ x, y\in B_{\eps  \Phi^{-1}(t)}$ and $t\in (0,   \Phi(\eps r)]$. \qed

Integrating \eqref{e:upp} and  \eqref{e:aa11}, we obtain that 
there exist constants $c_1,c_2>0$ such that 
\begin{align}\label{e:Euplo}
c_1 \Phi(r) \le \E^x[\tau_{B(x,r)}] \le c_2 \Phi(r), \quad x \in \R^d, r<1.
\end{align}

We say ({\bf UJS}) holds for $J$ if there exists a positive constant $c$ such that
for every $y\in \R^d$,
$$
 J(y) \le \frac{c}{r^d}
\int_{B(0,r)} J(y-z)dz \qquad
\hbox{whenever } r\le  |y|/2 . \eqno({\bf UJS})
 $$

\noindent
{\bf Proof of Theorem \ref{t:PHI}.}
Note that 
 ({\bf UJS}) always holds for our L\'evy density  $x \to  j(|x|)$ since $j$ is non-increasing. (see  \cite[page 1070]{CKK2}). 
Thus, using Proposition \ref{P:3.5-550},  \eqref{e:upp} (for the case $a=0$) and ({\bf UJS}), we see that   
 Theorem \ref{t:PHI} for the case $a=0$  is a special case of \cite[Theorem 1.17 or Theorem 4.3 and (4.11)]{CKW}.
Moreover, using Proposition \ref{P:3.5-550},  \eqref{e:upp} (for the case $a>0$) and ({\bf UJS}), the proof of Theorem \ref{t:PHI} for the case $a>0$ is almost identical to the proof for the case $a=0$ in \cite[Theorem 4.3]{CKW}.
We skip the details.
\qed.

For the remainder of this section, 
we use the convention that $\delta_{D}(\cdot) \equiv \infty$ when $D
=\bR^d$.
For the next two results,
$D$ is an arbitrary nonempty open set.

\begin{prop}\label{step1}
Suppose that $\phi$ has no drift and  satisfies   $ L_a(\gamma, C_L)$ for some $a\ge 0$. 
For every $T>0$ and $b>0$,  
  there exists
 $c=c(a, T,b,  \phi)>0$ such that
$$
p_D(t,x,y) \,\ge\,c\, (\Phi^{-1}(t))^{-d}
$$
for every $(t, x, y)\in (0, a^{-1}T)\times D\times D$ with
$\delta_D (x) \wedge \delta_D (y)
 \ge b \Phi^{-1}(t) \geq 4 |x-y|$.

 \end{prop}

\pf
Using Theorem \ref{t:PHI},  the proof for the case that $\phi$ satisfies   $ L_0(\gamma, C_L)$ is identical to that of \cite[Proposition 3.4]{CK}. Even through the proof is similar, 
for reader's convenience  we provide the proof for the case that
 $\phi$ satisfies   $ L_a(\gamma, C_L)$
 for $a> 0$. 

Without loss of generality we assume $a=1$. 
We fix $b, T>0$ and 
$(t, x, y)\in (0, T)\times D\times D$ satisfying
$\delta_D(x)\wedge \delta_D(y)\geq b \Phi^{-1} (t) \geq 4 |x-y|$.
Since 
$|x-y| \le b\Phi^{-1}(t)/4$, we have
\begin{align}
\label{e:BBB}
B(x, b\Phi^{-1}(t)/4) \subset B(y, b\Phi^{-1}(t)/2)\subset
 B(y, b\Phi^{-1}(t))  \subset D.
\end{align}
Thus, 
by the symmetry of $p_D$,  Theorem \ref{t:PHI} and Lemma \ref{lem:bf}(a),
 there exists $c_1=c_1(b, T)>0$ such that
$$
p_{B(x, b\Phi^{-1}(t)/4)} (t/2, x, w) \, \le  \, p_D(t/2, x, w) \, \le  \,  c_1 p_D(t,x,y) \quad \mbox{for every
} w \in B(x, b\Phi^{-1}(t)/4)  .
$$
This together with Lemma \ref{L:3.3}
yields that there exist $c_2, c_3>0$ such that \begin{eqnarray*}
p_D(t, x, y) &\geq & \frac{c_1^{-1}}{ | B(x, b\Phi^{-1}(t)/4)|}
 \, \int_{B(x, b \Phi^{-1}(t)/4)}
p_{B(x, b\Phi^{-1}(t)/4)} (t/2, x, w)dw \\
&=& c_2 (\Phi^{-1}(t))^{-d} \, \P_x \left( \tau_{B(x, b\Phi^{-1}(t)/4)} >
t/2\right) \,\geq \, c_3 \, (\Phi^{-1}(t))^{-d}.
\end{eqnarray*}
\qed

\begin{prop}\label{step3}
Suppose that $\phi$ has no drift and  satisfies   $ L_a(\gamma, C_L)$ for some $a\ge 0$. 
For every $b, T>0$,
there exists a constant $c=c(a, b, T)>0$ such that 
 \bee\label{e:lb11}p_D(t, x, y)\,\ge \,c\, t\, 
 j(|x-y|)
 \eee
 for every $(t, x, y)\in (0, a^{-1}T) \times D\times D$ with $\delta_D(x)\wedge \delta_D (y) \ge b \Phi^{-1}(t)$ and
$b \Phi^{-1}(t) \leq 4|x-y|$.
\end{prop}

\pf 
Again, using Proposition \ref{step1},  the proof for the case that $\phi$ satisfies   $ L_0(\gamma, C_L)$ is the same as that of \cite[Proposition 3.5]{CK}, and  for reader's convenience  we provide the proof for the case that
 $\phi$ satisfies   $ L_a(\gamma, C_L)$
 for $a> 0$.

Without loss of generality we assume $a=1$.
Throughout the proof we assume that $t \in  (0, T)$. 
By Lemma~\ref{L:3.3}, starting at $z\in B(y, \,  (12)^{-1} b \Phi^{-1}(t))$, with probability at least $c_1=c_1(b, T)>0$ the process $X$ does not move more than $ (18)^{-1} b \Phi^{-1}(t) $ by time $t$.
Thus, using the strong Markov property and the L\'evy system in \eqref{e:levy}, we obtain
\begin{align}
&\P_x \left( X^D_t \in B \big( y, \,  6^{-1} b \Phi^{-1}(t) \big) \right)\nn\\
&\ge  c_1\P_x(X_{t\wedge \tau_{B(x,  (18)^{-1} b \Phi^{-1}(t))}}^D\in B(y,\,  (12)^{-1}b \Phi^{-1}(t))\hbox{ and }t \wedge \tau_{B(x,  (18)^{-1}b  \Phi^{-1}(t))} \hbox{ is a jumping time })\nonumber \\
&= c_1 \E_x \left[\int_0^{t\wedge \tau_{B(x,  (18)^{-1}b \Phi^{-1}(t))}} \int_{B(y, \,  (12)^{-1}b \Phi^{-1}(t))}
j(|X_s- u|) duds \right]. \label{e:nv1}
\end{align}
Using the ({\bf UJS}) property of $j$ (see  \cite[page 1070]{CKK2}), we obtain
\begin{eqnarray}
 && \E_x \left[\int_0^{t\wedge \tau_{B(x,  (18)^{-1}b \Phi^{-1}(t))}} \int_{B(y, \,  (12)^{-1}b \Phi^{-1}(t))}
j(|X_s-u|) duds \right] \nn\\
 &= & \E_x \left[\int_0^{t} \int_{B(y, \,  (12)^{-1}b \Phi^{-1}(t))}
j(|X^{B(x,  (18)^{-1}b \Phi^{-1}(t))}_s-u|) duds \right] \nn\\
 &\ge & c_2\Phi^{-1}(t)^d \int_0^{t} \E_x \left[
j(|X^{B(x,  (18)^{-1}b \Phi^{-1}(t))}_s-y|) \right]ds \nn\\
& \ge & c_2\Phi^{-1}(t)^d \int_{t/2}^{t} \int_{B(x,  (72)^{-1}b \Phi^{-1}(t/2))}
j(|w-y|) p_{B(x,  (18)^{-1}b \Phi^{-1}(t))}(s, x,w)     dw ds.
   \label{e:nv2}
\end{eqnarray}
Since, for $t/2 <s <t$ and $w\in B(x,  (72)^{-1}b \Phi^{-1}(t/2))$,
 $$ \delta_{B(x,  (18)^{-1}b \Phi^{-1}(t))} (w)
 \ge (18)^{-1}b \Phi^{-1}(t)-  (72)^{-1}b \Phi^{-1}(t/2)\ge 2^{-1}(18)^{-1}b \Phi^{-1}(s)
 $$
and
 $$|x-w| <  (72)^{-1}b \Phi^{-1}(t/2) \le 4^{-1}  (18)^{-1}b \Phi^{-1}(s),   $$
we have by Proposition \ref{step1} that for $t/2 <s <t$ and $w\in B(x,  (72)^{-1}b \Phi^{-1} (t/2))$,
\begin{eqnarray}
p_{B(x,  (18)^{-1}b \Phi^{-1}(t))}(s, x,w)  \ge c_3\, (\Phi^{-1}(s))^{-d}  \ge c_3\, (\Phi^{-1}(t))^{-d}.
   \label{e:nv4}
\end{eqnarray}

Combining \eqref{e:nv1}, \eqref{e:nv2} with \eqref{e:nv4} and applying ({\bf UJS}) again, we get
\begin{align}
\P_x \left( X^D_t \in B \big( y, \,  6^{-1} b \Phi^{-1}(t) \big) \right)
 \ge & c_4 t  \int_{B(x,  (72)^{-1}b \Phi^{-1}(t/2))}
j(|w-y|)    dw
\nn\\
 \ge & c_5 t  (\Phi^{-1}(t/2))^{d}
j(|x-y|)   \ge  c_6 t  (\Phi^{-1}(t))^{d}
j(|x-y|).
   \label{e:nv6}
\end{align}
In the last inequality we have used Lemma \ref{lem:bf}(a).
Since by the semigroup property of $p_D$ and Proposition~\ref{step1},
\begin{eqnarray*}
p_D(t, x, y) &=& \int_{D} p_{D}(t/2, x, z)
p_{D}(t/2, z, y)dz\\
&\ge& \int_{B(y, \, b \Phi^{-1}(t/2)/ 6)}
p_{D}(t/2, x, z) p_{D}(t/2, z, y) dz\\
&\ge& c_7 (\Phi^{-1}(t/2))^{-d} \,
\bP_x \left( X^{D}_{t/2} \in B(y, 6^{-1}b
\Phi^{-1}(t/2)) \right), 
\end{eqnarray*}\
the proposition now follows 
 from this and  \eqref{e:nv6}. \qed

Recall that $B=(B_t:\, t\ge 0)$ is a Brownian motion in
${\mathbb R}^d$ and $S=(S_t:\, t\ge 0)$  a subordinator  independent of $B$.
Suppose that $U$ is an open subset of $\bR^d$.
We denote by $B^U$ the part process of $B$ killed upon leaving $U$.
The process $\{Z^U_t: t\ge 0\}$ defined by $Z^U_t=B^U_{{S}_t}$
is called a subordinate killed Brownian motion in $U$.
Let $q_U(t, x, y)$ be the transition density of $Z^U$.
Clearly,  $Z^U_t =B_{{S}_t}$ for every $t\in [0, \zeta)$ where $\zeta$ is the lifetime of $Z^U$. 
Therefore we have
\begin{equation}\label{e:pq}
p_{U}(t, z, w)\ge q_{U}(t, z, w) \quad \hbox{for } (t, z, w)\in (0,
\infty)\times U\times U.
\end{equation}

For a $C^{1,1}$ open set $D$ in $\bR^d$
with characteristics $(R_0, \Lambda)$, consider a $z \in \partial D$ and a $C^{1,1}$-function $\varphi=\varphi_z: \R^{d-1}\to \R$ satisfying $\varphi(0)=0$, $\nabla\varphi (0)=(0, \dots, 0)$, $\| \nabla\varphi \|_\infty \leq \Lambda$, $| \nabla \varphi(x)-\nabla \varphi(w)| \leq \Lambda |x-w|$ and an orthonormal coordinate system $CS_z$ of  $z=(z_1, \cdots, z_{d-1}, z_d):=(\wt z, \, z_d)$ with origin at $z$ such that $ D\cap B(z, R_0 )= \{y=({\tilde y}, y_d) \in B(0, R_0) \mbox{ in } CS_z: y_d > \varphi (\wt y) \}$. 
 Define 
 \begin{equation}\label{e:d_com}
\rho_z (x) := x_d -  \varphi (\wt x) \text{ and } D_z( r_1, r_2) :=\left\{ y\in D: r_1 >\rho_z(y) >0,\, |\wt y | < r_2
\right\}, \quad r_1, r_2>0, 
\end{equation}
where $(\wt x, x_d)$ are the coordinates of $x$ in $CS_z$. 
We also define
\begin{align}\label{e:kappa1}
\kappa=\kappa(\Lambda):=(1+(1+\Lambda)^2)^{-1/2}.
\end{align}
It is easy to see that for every $z \in\partial D$ and $r \le \kappa R_0$,
\begin{align}
\label{e:Uzr}
 D_z( r, r) \subset D\cap B(z,  r /\kappa).
\end{align}

It is well known (see, for instance \cite[Lemma
2.2]{So}) that there exists $L_0=L_0(R_0, \Lambda, d)>0$ such that for
every $z \in\partial D$ and $r \le \kappa R_0$, one can find a $C^{1,1}$
domain $V_z(r)$ with characteristics $(rR_0/L_0, \Lambda L_0/r)$
such that $D_z( 3r/2, r/2) \subset V_z(r)  \subset  D_z( 2r, r) $. 
In this paper, given a $C^{1, 1}$ open set $D$, 
$V_z(r)$ always refers to the $C^{1, 1}$ domain above.

\begin{prop}\label{prop:on-D-low}
Suppose that $\phi$ has no drift and  satisfies   $ L_a(\gamma, C_L)$ for some $a\ge 0$. 

  \noindent
  (a)
We assume that $D$ is a connected $C^{1,1}$ open set in $\bR^d$
with characteristics $(R_0, \Lambda)$ such that
the path distance of $D$
is comparable to the Euclidean distance with characteristic $\lambda_1$.
For any $T>0$, there exist positive constants
$c_1$ and $c_2$ depending on  $R_0, \Lambda, \lambda_1, T, \phi, \gamma, C_L, a, b$ such that 
 for every $(t, x, y)\in (0, T]\times D\times D$,
\begin{align}
\label{e:Glower}
p_D(t, x,y)\geq c_1  \left(1\wedge \frac{\delta_D(x)}{    \Phi^{-1} (t)}\right)
\left(1\wedge \frac{\delta_D(
y)}{  \Phi^{-1} (t)}\right)  \Phi^{-1}(t)^{-d}\exp \left( -\frac{c_{2}|x-y|^2} {\Phi^{-1}(t)^2} \right).
\end{align}

Moreover, 
 there exist $c_3, c_4>0$ such that for all $z \in \partial D$, $r \le \kappa R$ and $(t, x, y)\in (0,  \Phi(r)]\times {V_z(r)}\times {V_z(r)},$
\begin{align}
\label{e:Glowernew}
p_{V_z(r)}(t, x,y)\geq c_3  \left(1\wedge \frac{\delta_{V_z(r)}(x)}{    \Phi^{-1} (t)}\right)
\left(1\wedge \frac{\delta_{V_z(r)}(
y)}{  \Phi^{-1} (t)}\right)  \Phi^{-1}(t)^{-d}\exp \left( -\frac{c_{4}|x-y|^2} {\Phi^{-1}(t)^2} \right).
\end{align}

  \noindent
(b)  Furthermore, if $\phi$ satisfies   $ L_0(\gamma, C_L)$ and $D$ is a domain consisting of all the points above the graph of a 
bounded globally $C^{1,1}$ function, then \eqref{e:Glower} holds 
   for every $(t, x, y)\in (0, \infty)\times D\times D$.
\end{prop}
\proof 
(a) 
Le $\rho\in (0,1)$ be the constant in  Proposition \ref{prop:sub-low}. Without loss of generality we assume $T \ge \rho\phi(a)^{-1}$. 
Suppose that $x$ and $y$ are in $D$. Let $\widetilde{p}_D(t, z, w)$ be the transition density of
$B^D$.
By
\cite[Theorem 3.3]{Ch} (see also \cite[Theorem 1.2]{Zq3} where the comparability condition on the path distance in $D$ with the Euclidean distance is used),
there exist positive constants
$c_1=c_1(R_0, \Lambda, \lambda_0, T, \phi, \rho)$ and $c_2=c_2(R_0, \Lambda, \lambda_0)$
 such that for any $(s, z, w)\in (0,  \phi^{-1}(\rho T^{-1})^{-1}]
\times D\times
D$,
\bee \label{e:2.41}
\widetilde{p}_D(s, z, w)\ge c_1\left(1\wedge \frac{\delta_D(z)}{\sqrt{s}}
 \right)\left(1\wedge \frac{\delta_D(w)}{\sqrt{s}}
\right)s^{-d/2}e^{-c_2|z-w|^2/s}.
\eee

Recall that $q_D (t,x,y)$ is of the form
$$q_D (t,x,y) =\int_{(0,\infty)}\widetilde{p}_D(s, x, y)\P(S_t\in ds).$$
 Using this and \eqref{e:2.41} we get
\begin{align}
&p_D(t,x, y) \ge q_D (t,x,y) \nn\\
& \ge   \int_{\left[2^{-1}\phi^{-1}(t^{-1})^{-1},\phi^{-1}( \rho t^{-1})^{-1}\right]} 
\widetilde{p}_D(s, x, y)\P(S_t\in ds)\nonumber \\
    &\geq c_1
     \int_{\left[2^{-1}\phi^{-1}(t^{-1})^{-1},\phi^{-1}( \rho t^{-1})^{-1}\right]}\left(1\wedge \frac{\delta_D(x)}{\sqrt{s}}\right)
\left(1\wedge \frac{\delta_D(
y)}{\sqrt{s}}\right)s^{-d/2}e^{-c_2\frac{|x-y|^2}{s}}\P(S_t\in ds)\nonumber \\
  &\geq c_1 \left(1\wedge \frac{\delta_D(x)}{\sqrt{\phi^{-1}( \rho t^{-1})^{-1}}}\right)
\left(1\wedge \frac{\delta_D(
y)}{\sqrt{\phi^{-1}( \rho t^{-1})^{-1}}}\right)     \phi^{-1}(\rho t^{-1})^{d/2}e^{-2c_2 |x-y|^2\phi^{-1}(t^{-1})} \nn\\
& \qquad \times \P\left(2^{-1}\phi^{-1}(t^{-1})^{-1}\leq S_t\leq\phi^{-1}(\rho t^{-1})^{-1}\right).
\label{eq:tmp_20341}
\end{align}
Now, using 
\eqref{e:prg1} and Proposition \ref{prop:sub-low}, we conclude from \eqref{eq:tmp_20341} that 
\begin{align*}
  & p_D(t,x,y) \nn\\
 \geq   &
    c_3 \left(1\wedge \frac{\delta_D(x)}{\sqrt{\phi^{-1}(  t^{-1})^{-1}}}\right)
\left(1\wedge \frac{\delta_D(
y)}{\sqrt{\phi^{-1}(  t^{-1})^{-1}}}\right)  \phi^{-1}(t^{-1})^{d/2}  e^{-2c_2 |x-y|^{2}\phi^{-1}(t^{-1})}\, \nn\\
=  &c_3
    \left(1\wedge \frac{\delta_D(x)}{    \Phi^{-1} (t)}\right)
\left(1\wedge \frac{\delta_D(
y)}{  \Phi^{-1} (t)}\right)  \Phi^{-1}(t)^{-d}  \exp \left( -2c_{2}\frac{|x-y|^2} {\Phi^{-1}(t)^2} \right).
\end{align*}
We have proved \eqref{e:Glower}.

Using \cite[(4.4)]{KSV17}, we have that there exist $c_4, c_5>0$ such that   for any $s\in (0,  r^2]$ and any $z,w\in V_z(r)$,
\bee \label{e:2.41n}
\widetilde{p}_{V_z(r)}(s, z, w)\ge c_4\left(1\wedge \frac{\delta_{V_z(r)}(z)}{\sqrt{s}}
 \right)\left(1\wedge \frac{\delta_{V_z(r)}(w)}{\sqrt{s}}
\right)s^{-d/2}e^{-c_5|z-w|^2/s}.
\eee
Since  $t \le \Phi(r)$ if and only if $\phi^{-1}(t^{-1})^{-1}\le r^2$, we can repeat the proof of \eqref{e:Glower} and see that \eqref{e:Glowernew} holds true.

\noindent
(b) 
Suppose that $D$ is a domain consisting of all the points above the graph of a 
bounded globally $C^{1,1}$ function. Then  by \cite{So}, \eqref{e:2.41} holds for 
all $(s, z, w)\in (0, \infty)
\times D\times
D$. Using this fact and the assumption $ L_0(\gamma, C_L)$, one can follow the arguments in (a) line by line and prove (b). 
We skip the details.
\qed

\medskip

\section{Key estimates
}\label{s:KBHP}

In this section we prove key estimates on exit distribution for $X$ 
in $C^{1,1}$ open set with explicit decay rate.

Recall that $H(\lambda)=\phi(\lambda)-\lambda\phi'(\lambda)$, which is non-negative and non-decreasing on $(0, \infty)$.
We remark  here that $H$ loses the information on the drift of $\phi$. 

Throughout this section we assume that $H$ satisfies $ L_a(\gamma, C_L)$ and $ U_a(\delta, C_U)$ for some $a>0$ with $\delta<2$  and   the drift of the subordinator is zero. 
\begin{prop}\label{prop:off-lower2} 
For every $M>0$ there exists $c=c( a, M)>0$ such that for all $t >0$ and $x \in B(0, M)$  satisfying  $t\phi(|x|^{-2})\leq 1$ we have
 \[
  p(t,x)\geq c t|x|^{-d}H(|x|^{-2})\,.
  \]
  Thus, for all $x \in B(0, M)$,
  \begin{align}
  \label{e:jlow}
    j(|x|) \geq c |x|^{-d}H(|x|^{-2})\,.
  \end{align}

\end{prop}
\proof 
The proof is just a combination of  Proposition \ref{step3} and the proof of \cite[Proposition 3.6]{M}. We spell out the details for completeness. 
By \cite[Proposition 2.8]{M}  there exist  $L_1, L_2>1$ and $c_1>0$ such that for $|x|\le(aL_1)^{-1/2}$ and $t\phi(|x|^{-2})\leq 1$ it holds that 
\begin{align}
\label{e:p28}
    \P \left(|x|^2\leq S_t\leq L_2 |x|^2 \right)\geq c_1 tH(|x|^{-2})\,.
\end{align}
 Without loss of generality, we assume that $M > (aL_1)^{-1/2}$ and consider the following two cases separately.

\noindent
(1) $|x| \le (aL_1)^{-1/2}$ and $t\phi(|x|^{-2})\leq 1$: 
In this case, 
by \eqref{eq:heat_kernel} and  \eqref{e:p28} we obtain
\begin{align*}
 p(t,x)&\geq (4\pi)^{-d/2} \int_{[|x|^2,L_2|x|^2]}s^{-d/2}e^{-\frac{|x|^2}{4s}}\P(S_t\in ds)\\
 &\geq (4\pi)^{-d/2}  L_2^{-d/2}|x|^{-d}e^{-1/4}\P(|x|^2\leq S_t\leq L_2 |x|^2)\\
 &\geq c_1(4\pi)^{-d/2}  L_2^{-d/2} e^{-1/4}  t|x|^{-d}H(|x|^{-2})\,.
 \end{align*}
 
 \noindent
(2) $(aL_1)^{-1/2}  <  |x|   \le M$ and $t\phi(|x|^{-2})\leq 1$: 
In this case, $t \le \phi(|x|^{-2})^{-1} \le \phi(M^{-2})^{-1}$.
Thus by Proposition \ref{step3} we obtain
\begin{align*}
 p(t,x)&\geq c_2 t 
 j(|x|)\geq c_2 t 
 j(M)
  \geq c_3 t|x|^{-d}H(|x|^{-2}) \,.
 \end{align*}
\qed

We now revisit \cite{M} and improve the main result of \cite{M}
for the cases that $H$ satisfies the lower and upper scaling conditions near  infinity.

\begin{thm}\label{tm:mainb}
  For every $T, M>0$ there exists $c=c(a,   T, M)>0$ such that for all $t \le T$ and $x \in B(0, M)$,  
  $$
  p(t,x)\geq c\left( \phi^{-1}(t^{-1})^{d/2} \wedge \big( t|x|^{-d}H(|x|^{-2})+ \phi^{-1}(t^{-1})^{d/2}e^{-2^{-1}|x|^2\phi^{-1}(t^{-1})}\big)\right).
$$
\end{thm}

\pf This theorem follows from  Lemma 
\ref{lem:bf}(b),  Propositions \ref{prop:on-low} and \ref{prop:off-lower2}. 

\qed

Let $T_A:=\inf\{t>0: X_t \in A\}$, the first hitting time of $X$ to $A$.
Observe that for every Borel subset $A \subset U$ and $r>0$, we have
\begin{align}
\label{e:elow1}
&\P_x\(T_A< \tau_U \wedge \Phi(r) \) \ge \P_x \Big(\int_0^{ \Phi(r) }{\bf 1}_A(X^U_s)ds>0\Big)\ge \frac{1}{\Phi(r)} \int_0^{ \Phi(r)}\P_x \Big(\int_0^{ \Phi(r) }{\bf 1}_A(X^U_s)ds>u\Big)du
\nn\\
&
= \frac{1}{\Phi(r)}\E_x \int_0^{\Phi(r) }{\bf 1}_A(X^U_s)ds 
 \ge \frac{1}{\Phi(r)}\int_0^{\Phi(r) }\int_A p_U(s, x,y)dy ds.
\end{align}

Using  Levy system, \eqref{e:t:up1j} and \eqref{e:Euplo},  we have that for $w\in \R^d$ and  $0 <4r \le R <1$,
\begin{align}\label{e:ggggg1}
&\P_w \left(X_{\tau_{ B(w, r)}} \in  B( w, R)^c\right)
\le 
\E_w[\tau_{ B(w, r)}]  \sup_{y \in B(w, r)} \int_{B( w, R)^c} j(|y-z|) dz\nn\\
&\le
\frac{c_1}{\phi(r^{-2}) }\left( \int_{R}^1 H(s^{-2}) s^{-1} ds +c \right)
 \le c_{2}\frac{H(R^{-2})}{ \phi(r^{-2})   } \le  c_{2}\frac{\phi(R^{-2})}{ \phi(r^{-2})   } . 
\end{align}

Now we prove the following estimate, which is inspired by the proof of \cite[Lemma 5.3]{Gu}.
(See also \cite{KSV17, KSV18}.)
We recall that $\rho_z$, $D_z( r_1, r_2)$ and $\kappa$ are defined in \eqref{e:d_com} and \eqref{e:kappa1} respectively. 
\begin{prop}\label{p:H12}
Let $D\subset \R^d$ be a $C^{1,1}$ open set with characteristics
$(R_0, \Lambda)$. 
Assume that $H$ satisfies $ L_a(\gamma, C_L)$ and $ U_a(\delta, C_U)$ with $\delta<2$ and $\gamma >2^{-1}{\bf 1}_{\delta \ge 1}$
   for some $a \ge 0$.
Then there exists a constant  $c=c(\phi, R_0,  \Lambda)>0$ such that for every $r \le  \kappa^{-1} (R_0\wedge 1)/2$, $z \in D$ 
 and $x \in D_z(2^{-3}r, 2^{-4}r)$,
$$\P_x\big( X_{\tau_{ D_z (  r , r)}} \in
D\big)\leq c\, \P_x\big( X_{ \tau_{ D_z (  r , r)}}\in
D_z (  2 r , r)\big).$$ 
\end{prop}
\pf
Without loss of generality we assume $z=0$. 
Let 
$E_2:=\{X_{ \tau_{ D_0 (  r , r)}} \in
D  \}$ and 
$
E_1:=\{X_{ \tau_{ D_0 (  r , r)}} \in
D_0 (  2 r , r)\}. 
$
We claim that $\P_x( E_2)\leq c_{0} \P_x( E_1)$ 
for all $r \le  \kappa^{-1} (R_0\wedge 1)/2$ and $x \in D_0 ( 2^{-3} r , 2^{-4} r )$. 

When $\delta < 1$, we use \cite[Theorem 1.8]{GKK} and get the claim immediately. Thus, throughout the proof we assume that $\delta \ge 1$.

 Recall from  the paragraph before
Proposition  \ref{prop:on-D-low} that,  for $z \in\partial D$ and $r \le \kappa R_0$,  $V_0(r)$ is a $C^{1,1}$
domain with characteristics $(rR_0/L_0, \Lambda L_0/r)$ such that $D_0( 3r/2, r/2) \subset V_0(r)  \subset  D_0( 2r, r) $. 
Note that for $w \in D_0 ( 2^{-3} r , 2^{-4} r )$, we have $\delta_{V_0(r)}(w)=\delta_{D}(w)$.
Using this,  \eqref{e:Glowernew} and  \eqref{e:elow1}, we have that for  $w \in D_0 ( 2^{-3} r , 2^{-4} r )$,
\begin{align}\label{e:D2_43n2}
&\P_w ( E_1 )  \ge \P_w \big(  \tau_{ V_0(r)} >T_{D_0 (  5r/4, r/4)\setminus D_0 (  r , r/4))}  \big)\nn\\
  & \ge \frac{1}{\Phi(r)}\int_{\Phi(r)/2 }^{\Phi(r) }\int_{ D_0 (  5r/4, r/4)\setminus D_0 (  r , r/4))} p_{V_0(r)}(s, w,y)dy ds \nn\\
    & \ge \frac{c_1}{\Phi(r)}\int_{\Phi(r)/2 }^{\Phi(r) }\int_{  D_0 (  5r/4, r/4)\setminus D_0 (  r , r/4))}
     \left(1\wedge \frac{\delta_{V_0(r)}(w)}{    \Phi^{-1} (s)}\right)
 \Phi^{-1}(s)^{-d}
    dy ds \nn\\
      & \ge c_2 \frac{\delta_D(w)r^d}{\Phi(r)}\int_{\Phi(r)/2 }^{\Phi(r) }
 \frac{ds}{    \Phi^{-1} (s)^{d+1}} \ge c_3 \frac{\delta_D(w)}{r}.
\end{align}

We define, for $i\ge 1$,
$$
J_i=D_0(2^{-i-2}r,s_i)\setminus  D_0(2^{-i-3}r,s_i),\ \ \ \ s_i=\frac{1}{4}\left(\frac{1}{2}-\frac{1}{50}\sum_{j=1}^i\frac{1}{j^2}\right)r,
$$
and $s_0=s_1$.  
Note that $r/(10)<s_i <r/8$.
For $i\geq 1$,  set
\begin{align}\label{q1q}
d_i=d_i(r)=\sup_{z\in J_i}\P_z( E_2)/\P_z( E_1),\ \ \
\widetilde{J}_i=D_0(2^{-i-2}r,s_{i-1}),\ \ \ \
\tau_i=\tau_{\widetilde{J}_i}.
\end{align}
Repeating the argument leading to \cite[(6.29)]{KSV17}, we get that 
 for $z\in J_i$ and $i\geq 2$,
\begin{align}\label{q12q}  \P_z( E_2) 
\leq    \left ( \sup_{1\leq k\leq i-1}  d_k \right) \P_{z}(E_1) +\P_z\left( X_{\tau_i}\in
D \setminus \cup_{k=1}^{i-1}J_k \right).
\end{align}

For  $i\ge 2$, define $\sigma_{i, 0}=0, \sigma_{i, 1}=\inf\{t>0: |X_t-X_0|\geq 2^{-i-2}r\} $ and
$\sigma_{i, m+1}=\sigma_{i,1}\circ\theta_{\sigma_{i, m}}$
for $m\geq 1$. 

We first claim that  for all $w \in\widetilde{J}_i$, 
$ \P_{w}(X_{\sigma_{i, 1}}\notin \widetilde{J}_i)$ is bounded below by a strictly positive constant. 
We prove the claim for $w \in
\widetilde{J}_i \setminus D_0(2^{-i-3}r,s_{i-1})=
\{ y\in D: 2^{-i-2}r>\rho_0(y)  \ge 2^{-i-3}r,\, |\wt y | < s_{i-1}
\}.
$
Since  $
\widetilde{J}_i= \{ y=(\wt y, t): 0<t- \wh \varphi(\wt y) <2^{-i-2}r, \quad\, |\wt y | < s_{i-1}
\}$ with $t:=2^{-i-2}r-y_d$ and  $\wh \varphi(\wt y):=-\varphi(\wt y) $, the proof for the case $w \in D_0(2^{-i-3}r,s_{i-1})$
is same.

We choose $\eps \in (0, 2^{-4}/\Lambda)$ small so that
\begin{align}
\label{e:Fn1}
(2\eps^2+1) \( \frac{3+\eps \Lambda}{1-\Lambda\eps} \)^2+ 2 \eps^2  <16.
\end{align}
Fix $w \in
\widetilde{J}_i \setminus D_0(2^{-i-3}r,s_{i-1})$ and 
 define $A:=B((\wt w, w_d+2^{-i-1}r),  \eps 2^{-i-4}r)$ and 
 $$
 V:=B((\wt w, w_d-2^{-i-4}r),  3 \cdot 2^{-i-2} r) \cap \{ y_d > w_d-2^{-i-4}r, |\wt y-\wt w| < \eps (y_d - w_d+2^{-i-4}r)\}.
$$
For $y \in A$, we have 
$y_d-w_d  \ge 2^{-i-1}r -| y_d-w_d-2^{-i-1}r| > 2^{-i-1}r -\eps 2^{-i-4} r >2^{-i-2}r.
$ Thus, for $y \in A$ we have $y \notin B(w, 2^{-i-2}r)$, 
$|\wt w-\wt y| \le \eps 2^{-i-4}r <\eps (y_d - w_d+2^{-i-4}r)$ and 
$$
\rho_0(y) \ge y_d -w_d +\rho_0(w)-|  \varphi (\wt w)- \varphi (\wt y)| > (2^{-i-2}+(1-\eps \Lambda)2^{-i-4})r >2^{-i-2}r.
$$
Therefore \begin{align}
\label{e:Fn2}
A \subset V \setminus (\widetilde{J}_i  \cup B(w, 2^{-i-2}r)).
\end{align}
If  $y \in V \cap \widetilde{J}_i$ and $y_d <w_d$, then clearly $|y_d-w_d|=w_d-y_d \le 2^{-i-4}r$.
If  $y \in V \cap \widetilde{J}_i$ and $y_d  \ge w_d$, then
$y_d-w_d=\rho_0(y)-
\rho_0(w)+|  \varphi (\wt w)- \varphi (\wt y)|<3 \cdot 2^{-i-4}r +\Lambda \eps |y_d-w_d| +\Lambda \eps 2^{-i-4}r$
so that $ |y_d-w_d| <2^{-i-4}r (3+\Lambda \eps)/(1-\Lambda \eps)$.
Thus using \eqref{e:Fn1}, we have that  for $y \in V \cap \widetilde{J}_i$, 
\begin{align*}
&|y-w|^2 \le \eps^2(|y_d-w_d|+ 2^{-i-4}r)^2 + |y_d-w_d|^2
\le (2\eps^2+1)|y_d-w_d|^2+ 2\eps^2(2^{-i-4}r)^2
\\
&\le \(2\eps^2+1) \Big( \frac{3+\eps \Lambda}{1-\Lambda\eps} \Big)^2+ 2 \eps^2\)(2^{-i-4}r)^2 < (2^{-i-2}r)^2,
\end{align*}
which implies that 
\begin{align}
\label{e:Fn3}
V \cap \widetilde{J}_i \subset B(w, 2^{-i-2}r)
\end{align}
On the other hand, for 
$y \in \frac12A:=B((\wt w, w_d+2^{-i-1}r),  \eps 2^{-i-5}r)$, we have $\delta_V (w) \wedge \delta_V(y) \ge c_0 2^{-i-1}r 
$ and $|w-y| \le 2^{-i}r$. Since we assume that $\gamma >1/2$, we can find a large $M$ so that
$$
\frac{\Phi^{-1}(2s)}{\Phi^{-1}(s/M) } \le c_4(2M)^{1/(2\gamma)} <\frac{Mc_0}{48} \,\text{ and } \,\frac{ \Phi(s) }{\Phi(c_0 s)}  \le M\quad \text{ for all }s \in (0,1).$$
Thus, when $\Phi(2^{-i-2}r)/2 \le s \le \Phi(2^{-i-2}r)$ and $|z_1-z_2| \le 3 \cdot 2^{-i}r/M$ with $\delta_V (z_i) \ge c_0 2^{-i-2}r$, we see that 
 $|z_1-z_2| \le 12  \cdot 2^{-i-2}r/M \le 12  \Phi^{-1}(2s)/M \le c_0 \Phi^{-1}(s/M)/4$ and $\delta_V (z_i) \ge c_0 2^{-i-2}r \ge \Phi^{-1}(s/M)$
  (because  $ M\ge \Phi(2^{-i-2}r) /\Phi(c_0 2^{-i-2}r)\ge s/\Phi(c_0 2^{-i-2}r)$). Thus, by 
 Proposition \ref{step1}, for such $y,z$ and $s$, 
using this and a chaining argument through the semigroup property, we have
\begin{align}
\label{e:Fn4}
p_V(s, w,y) \ge    c_6 (2^{-i}r)^{-d}, \quad \text{for }
\Phi(2^{-i-2}r)/2 \le s \le \Phi(2^{-i-2}r) \text{ and }
y \in \frac12 A.
\end{align}
By \eqref{e:elow1} and \eqref{e:Fn1}--\eqref{e:Fn4}, we have that for all $w\in \widetilde{J}_i \setminus D_0(2^{-i-3}r,s_{i-1})$,
\begin{align*}
&\P_{w}\big(X_{\sigma_{i, 1}}\notin \widetilde{J}_i\big) \ge 
\P_w(T_{ \frac12 A}< \tau_{V} \wedge \Phi(2^{-i-2}r) )\nn\\
&\ge \frac{1}{\Phi(2^{-i-2}r)}\int_{\Phi(2^{-i-2}r)/2 }^{\Phi(2^{-i-2}r) }\int_{ \frac12 A} p_{V}(s, w,y)dy ds\ge  c_6 \frac{|\frac12 A|}{\Phi(2^{-i-2}r)}\int_{\Phi(2^{-i-2}r)/2 }^{\Phi(2^{-i-2}r) } (2^{-i}r)^{-d} ds, \end{align*}
which is  a positive constant independent of $i$. We have proved the claim. 
 
Thus, we have that there exists $k_1 \in (0,1)$ such that 
\begin{align}\label{e:rs}
\P_{w}\big(X_{\sigma_{i, 1}}\in \widetilde{J}_i\big)
=1-\P_{w}\big(X_{\sigma_{i, 1}}\notin \widetilde{J}_i\big)
  <k_1,\ \ \ w\in  \widetilde{J}_i.
\end{align}
For the purpose 
of  further estimates, we now choose a 
positive integer $l\ge 1$ 
such that $k_1^l\le 4^{-1}$.
Next we choose $i_0 \ge 2$ large enough so that 
$2^{-i}<1/(200 l i^3)$ for all $i\ge i_0$.
Now we assume $i\ge i_0$. Using \eqref{e:rs} and the strong Markov property we have  that  for   $z\in J_i$,
\begin{align}\label{q132q}    
&\P_z( \tau_{i}>\sigma_{i, li})\leq 
\P_z\big(X_{\sigma_{i, k}}\in \widetilde{J}_i, 1\leq k\leq
li\big )\nonumber\\
&=
 \E_z \left[ \P_{X_{\sigma_{i, li-1}}} (X_{\sigma_{i, 1}}\in  \widetilde{J}_i) : X_{\sigma_{i, li-1}} \in  \widetilde{J}_i,  X_{\sigma_{i, k}}\in \widetilde{J}_i, 1\leq k\leq
li-2
  \right]
 \nonumber\\      
& \leq \P_z\big(X_{\sigma_{i, k}}\in \widetilde{J}_i, 1\leq k\leq
li-1 \big )k_1\leq k_1^{li}.
\end{align}
Note that if $z\in J_i$ and $y\in D \setminus[ \widetilde{J}_i \cup(\cup_{k=1}^{i-1}J_k)]$, 
then $|y-z|\ge (s_{i-1}-s_i) \wedge (2^{-3}-2^{-i-2}) r = r/(200 i^2)$. 
Furthermore, since  $2^{-i-2} r< r/(200 i^2)$, $\tau_i$ must be one of  the
$\sigma_{i, k}$'s, $k\le li$. 
Hence, on
$\{X_{\tau_{i}}\in D \setminus \cup_{k=1}^{i-1}J_k,\ \ \tau_{i}\leq
\sigma_{i, li}\}$ with $X_0=z\in J_i$,  there exists $k$, $1\le k\le li$, such that 
$|X_{\sigma_{i, k}}-X_0|=|X_{\tau_i}-X_0|> r/(200 i^2)$.
Thus  for some $1\leq k\leq li$,
$$
\sum_{j=1}^k\big|X_{\sigma_{i, j}}-X_{\sigma_{i, j-1}}\big|>
\frac{r}{200i^2}\, .
$$
which implies for some $1\leq k'\leq k\le  li$, 
$$ 
\big|X_{\sigma_{i, k'}}-X_{\sigma_{i, k'-1}}\big|\geq
\frac1k \frac{r}{200i^2}\ge \frac{1}{ li} \frac{r}{200i^2}.
$$
Thus, 
  using  the strong Markov property  and then using \eqref{e:ggggg1} (noting that  
$4 \cdot 2^{-i-2}<1/(200 l i^3)$ for all $i\ge i_0$)
 we have 
 \begin{align}\label{dsa}& \P_z \left(X_{\tau_{i}}\in D \setminus
\cup_{k=1}^{i-1}J_k,\ \ \tau_{i}\leq \sigma_{i, li} \right)\nonumber\\
 \leq  &\sum_{k=1}^{li} \P_z\left(|
X_{\sigma_{i, k}}- X_{\sigma_{i, k-1}}|\geq r/(200li^3),
X_{\sigma_{i, k-1}}\in \widetilde{J}_{i}
\right ) \nonumber\\
 \leq  &li \sup_{z\in \widetilde{J}_{i}} \P_z\left(|
X_{\sigma_{i, 1}} - z |\geq r/(200li^3) \right)
  \leq  c_{7}li  \frac{\phi((200 li^3)^2/r^2)}{ \phi(2^{2(i+2)} r^{-2})   } .
\end{align}
Since
\begin{align*}
 \frac{\phi((200 li^3)^2/r^2)}{ \phi(2^{2(i+2)} r^{-2})   } \ge c_{8} 
  \frac{(200 li^3)^{2}}{ (2^{2(i+2)} ) ^{2}} 
\ge  c_{9} {i^{6}}(4^{})^{-i},
\end{align*} 
by \eqref{e:D2_43n2}, (\ref{q132q}), (\ref{dsa}) and Lemma \ref{lem:bf}(b),  for $z\in J_i$, $i\ge i_0$, we have
\begin{align*}
&\frac{\P_z( X_{\tau_i}\in
D \setminus \cup_{k=1}^{i-1}J_k)}{\P_z(E_1)}\leq
 \frac{1}{\P_z(E_1)} \left(k_1^{li}+c_{24}li  \frac{\phi((200 li^3)^2/r^2)}{ \phi(2^{2(i+2)} r^{-2})   }  \right)\nn\\
& \le \frac{c_{10}i}{\P_z(E_1)} \frac{\phi((200 li^3)^2/r^2)}{ \phi(2^{2(i+2)} r^{-2})   } 
\le c_{11}i2^{i} \frac{\phi((200 li^3)^2/r^2)}{ \phi(2^{2(i+2)} r^{-2})   } 
\le c_{12}i2^{i} {i^{6\gamma}}(2^{\gamma})^{-2i}
\le c_{13}{i^{13}}2^{-(2\gamma-1)i}.\end{align*} 
By this and  (\ref{q12q}), for $z\in J_i$ 
and $i\geq i_0$,
\begin{align*}  
\frac{\P_z( E_2)}{\P_z(E_1)} \leq  \sup_{1\leq k\leq i-1}  d_k   +\frac{\P_z( X_{\tau_i}\in
D \setminus \cup_{k=1}^{i-1}J_k)}{\P_z(E_1)} 
\leq \sup_{1\leq k\leq i-1}d_k+ c_{13}{i^{13}}2^{-(2\gamma-1)i}.
\end{align*}
This implies that
\begin{align}  
 \sup_{ r \le \kappa^{-1} (R_0\wedge 1)/2} d_i(r)& \leq \sup_{1\leq k\leq i_0-1\atop r \le \kappa^{-1} (R_0\wedge 1)/2} d_k(r)+   c_{14} \sum_{k=i_0}^{\infty} {k^{13}}2^{-(2\gamma-1)k}=:c_{15} <\infty. \nonumber
\end{align}
Thus the claim above is valid, since $D_0 ( 2^{-3} r , 2^{-4} r )\subset \cup_{k=1}^\infty J_k$. The proof is now complete.
\qed

The next two results should be well-known but we could not find any reference. We provide the full details.
\begin{lemma}
\label{l:r-time1}
For any non-negative locally integrable  function $t \to k(t)$ on $(0, \infty)$ and every $R>0$, $s \in (0, R/2)$ and $\eps \in (0, s/2)$, 
\begin{align}\label{e:r-di22}
\left(\int_{s+\eps}^{R+s} + \int_{-R+s}^{s-\eps} \right)((t_+)^2-s^2)k(|t-s|)dt
=\int_{\eps}^R ({\bf 1}_{u<s}2u^2 +{\bf 1}_{u \ge s}(u^2+s(2u-s))) k(|u|)du.
\end{align}
Thus,
$$
P.V. \int_{-R+s}^{R+s} ((t_+)^2-s^2)k(|t-s|)dt= \int_{0}^R ({\bf 1}_{u<s}2u^2 +{\bf 1}_{u \ge s}(u^2+s(2u-s))) k(|u|)du.
$$
\end{lemma}
\pf
Using the change of variables $u=t-s$ in the first integral and $u=s-t$ in the second integral, we get that for $\eps \in (0, s/2)$,
\begin{align*}
&
\left(\int_{s+\eps}^{R+s} + \int_{-R+s}^{s-\eps} \right)((t_+)^2-s^2)k(|t-s|)dt\\
=&\int_{\eps}^R((s+u)^2-s^2) k(|u|)du + \int_{\eps}^R([(s-u)_+]^2-s^2) k(|u|)du \nn \\
=&\int_{\eps}^R((s+u)^2+[(s-u)_+]^2-2s^2) k(|u|)du\nn\\
=&\int_{\eps}^s((s+u)^2+(s-u)^2-2s^2) k(|u|)du+\int_s^R((s+u)^2-2s^2) k(|u|)du\nn\\
=&\int_{\eps}^s2u^2 k(|u|)du+\int_s^R(u^2+s(2u-s)) k(|u|)du.
\end{align*}
Letting $\eps \to 0$, we also have proved the second claim of the lemma.
\qed

\begin{lemma}
\label{l:r-time2}
For every $R>0$ and $x=(\wt 0, x_d) \in \R^d$ with $x_d>0$, 
\begin{align}\label{e:r-di3}
&\frac{1}{2d} \int_{B(0,R)}|z|^2j(|z|)dz\nn\\
 \le& P.V. \int_{\{(\wt w, w_d) \in \R^d: |\wt w|<R, |w_d-x_d|<R\}}
([(w_d)_+]^2-x_d^2) j(|w-x|)dw\\
 \le& \frac{1}{d} \int_{B(0,\sqrt 2 R)}|z|^2j(|z|)dz < \infty. \nn
\end{align}
\end{lemma}
\pf
By Lemma \ref{l:r-time1}, for all small $\eps \in (0, x_d/2)$, 
\begin{align*}
& \int_{\{(\wt w, w_d) \in \R^d: |\wt w|<R, |w_d-x_d|<R, 
 |\wt w|^2+ |w_d-x_d|^2>\eps^2
 \}}
([(w_d)_+]^2-x_d^2) j(|w-x|)dw \\
&
=\int_{\{|\wt w|<R\}} \int_{\{ \sqrt{(\eps^2- |\wt w|^2)_+}<|w_d-x_d|<R\}} 
([(w_d)_+]^2-x_d^2) j((|w_d-x_d|^2+|\wt w|^2)^{1/2})  dw_d d\wt w \\
&= \int_{\{|\wt w|<R\}}  \int_{\sqrt{(\eps^2- |\wt w|^2)_+}}^{R} ({\bf 1}_{u<x_d}2u^2 +{\bf 1}_{u \ge x_d}(|u|^2+x_d(2u-x_d))) j((|u|^2+|\wt w|^2)^{1/2})  du d\wt w. 
\end{align*}
Thus by the monotone convergence theorem, \eqref{e:r-di3} is equal to 
\begin{align*}
&\frac{1}{2} \int_{\{|\wt w|<R\}}  \int_{-R}^{R} ({\bf 1}_{|u|<x_d}2u^2 +{\bf 1}_{|u| \ge x_d}(u^2+x_d(2u-x_d))) j((|u|^2+|\wt w|^2)^{1/2})  du d\wt w \\
&\ge
\frac{1}{2} \int_{B(0, R)} |u|^2 j((|u|^2+|\wt w|^2)^{1/2})  du d\wt w=\frac{1}{2d} \int_{B(0,R)}|z|^2j(|z|)dz.
\end{align*}
Since $x_d(2u-x_d) \le u^2$, wee also have the upper bound as
\begin{align*}
&\frac{1}{2} \int_{\{|\wt w|<R\}}  \int_{-R}^{R} ({\bf 1}_{|u|<x_d}2u^2 +{\bf 1}_{|u| \ge x_d}(u^2+x_d(2u-x_d))) j((|u|^2+|\wt w|^2)^{1/2})  du d\wt w \\
&\le
 \int_{B(0, \sqrt 2R)} |u|^2 j((|u|^2+|\wt w|^2)^{1/2})  du d\wt w=\frac{1}{d} \int_{B(0,\sqrt 2 R)}|z|^2j(|z|)dz.
\end{align*}
\qed

Let $\psi(r)=1/H(r^{-2})$. We first note that $\Phi (r) \le \psi(r)$ and 
\begin{equation}\label{eqn:polypsi}
c_1 \Big(\frac Rr\Big)^{2\gamma} \,\leq\,
 \frac{\psi (R)}{\psi (r)}  \ \leq \ c_2
\Big(\frac Rr\Big)^{2\delta}
\qquad \hbox{for every } 0<r<R<1.
\end{equation}

Since
$$
\int_0^r\frac{s}{\psi(s)}ds=\int_0^rs H(s^{-2})ds=\frac12\int_{r^{-2}}^\infty  \frac{H(t)}{t^2}dt=
-\frac12\int_{r^{-2}}^\infty  (\frac{\phi(t)}{t})'dt=\frac{r^{2}}2 \phi(r^{-2})=\frac{r^{2}}{2\Phi(r)},
$$
$\Phi$ and $\psi$ are also related as 
\begin{align}
\label{e:r-di1}
\Phi(r)=\frac{r^2}{2\int_0^r\frac{s}{\psi(s)}ds}.
\end{align}

Using \eqref{e:jlow}, \eqref{eqn:polypsi} and \eqref{e:r-di1}, we get that for $R<1$,
\begin{align}
\label{e:r-di5}
& \int_{B(0,R)}|z|^2j(|z|)dz \ge c_1^{-1} c_2(d) \int_0^R \frac{r}{\psi(r)}dr= \frac{ c_2(d)}{2c_1} \frac{R^2}{\Phi(R)},
\end{align}
and
\begin{align}
\label{e:r-di6}
& \int_{B(0,R)^c}j(|z|)dz \le c_2(d) (c_1\int_R^1 \frac{dr}{r\psi(r)}
+\int_1^\infty j(r)dr)=c_2(d) (\frac{c_1}{ \psi (R)}\int_R^1  \frac{\psi(R)}{r\psi(r)}dr
+c_3)\nn\\
&\le \frac{c_2(d)c_1c_4(1-R^{2\delta})+c_3}{ \psi (R)} \le  \frac{c_2(d)c_1c_4+c_3}{ \psi (R)} \le  \frac{c_2(d)c_1c_4+c_3}{ \Phi (R)}
\end{align}
 Choose 
\begin{align}\label{e:r-diM}
M_0:=4[c_1d(c_2(d)c_1c_4+c_3)/c_2(d)]^{1/2}>4.
\end{align}
By \eqref{e:r-di5} and \eqref{e:r-di6},  if $r \le R/M_0$ then
\begin{align}
\label{e:r-di7}
r^2\int_{B(0,R)^c}j(|z|)dz   \le R^2  \frac{c_2(d)c_1c_4+c_3}{ M_0^2\Phi (R)}
 \le \frac{ c_2(d)}{8dc_1} \frac{R^2}{\Phi(R)}
\le \frac{1}{4d}  \int_{B(0,R)}|z|^2j(|z|)dz.
\end{align}
We use this constant $M_0$ in Lemma \ref{l:r-time3}, Proposition \ref{p:boundn1} and
Theorem \ref{t:n_estimates} below.

For any function $f:\R^d\to \R$ and $x\in \R^d$, we define an operator as follows:
\begin{align*}
\gener f(x)&:=P.V.\int_{\R^d} (f(y)-f(x))j(|x-y|)dy,\\
{\cal D} (\gener)&:=\left\{f\in C^{2}(\R^d):P.V.\int_{\R^d} (f(y)-f(x))j(|x-y|)dy\,\mbox{ exists and is finite.}\right\}.
\end{align*}
Recall that $C_0^2(\R^d)$ is the collection of $C^2$ functions in $\R^d$ vanishing at infinity. It is well known that $C_0^2(\R^d)\subset {\cal D} (\gener)$ and that, by the rotational symmetry of $X$,
\begin{align}
\label{e:AL}
A |_{C_0^2(\R^d)}=\gener |_{C_0^2(\R^d)}
\end{align}
 where  $A$ is the infinitesimal generator of $X$. We also recall that $\delta_D (x)$
is the distance of the point $x$ to $D^c$.
\begin{lemma}
\label{l:r-time3}
Suppose that   $D$ is a $C^{1,1}$ open set in $\R^d$ with characteristics $(R_0, \Lambda)$.
For any $z\in \partial D$ and $r\le (1\wedge R_0)/4$, we define 
$$f(y)=f_{r,z}(y):=(\delta_D(y))^2{\bf 1}_{D\cap B(z, 2r)} (y).$$
Then there exist $c=c(\phi,  \Lambda, d)>1$ and 
$\wt R=\wt R(\phi,  \Lambda, d)\in (0, (1\wedge R_0)/4)$ independent of $z$ such that 
for all $r \le \wt R$,  
$\gener f$ is well-defined in $D\cap B(z, r/M_0)$ and 
\begin{align}\label{generd}
 c \frac{r^2}{\Phi(r)}  \ge \gener f(x)\ge c^{-1} \frac{r^2}{\Phi(r)} \,\quad\mbox{ for all } x\in D\cap B(z, r/M_0). 
\end{align}
\end{lemma}
\pf
Since the case of $d=1$ is easier, we give the proof only for $d\ge 2$. 
Without loss of generality we assume that $\Lambda>1$.
For $x\in D\cap B(z, r/M_0)$, choose  $z_x\in \partial D$ be a point satisfying $\delta_D(x)=|x-z_x|$.
Let $\varphi$ be a  $C^{1,1}$ function and $CS=CS_{z_x}$ be an orthonormal coordinate system with $z_x$ chosen as the origin so that $\varphi(\wt 0)=0$, $ \nabla\varphi(\wt 0)=(0, \dots, 0)$, $\| \nabla \varphi \|_\infty \leq \Lambda$, $|\nabla \varphi (\wt y)-\nabla \varphi (\wt z)| \leq \Lambda |\wt y-\wt z|$, and $x=(\wt{0}, x_d)$, $ D\cap B(z_x, R_0)=\{y=(\wt{y}, y_d)  \in B(0, R_0) \mbox{ in } CS: y_d>\varphi(\wt{y})\}$.
We fix the function $\varphi$ and the coordinate system $CS$, and  consider the truncated square function $[(y_d)_+]^2$ in $CS$. 
Let
$$\bB_x=\bB_x(r):={\{(\wt w, w_d) \mbox{ in } CS:  |\wt w|<r, |w_d-x_d|<r\}} \subset B(z,2r),$$
and we define $\wh \varphi :B(\wt{0}, r)\to \R$ by $\wh \varphi(\wt{y}):=2\Lambda|\wt{y}|^{2}$.
Since $\nabla\varphi (\wt{0})=0$, by the mean value theorem we have 
$-\wh \varphi(\wt{y})\le\varphi(\wt{y})\le \wh \varphi(\wt{y})$  for any $y\in D\cap B(x, r/2)$
and so that
\begin{align*}
\{z=(\widetilde{z}, z_d)\in \bB_x : z_d\ge& \wh{\varphi}(\widetilde z)\}\subset D \cap  \bB_x \subset \{z=(\widetilde{z}, z_d)\in \bB_x : z_d\ge -\wh{\varphi}(\widetilde z)\}.
\end{align*}
Let $A:=\{y\in \bB_x:  -\wh{ \varphi}(\widetilde{y})\le y_d\le \wh{\varphi}(\widetilde{y})\}$
and 
$E:=\{y\in \bB_x: y_d>  \wh{\varphi}(\widetilde{y})\}\subset D$ so that 
\begin{align}
\label{e:r-di10}
& \int_{\bB_x(r)} \left|[(y_d)_+]^2-(\delta_D(y))_+^2 \right| j(|y-x|)dy\nn\\
\le& \int_A (y_d^2 + \delta_D(y)^2)j(|y-x|)dy +\int_{E} |y_d^2-\delta_D(y)^2|j(|y-x|)dy\nn\\
\le& 2^5\Lambda^2\int_A |\wt{y}|^{4} j(|\wt{y}|)dy +c_0r\int_{E}  |y_d -\delta_D(y)|j(|y-x|)dy
\end{align}
where we have used 
$y_d^2 + \delta_D(y)^2\le 2(2\wh{\varphi} (\widetilde{y}))^2=
2(4\Lambda|\wt{y}|)^2$ for $y \in A$.
We will show that the above is less than $c_1 r^3/\Phi(r)$.

First,  let $m_{d-1}(dy)$ be the Lebesgue  measure on $\R^{d-1}$.
Since $m_{d-1}(\{y:|\widetilde{y}|=s, - \wh{\varphi}(\widetilde{y})\le y_d\le \wh{\varphi}(\widetilde{y})\}) \le c_2 s^d$ for $0<s<r$,
using  polar coordinates for $|\wt{y}|=s$, by \eqref{e:r-di1} and \eqref{e:t:up1j}
\begin{align}
\label{e:r-di11}
\int_A |\wt{y}|^{4} j(|\wt{y}|)dy \le  r^3\int_A |\wt{y}| j(|\wt{y}|)dy \le c_3 r^3\int_0^r  \frac{{s}}{\psi(s)}ds  =\frac{c_3 r^5}{2\Phi(r)}.
\end{align}
Second, when $y\in E$, we have that $|y_d-\delta_D(y)|\le (1\wedge R_0)^{-1}\wh{\varphi}(\widetilde{y})$.
Indeed,  if $0<y_d \le \delta_D(y)$ and $y\in E$,  $ \delta_{D}(y)\le y_d+|\varphi(\widetilde{y})|\le y_d+\wh{\varphi}(\widetilde{y})$. 
Since we assume that $\Lambda>1$, we have $|\wt y|^2+(R_0-y_d)^2 <|\wt y|^2+(R_0-2\Lambda |\wt y|^2)^2<R^2$. Thus,  if $y_d \ge \delta_D(y)$ and $y\in E$, using the interior ball condition, 
we have 
 \begin{align*}
&y_d -\delta_D(y) \le y_d-R_0+\sqrt{ |\wt y|^2+(R_0-y_d)^2} \\
&= \frac{|\wt y|^2}
{\sqrt{ |\wt y|^2+(R_0-y_d)^2} + (R_0-y_d)}\,\le\, \frac{ |\wt y|^2} {2 (R_0-y_d)}
\le
\frac{ |\wt y|^2}{1\wedge R_0} \le\frac{ \wh{\varphi}(\widetilde{y}) }{1\wedge R_0} .
\end{align*}
Thus,
\begin{align}
\label{e:r-di111}
\int_{E} |y_d -\delta_D(y)|j(|y-x|)dy \le 
2 \Lambda (1\wedge R_0)^{-1} \int_{E}  |\wt{y}|^2  
   j((|y_d-x_d|+|\wt y|)/2)dy_dd\wt y.
\end{align}
Since $E \subset \{ (\wt{y}, y_d): |\widetilde{y}|<r, \, \wh{\varphi}(\widetilde{y})  <y_d< \wh{\varphi}(\widetilde{y})+2r\},$
using the polar coordinates for $|\wt{y}|=v$ and  the change of the variable $s:= y_d-\wh{\varphi}(v)$, we have  by  \eqref{e:t:up1j} and Lemma \ref{lem:bf},
\begin{align}
\label{e:r-di12}
\int_{E}  |\wt{y}|^2  
   j((|y_d-x_d|+|\wt y|)/2)dy_dd\wt y
\le c_4\int_{0}^{r} \int_{0}^{2r} 
\frac{dsdv}{\psi(v+ |s+\wh{\varphi}(v)-x_d|)}.\end{align}
Using 
\cite[Lemma 4.4]{KSV14}
with 
non-increasing functions $f(s)\equiv 1$ and $g(s):= \psi(s)^{-1}$ and 
$x(r)=x_d-\wh{\varphi}(r)$
 and get 
\begin{align}
\label{e:r-di13}
\int_{0}^{r} \int_{0}^{2r}
\frac{dsdv}{\psi(v+ |s+\wh{\varphi}(v)-x_d|)} \le  2
\int_0^{3r} (\int_0^u ds) \frac{du}{\psi(u)} \le \int_0^{3r} u\frac{du}{\psi(u)}.
\end{align}
Applying \eqref{e:r-di11}--\eqref{e:r-di13} to \eqref{e:r-di10} and using \eqref{e:r-di1}, 
we have that
\begin{align}
\label{e:r-di14}
 \int_{\bB_x(r) }|[(y_d)_+]^2-(\delta_D(y))_+^2|j(|y-x|)dy
\le c_5 r \int_0^{3r} u\frac{du}{\psi(u)} \le  c_6 r \frac{r^2}{\Phi(r)}.
\end{align}

On the other hand, since $x_d =\delta_D(x) \le r/M_0$, we see that $ \gener f(x)$ is well-defined and 
\begin{align}
\label{e:r-di8}
\int_{\bB_x(r)^c}(f(y)-x_d^2)j(|y-x|)dy \ge -x_d^2 \int_{\bB_x(r)^c}j(|y-x|)dy\ge -(r/M_0)^2 \int_{B(0, r)^c}j(|z|)dz.
\end{align}
Thus, using our choice of the positive constant  $M_0$, \eqref{e:r-di5}, \eqref{e:r-di7} and Lemma 
\ref{l:r-time2}, we have
\begin{align}
&\gener f(x) =P.V.\int_{\bB_x(r)}(f(y)-x_d^2)j(|y-x|)dy+ \int_{\bB_x(r)^c}(f(y)-x_d^2)j(|y-x|)dy\nn\\
&=P.V.\int_{\bB_x(r)}([(y_d)_+]^2-x_d^2)j(|y-x|)dy+ \int_{\bB_x(r)^c}(f(y)-x_d^2)j(|y-x|)dy
\nn\\
&\qquad+\int_{\bB_x(r)}(f(y)-[(y_d)_+]^2)j(|y-x|)dy
\label{e:r-di99}\\
 & \ge c_7 \frac{r^2}{\Phi(r)}- \int_{\bB_x(r) }|[(y_d)_+]^2-(\delta_D(y))_+^2|j(|y-x|)dy \ge (c_7-rc_6) \frac{r^2}{\Phi(r)}.\label{e:r-di9}
\end{align}
Let $c_8 := 
(1\wedge R_0 \wedge (c_7/c_6))/4
$. 
Then, from
\eqref{e:r-di9} and \eqref{e:r-di14} we conclude that 
 for all $r \le c_8$, $z \in \partial D$ and $x\in D\cap B(z, r/M_0)$, 
 $
 \gener f(x) > 2^{-1} c_7 \frac{r^2}{\Phi(r)}, 
$ and, by Lemma 
\ref{l:r-time2}, \eqref{e:r-di6}, \eqref{e:r-di14} and \eqref{e:r-di99} we also have 
\begin{align*}
&\gener f(x) \le c_9\frac{r^2}{\Phi(r)} +r^2 \int_{\bB_x(r)^c}j(|y-x|)dy+\int_{\bB_x(r)}|f(y)-[(y_d)_+]^2|j(|y-x|)dy \le c_{10} \frac{r^2}{\Phi(r)}.
\end{align*}
We have proved the lemma.
\qed

Since \eqref{e:AL} holds, we have  Dynkin's formula for $\gener$: for each $g \in C_c^{2}(\R^d)$ and any bounded open subset $U$ of $\R^d$ we have
\begin{equation} \label{n_Dynkin_formula}
\E_x\int_0^{\tau_U}  \gener g(Z_t) dt =\E_x[g(Z_{\tau_U})]- g(x).
\end{equation}

Note that, since $H$ may not be comparable to $\phi$, the next result can not be obtained using L\'evy system and \eqref{e:jlow}.

\begin{prop}\label{p:boundn1}
Suppose that   $D$ is a $C^{1,1}$ open set in $\R^d$ with characteristics $(R_0, \Lambda)$.
Let $\wt R$ be the constant in Lemma \ref{l:r-time3}.
There exists a constant $c >0$ such that for any $z\in \partial D$, $r \le \wt R$, open set $U \subset  D\cap B(z, r/M_0)$, 
and $x \in U$, 
$$\P_{x}(X_{\tau_U} \in B(z, 2r ))  \ge c_1
 \frac{\E_{x}[\tau_U]}{\Phi(r)}.  $$
\end{prop}
\pf 
Fix $z\in \partial D$, $r\le \wt R$ and an open set $U \subset  D\cap B(z, r/M_0)$. Define $f(y)=(\delta_D(y))^2{\bf 1}_{D\cap B(z, 2r)} (y).$
Then by
Lemma \ref{l:r-time3},
there exists $c_1=c_1(\phi,  \Lambda, d) \in (0,1)$ such that for all $r \le \wt R$ and $y\in D\cap B(z, r/M_0)$, 
$
 c_1^{-1} \frac{r^2}{\Phi(r)} \ge \gener f(y)\ge c_1 \frac{r^2}{\Phi(r)}.
$
Let $v\ge 0$ be a smooth radial function such that $v(y)=0$ for $|y|>1$ and $\int_{\R^d} v(y) dy=1$.
For $k\geq 1$,  define $v_k(y):=2^{kd} v (2^k y)$ and $f^{(k)}_{ r}:= v_k*f \in C_c^2(\R^d)$,
and let 
$B_k:=\{y \in U: \delta_{U }(y) \ge 2^{-k}\}$. 
We note that
\begin{align*}
&\int_{|w-y|>\eps} (f^{(k)}_r(y)-f^{(k)}_r(w))j(|w-y|)dy\\
=&\int_{|u|<2^{-k}}v_k(u)\int_{|w-y|>\eps}\(f(y-u)-f(w-u)\)j(|w-y|)dydu.
\end{align*}
By letting $\varepsilon \downarrow 0$ and using the dominated convergence theorem,
it follows that for $w \in B_k$ and all large $k$,
$$
\gener f^{(k)}_r(w)=  \int_{|u|<2^{-k}} v_k(u) \gener^u f(w)\, du  \ge  c_1 \frac{r^2}{\Phi(r)} \int_{|u|<2^{-k}} v_k(u) \, du =  c_1 \frac{r^2}{\Phi(r)}.
$$
Therefore, by the Dynkin's formula in \eqref{n_Dynkin_formula}  we have that for $x \in B_k$ and all large $k$,
$$
c_1 r^2\frac{\E_x[\tau_{B_k}]}{\Phi(r)}  \le \E_{x}\int_0^{\tau_ {B_k}} \sL f^{(k)}_r(X_s)ds = \E_{x} f^{(k)}_r(X_{\tau_ {B_k} }). $$
By letting $k\to \infty$, for any $x \in U$, we conclude that 
$$
\P_{x}(X_{\tau_U} \in B(z, 2r )\setminus U)  \ge
\frac{\E_{x} f(X_{\tau_U })}{\sup_{z \in \text{supp}(f)\setminus U} f(z)}
 \ge  c_1\frac{\E_{x}[\tau_U]}{4\Phi(r)}.
$$
\qed

Let $X^d$ be the last coordinate of $X$ and let $L_t$ be the local time {at $0$} for
{($\sup_{s\le t}X^d_s)-X^d_t$}.
Using its right-continuous inverse $L^{-1}_s$, 
define the ascending ladder-height process as
$H_s = X_{L^{-1}_s}^{d} $.
We define  $V$, 
the renewal function of the ascending ladder-height process $H$, as
\begin{equation*}\label{e:defV}
V(x) = \int_0^{\infty}\P(H_s \le x)ds, \quad
x {\in \R}.
\end{equation*}
 It is well-known that $V$ is subadditive (see \cite[p.74]{Be}). 
 Note that, since  the resolvent measure of $X^d_t$ is absolutely continuous, 
by \cite[Theorem 2]{Si}, 
 $V$ is absolutely continuous and $V$ and $V^\prime$ are 
 harmonic  for the process $X_t^d$ on $(0,\infty)$. 
Thus,  by the strong Markov property, 
$V((x_d)_+)$ and $V'((x_d)_+)$ are  harmonic in  the upper half space $\R^d_+:=\{x=(\widetilde{x},x_d) \in \R^d:x_d>0 \}$ with respect to $X$.
Furthermore, the function $V(r)$ is comparable to $\Phi(r)^{1/2}$ (see \cite[Corollary 3]{BGR1}): there exists $c>1$ such that 
\begin{equation}\label{compVhp}
c^{-1}\Phi(r)^{1/2} \le V(r) \le c \Phi(r)^{1/2} \qquad\mbox{ for any }\,\, r>0.
\end{equation} 

Using \cite[(2.23) and Lemma 3.5]{BGR2}, we see that \cite[Proposition 3.2]{GKK} also holds in our setting. 
Moreover, if we assume \eqref{e:jdouble}, then we can use  \cite[Theorem 1]{KR} so that \cite[Proposition 3.1]{GKK} holds in our setting too.
Therefore, by following the proof of 
\cite[Proposition 3.3]{GKK} line by line, we have the following.
\begin{thm}\label{st3} 
Let $w(x):=V((x_d)_+)$. Suppose that  \eqref{e:jdouble} holds. Then, for any $x\in \R^d_+$, $\gener w(x)$ is well-defined and $\gener w(x)=0$.
\end{thm}

We observe that, by a direct calculation using \eqref{e:r-di1},  
$$
\left(\frac{s}{\Phi(s)^{1/2}}\right)'=\left((\frac{s^2}{\Phi(s)})^{1/2}\right)'=2^{-1} (\frac{s^2}{\Phi(s)})^{-1/2} 2\frac{s}{\psi(s)}=
\frac{\Phi(s)^{1/2}}{\psi(s)}.
$$
Thus, using this and the fact $\lim_{s \to 0}{s}{\Phi(s)^{-1/2}}=0$ which also can be seen from \eqref{e:r-di1}, we have
\begin{align}
\label{e:pPre1}
\int _0 ^{r}
\frac{\Phi(s)^{1/2}}{\psi(s)}
 ds=
 \int _0 ^{r}
\left(\frac{s}{\Phi(s)^{1/2}}\right)'
 ds =\frac{r}{\Phi(r)^{1/2}}. 
\end{align}

\begin{lemma}\label{l:newcomp}
Assume that $H$ satisfies $ L_a(\gamma, C_L)$ and $ U_a(\delta, C_U)$ with $\delta<2$ and $\gamma >2^{-1}{\bf 1}_{\delta \ge 1}$
   for some $a \ge 0$.
Let $\gamma_1:=\gamma {\bf 1}_{\delta <1}+ (2\gamma-1){\bf 1}_{\delta \ge 1}>0$.
There exist $c_1, c_2, c_3>0$ such that for all positive constants $R \le 1$ and $\lambda>1$,
\begin{align}
\label{e:newco1}
\int_{R/\lambda}^{R}\frac{\Phi(t)^{1/2}}{\psi(t)t}dt \ge {c_1}
(\lambda^{\gamma_1}-1) \frac{\Phi(R)^{1/2}}
{\psi(R)}  \ge c_2 (\lambda^{\gamma_1}-1) R^{-\gamma_1}, \end{align}
and 
\begin{align}
\label{e:newco2}
\int_{R}^1\frac{\Phi(t)^{1/2}}{\psi(t)t}dt \le c_3\frac{\Phi(R)^{1/2}}{\psi(R)}.
\end{align}
\end{lemma}
\pf
If $\delta <1$ then $\psi$ and $\Phi$ are comparable near $0$, thus, by \eqref{eqn:poly} for $t \le R \le 1$,
\begin{align*}
\frac{\Phi(t)^{1/2}}
{\Phi(R)^{1/2}} \frac{\psi(R)}{\psi(t)} \ge c_1 \frac{\Phi(R)^{1/2}}{\Phi(t)^{1/2}}
\ge c_2(t/R)^{ -\gamma} .
\end{align*}
 By \eqref{eqn:polypsi} and Lemma \ref{lem:bf}(a), if $\delta  \ge 1$ then  for $t \le R \le 1$, 
\begin{align*}
  \frac{\Phi(t)^{1/2}}
{\Phi(R)^{1/2}} \frac{\psi(R)}{\psi(t)} \ge c_3 (t/R) (R/t)^{2 \gamma}=c_3(t/R)^{1-2 \gamma} .
\end{align*}
Thus, for $t \le R \le 1$, 
\begin{align}\label{e:newcomp1}
 \frac{\Phi(t)^{1/2}}
{\Phi(R)^{1/2}} \frac{\psi(R)}{\psi(t)} \ge  c_4(t/R)^{-\gamma_1} .
\end{align}
Using \eqref{e:newcomp1} we have that for all $R \le 1$ and $\lambda>1$, \begin{align*}
&\int_{R/\lambda}^{R}\frac{\Phi(t)^{1/2}}{\psi(t)t}dt = \frac{\Phi(R)^{1/2}}
{\psi(R)} 
\int_{R/\lambda}^{R} \frac{\Phi(t)^{1/2}}
{\Phi(R)^{1/2}} \frac{\psi(R)}{\psi(t)t}     dt
\ge c_4  \frac{\Phi(R)^{1/2}}
{\psi(R)} R^{\gamma_1}\int_{R/\lambda}^{R}t^{-\gamma_1-1} dt\\
&=\frac{c_4}{\gamma_1} \frac{\Phi(R)^{1/2}}
{\psi(R)}  R^{\gamma_1} 
((R/\lambda)^{-\gamma_1}-R^{-\gamma_1})
=\frac{c_4}{\gamma_1} 
(\lambda^{\gamma_1}-1) \frac{\Phi(R)^{1/2}}
{\psi(R)}, 
\end{align*}
and
\begin{align*}
&\int_{R}^1\frac{\Phi(t)^{1/2}}{\psi(t)t}dt = \frac{\Phi(R)^{1/2}}
{\psi(R)} 
\int_{R}^1\frac{\Phi(t)^{1/2}}
{\Phi(R)^{1/2}} \frac{\psi(R)}{\psi(t)t}     dt
\le c_4^{-1} R^{\gamma_1} \frac{\Phi(R)^{1/2}}
{\psi(R)} 
\int_{R}^1t^{-\gamma_1-1} dt\\
&=\frac{c_4^{-1}}{\gamma_1} R^{\gamma_1} 
(R^{-\gamma_1}-1) \frac{\Phi(R)^{1/2}}
{\psi(R)} 
\le \frac{c_4^{-1}}{\gamma_1}\frac{\Phi(R)^{1/2}}{\psi(R)}.
\end{align*}
The second inequality in \eqref{e:newco1} also follows from \eqref{e:newcomp1} (with $R=1$ and $t=R$).
\qed

\begin{prop}\label{g:st1}
Let $D\subset \R^d$ be a $C^{1,1}$ open set with characteristics
$(R_0, \Lambda)$. 
Assume that  \eqref{e:jdouble} holds and that $H$ satisfies $ L_a(\gamma, C_L)$ and $ U_a(\delta, C_U)$ with $a > 0$, $\delta<2$ and $\gamma >2^{-1}{\bf 1}_{\delta \ge 1}$.
For any $z\in \partial D$ and $r\le 1\wedge R_0$, we define 
$$h_r(y)=h_{r,z}(y):=V(\delta_D(y)){\bf 1}_{D\cap B(z, r)} (y).$$
Then, there exists $C_*=C_*(\phi,  \Lambda, d)>0$ independent of $z$ such that $\gener h_r$ is well-defined in $D\cap B(z, r/4)$ and 
\begin{align}\label{gener}
|\gener h_r(x)|\le C_* \frac{\Phi(r)^{1/2}}
{\psi(r)}\quad\mbox{ for all } x\in D\cap B(z, r/4). 
\end{align}
\end{prop}

\begin{proof}
Since the case of $d=1$ is easier, we give the proof only for $d\ge 2$. 
Without loss of generality we assume that $\Lambda>1$.

For $x\in D\cap B(z, r/4)$, let $z_x\in \partial D$ be a point satisfying $\delta_D(x)=|x-z_x|$.
Let $\varphi$ be a  $C^{1,1}$ function and $CS=CS_{z_x}$ be an orthonormal coordinate system with $z_x$ as the origin so that $\varphi(\wt 0)=0$, $ \nabla\varphi(\wt 0)=(0, \dots, 0)$, $\| \nabla \varphi \|_\infty \leq \Lambda$, $|\nabla \varphi (\wt y)-\nabla \varphi (\wt z)| \leq \Lambda |\wt y-\wt z|$, and $x=(\wt{0}, x_d)$, $ D\cap B(z_x, R_0)=\{y=(\wt{y}, y_d)  \in B(0, R_0) \mbox{ in } CS: y_d>\varphi(\wt{y})\}$.
We fix the function $\varphi$ and the coordinate system $CS$, and define a function $g_x(y)=V(\delta_{\R^d_+}(y))=V(y_d)$, where  $\R^d_+=\{y=(\widetilde{y}, y_d) \mbox{  in } CS :y_d>0\}$ is the half space in $CS$. 

Note that  $h_r(x)=g_x(x)$, and  that $\gener (h_r- g_x)=\gener h_r$ by Theorem \ref{st3}.
So, it suffices to show that $\gener (h_r- g_x)$ is well defined and
that there exists a constant $c_0>0$ independent of  $x\in D\cap B(z, r/4)$ and $z \in \partial D$  such that
\begin{align}\label{e:claims}
\int_{D\cup \R^d_+}|h_r(y)-g_x(y)|j(|x-y|) dy\le c_0\frac{\Phi(r)^{1/2}}
{\psi(r)}. \end{align}

We define $\wh \varphi :B(\wt{0}, r)\to \R$ by $\wh \varphi(\wt{y}):=2\Lambda|\wt{y}|^{2}$.
Since $\nabla\varphi (\wt{0})=0$, by the mean value theorem we have 
$-\wh \varphi(\wt{y})\le\varphi(\wt{y})\le \wh \varphi(\wt{y})$  for any $y\in D\cap B(x, r/2)$
and so that
\begin{align*}
\{z=(\widetilde{z}, z_d)\in B(x, r/2) : z_d\ge& \wh{\varphi}(\widetilde z)\}\subset D \cap  B(x, r/2)\nn\\
&\subset \{z=(\widetilde{z}, z_d)\in B(x, r/2) : z_d\ge -\wh{\varphi}(\widetilde z)\}.
\end{align*}
Let $A:=\{y\in (D\cup \R^d_+)\cap B(x, r/4):-\wh{ \varphi}(\widetilde{y})\le y_d\le \wh{\varphi}(\widetilde{y})\}$,  
$E:=\{y\in B(x, r/4): y_d> \wh{\varphi}(\widetilde{y})\}\subset D$, 
\begin{align*}
&{\rm I}:=\int_{B(x, r/4)^c}( h_r(y)+g_x(y))j(|x-y|)dy=\int_{B(0, r/4)^c}( h_r(x+z)+g_x(x+z))j(|z|)dz,\nn\\
&{\rm II}:=\int_A (h_r(y)+g_x(y))j(|x-y|)dy,\qquad \mbox{ and }\qquad {\rm III}:=\int_E |h_r(y)-g_x(y)|j(|x-y|)dy.
\end{align*}

First, since $h_r\le V(r)$ and $V(x_d+z_d)\le V(x_d)+V(|z|)$, we have
\begin{align}\label{e:I}
{\rm I}&\le V(r)
\int_{B(0, {r}/{4})^c} j(|z|)dz +\int_{B(0, {r}/{4})^c} \left(V(x_d)+V(|z|)\right)j(|z|)dz\nn\\
& \le c_1 V(r) \left( \int_{r/4}^1 j(s)s^{d-1}ds +1\right)
+ \left( \int_{r/4}^1 j(s)V(s) s^{d-1}ds + \int_{1}^{\infty} j(s)V(s) s^{d-1}ds\right)\nn\\&\le c_2\left(\frac{\Phi(r)^{1/2}}
{\psi(r)} +\int_{r/4}^1 \frac{\Phi(s)^{1/2}ds}{s\psi(s)}+1\right)\le c_3\frac{\Phi(r)^{1/2}}
{\psi(r)} \end{align}
In the second to last inequality above, we have used \eqref{e:t:up1j}, \eqref{eqn:polypsi}, \eqref{compVhp} and \cite[Lemma 3.5]{BGR2}.
In the last inequality above, we have used Lemma \ref{l:newcomp}.

Second, let $m_{d-1}(dy)$ be the Lebesgue  measure on $\R^{d-1}$.
Since $m_{d-1}(\{y:|\widetilde{y}|=s, - \wh{\varphi}(\widetilde{y})\le y_d\le \wh{\varphi}(\widetilde{y})\}) \le c_4 s^d$ for $0<s<r/4$,
and $h_r(y)+g_x(y)\le 2V(2\wh{\varphi} (\widetilde{y}))\le 8(\Lambda+1) V(|\widetilde {y}|)$, 
we get 
$${\rm II}\le 8(\Lambda+1) \int_{0}^{r/4}
\int_{|\widetilde{y}|=s} {\bf1}_A(y) V(|\widetilde {y}|)\nu( |\widetilde{y}|)m_{d-1}(dy)ds
 \le 8c_4(\Lambda+1)\int _0 ^{r}V(s)j(s)s^{d} ds.$$

 Thus, by \eqref{e:t:up1j}, \eqref{compVhp} and \eqref{e:pPre1},
 \begin{align}\label{e:II}
{\rm II}\le \,c_5\,\int _0 ^{r}
\frac{\Phi(s)^{1/2}}{\psi(s)}
 ds=c_5\frac{r}{\Phi(r)^{1/2}}  \le c_5\frac{1}{\Phi(1)^{1/2}} =c_5.
\end{align}
Lastly, when $y\in E$, using
$|y_d-\delta_D(y)|\le  (1\wedge R_0)^{-1}\wh{\varphi}(\widetilde{y})=
2\Lambda (1\wedge R_0)^{-1}|\wt{y}|^{2}
$ (See the proof of Lemma \ref{l:r-time3}.)
and 
the scale invariant Harnack inequality for $X^d$ to $V'$ (Theorem \ref{t:PHI}), we have that
$$
|h_r(y)-g_x(y)| \le \Big(\sup_{u\in (y_d\wedge \delta_D(y), \,\, y_d\vee \delta_D(y))} 
V^{\prime}\(u\)  \Big)|y_d-\delta_D(y)| \le 
c_6V^{\prime}\(y_d-\wh{\varphi}(\wt{y})\) |\widetilde{y}|^{2}$$
Since  $V^{\prime}(s)\le c_8 s^{-1}{V(s)}\le c_9 s^{-1} \Phi(s)^{1/2}$ by \cite[Theorem 1]{KR}  and \eqref{compVhp}, using \eqref{e:t:up1j} and the polar coordinates for $|\wt{y}|=v$ and the change of variable $s:= y_d-\wh{\varphi}(v)$, we obtain 
\begin{align*}
{\rm III} \le&\, c_6 \int_{\{ (\wt{y}, y_d): |\widetilde{y}|<r/4, \, \wh{\varphi}(\widetilde{y})  <y_d< \wh{\varphi}(\widetilde{y})+r/2\}} V^{\prime}(y_d-\wh{\varphi}(\wt{y}))|\widetilde{y}|^{2} j(|x-y|)dy\\
\le&\, c_7\int_0^{r/4}\int_{0}^{r/2}\frac{V^{\prime}(s)}{ \psi((v^2+|s+\wh{\varphi}(r)-x_d|^2)^{1/2})}
\frac{v^d}{(v^2+|s+\wh{\varphi}(r)-x_d|^2)^{d/2}}
dy_d dv\\
 \le&c_8\int_0^{r/4}\int_{0}^{r/2}\frac{s^{-1}\Phi(s)^{1/2}}{ \psi((v^2+|s+\wh{\varphi}(r)-x_d|^2)^{1/2})}
dy_d dv.
\end{align*}
Applying \cite[Lemma 4.4]{KSV14} with 
non-increasing functions $s^{-1}\Phi(s)^{1/2}$ and $f(s):= \psi(s)^{-1}$ and 
$x(r)=x_d-\wh{\varphi}(r)$, we have that  
\begin{align}\label{e:III0}
{\rm III}\le c_{9} \int_0^{r}\int_0^u \frac{\Phi(s)^{1/2}}{s}ds\frac{du}{\psi(u)} =: c_{9}\, {\rm IV}.
\end{align}
We claim that  ${\rm IV} \le  c_{10} <\infty.$

If $\delta <1$ then $\psi(t)^{1/2}$, $\Phi(t)^{1/2}$ 
and $\int_0^u \frac{\Phi(s)^{1/2}}{s}ds $ are comparable near zero. Thus, by \eqref{eqn:poly},
$${\rm IV}\le c_{11} \int_0^r \Phi^{-1/2}(u)du \le c_{12} \int_0^r u^{-\delta}du\le c_{12} \int_0^1 u^{-\delta}du \le c_{13}.
$$
If $\delta \ge 1$, using the assumption $\gamma >2^{-1}$, we see from \eqref{eqn:polypsi} that for $s<u<r$,
$$
\int_s^r  \frac{\psi(s)}{\psi(u)}   du \le c_{14} s^{2\gamma}   \int_s^r u^{-2\gamma}   du =
\frac{c_{14}}{2\gamma-1} s^{2\gamma} (s^{1-2\gamma}-r^{1-2\gamma}) \le \frac{c_{14}}{2\gamma-1}  s.
$$
Thus, using \eqref{e:pPre1} and the fact that $\frac{r}{\Phi(r)^{1/2}}$ is non-decreasing,
\begin{align*}
 {\rm IV} =\int_0^{r} \left(\int_s^{r}  \frac{\psi(s)}{\psi(u)}   du \right) \frac{\Phi(s)^{1/2}}{s\psi(s)}ds ds \le \frac{c_{14}}{2\gamma-1}\int_0^{r} \frac{\Phi(u)^{1/2}}{\psi(u)}du
= \frac{c_{14}}{2\gamma-1} \frac{r}{\Phi(r)^{1/2}} \le \frac{c_{14}}{2\gamma-1}.
\end{align*}
We have proved the claim  ${\rm IV} \le  c_{10} <\infty.$
Combining \eqref{e:I}--\eqref{e:III0} with this and using Lemma \ref{l:newcomp}, we conclude that  \eqref{e:claims} holds.
\end{proof}

We are now ready to prove key estimates on exit probabilities.
\begin{thm}\label{t:n_estimates}
Let $D\subset \R^d$ be a $C^{1,1}$ open set with characteristics
$(R_0, \Lambda)$. 
Assume that  \eqref{e:jdouble} holds and that $H$ satisfies $ L_a(\gamma, C_L)$ and $ U_a(\delta, C_U)$ with $\delta<2$ and $\gamma >2^{-1}{\bf 1}_{\delta \ge 1}$
   for some $a > 0$. Then
  there exist  positive constants  $R_* < (R_0 \wedge1)/4$ and $c_1, c_2>1$    such that 
  the following two estimates hold true.

   \noindent
   (a) For every 
 $R \le R_*$, $z\in \partial D$,
   open set $U \subset D \cap B(z, R)$ and $x \in U$,
 \begin{align}\label{e:n_estimates1}
 \E_{x}[\tau_{U}] \ge c_1^{-1} \Phi(\delta_D(x))^{1/2} \Phi(R)^{1/2}.
\end{align}
 
    \noindent
   (b) For every 
   $R \le R_*$, $z\in \partial D$ and 
   $x \in D_z( 2^{-3}R,2^{-4}R) $, 
    \begin{align}\label{e:n_estimates2}\E_{x}[\tau_{D_z(R,R)}] \le c_2
\Phi(R)   \P_{x}(X_{ \tau_{D_z(R,R)} }\in D_z(2R,R) )
   \le c_1 c_2\Phi(\delta_D(x))^{1/2} \Phi(R)^{1/2}.\end{align}
  \end{thm}

 \pf  
Fix $R  \le 1\wedge R_0$ and without loss of generality, we assume $z=0$.  
Define $h(y)=V(\delta_D(y)){\bf 1}_{D\cap B(0, R)} (y).$

Using the same approximation argument in 
the proof of Proposition \ref{p:boundn1} and the Dynkin's formula, we have that, 
for every $\lambda \ge 4$, open set $U \subset D \cap B(0, \lambda^{-1}R)$  and $x \in U$,
\begin{align*}
 \E_{x}\left[h_R\left(X_{ \tau_{U}} \right)  \right]+ C_*  \frac{\Phi(R)^{1/2}}
{\psi(R)} \E_x\left[\tau_{U}\right] \ge V(\delta_D(x)) \ge \E_{x}\left[h_R\left(X_{ \tau_{U}} \right)  \right] - C_*  \frac{\Phi(R)^{1/2}}
{\psi(R)} \E_x\left[\tau_{U}\right],
\end{align*}
where $C_*>0$ is the constant in Proposition \ref{g:st1}.

Since $j(|y-z| ) \ge  j(2|y|) \ge  c_1 |y|^{-d} \psi(|y|)^{-1}$ for any $z  \in D \,\cap\, B(0, \lambda^{-1}R)$ and $y \in D\,\cap \,(B(0, R) \setminus B(0, \lambda^{-1}R))$,  by L\'evy system we obtain 
\begin{align*}
\E_x\left[h_R\left(X_{\tau_{U}}\right)\right]
&\ge  \E_x \int_{ D\cap \(B(0,R)\setminus B(0, \lambda^{-1}R)\)} \int_0^{\tau_{U}}  j(|X_t-y|)dt h_R(y) dy \nonumber \\
&\ge  \, c_1  \, \E_x\left[\tau_{U}\right] \int_{  D\cap \(B(0, R)\setminus B(0, \lambda^{-1}R)\)}   |y|^{-d} \psi(|y|)^{-1}h_R(y) dy.
\end{align*}
Let $A:=\{(\wt{y}, y_d): 2 \Lambda |\wt y|<y_d\}$. 
Since $y_d>2 \Lambda |\wt y|>2\Lambda |\wt{y}|^{2}>\varphi(\wt y)$ for any $y\in A \cap B(0, R)$,    we have $A \cap B(0, R)\subset D\cap B(0,R)$ and for any $y\in A \cap B(0, R)$, 
 $$
\delta_D(y)\ge (2\Lambda)^{-1} \left(y_d-\varphi(\wt y)\right) \ge  (2\Lambda)^{-1} (y_d -\Lambda|\wt y|)> (4\Lambda)^{-1}y_d \ge  (4\Lambda((2\Lambda)^{-2}+1)^{1/2})^{-1}|y|.
$$
By this and changing to polar coordinates with $|y|=t$ and  \eqref{compVhp}, we obtain that
\begin{align*} 
&\int_{  D\cap \(B(0, R)\setminus B(0, \lambda^{-1}R)\)}   |y|^{-d} \psi(|y|)^{-1}h_R(y) dy 
\nonumber\\
&\ge c_2\int_{ A\cap \(B(0, R)\setminus B(0, \lambda^{-1}R)\)}|y|^{-d} \psi(|y|)^{-1}V(|y|) dy
 \ge c_3\int_{\lambda^{-1}R}^{R}\frac{\Phi(t)^{1/2}}{\psi(t)t}dt.
\end{align*}
By Lemma \ref{l:newcomp}, the above is great than  $ c_4
(\lambda^{\gamma_1}-1) \frac{\Phi(R)^{1/2}}
{\psi(R)}.$ Thus, 
we can use a $\lambda_0$ large (In fact, one can choose $\lambda_0=(1+ c_1^{-1}c_4^{-1} 2C_*)^{-\gamma_1}$.) so that for all $\lambda \ge \lambda_0$, $R  \in (0, 1\wedge R_0)$
 and for every open set $U \subset D \cap B(0, \lambda^{-1}R)$,
\begin{align}
&V(\delta_D(x))\ge \E_{x}\left[h_R\left(X_{ \tau_{U}} \right)  \right] - C_*  \frac{\Phi(R)^{1/2}}
{\psi(R)} \E_x\left[\tau_{U}\right] \ge \frac12\E_{x}\left[h_R\left(X_{ \tau_{U}} \right)  \right]\label{e:ss22}\\ 
\mbox{ and }\,\, &V(\delta_{D}(x))  \le   \E_{x}\left[h_R\left(X_{ \tau_{U}} \right)  \right]+ C_*  \frac{\Phi(R)^{1/2}}
{\psi(R)}\E_x\left[\tau_{U}\right] \le  \frac32\E_{x}\left[h_R\left(X_{ \tau_{U}} \right)  \right].\label{e:ss222}
\end{align}
By \cite[Lemma 2.4]{CK}, \eqref{e:ss222} and \eqref{compVhp}, we get 
\begin{align}
&\frac{2}{3} V(\delta_D(x))\le  \E_{x}\left[h_R\left(X_{ \tau_{U}} \right)  \right]  \le 
V(R) \P_{x}\left(X_{ \tau_{U}} \in D \cap B(0, R)\right) \nn \\
&\le c_5 V(R) \Phi(R)^{-1}\E_{x}[\tau_{U}]
\le c_6 \Phi(R)^{-1/2}\E_{x}[\tau_{U}].
\label{e:ss1new0}
\end{align}
Now, \eqref{e:n_estimates1} follows from \eqref{compVhp} and \eqref{e:ss1new0}.

Let $\lambda_1:=\lambda_0 \vee M_0$ where $M_0$ is the constant in \eqref{e:r-diM} and $U_1:=U_1(R):=D_0(\kappa\lambda_1^{-1}R, \kappa \lambda_1^{-1}R)  \subset D \cap B(0, \lambda_1^{-1}R)$. Then, 
by \eqref{e:ss22} and Proposition \ref{p:H12}
for all $x \in U_2:=D_0(2^{-3}\kappa\lambda_1^{-1}R, 2^{-4}\kappa \lambda_1^{-1}R)$
\begin{align}
&2 V(\delta_D(x))\ge  \E_{x}\left[h_R\left(X_{ \tau_{U_1}} \right) {\bf 1}_{ D_0(2\kappa\lambda_1^{-1}R, \kappa \lambda_1^{-1}R)   } \right] \nn\\& \ge 
V(\kappa\lambda_1^{-1}R) \P_{x}\left(X_{ \tau_{U_1}} \in D_0(2\kappa\lambda_1^{-1}R, \kappa \lambda_1^{-1}R)  \right)  \ge c_7 \Phi(R)^{1/2} \P_{x}\left(X_{ \tau_{U_1}} \in D \right) .
\label{e:ss1new00}
\end{align}
Recall that $\wt R$ is the constant in Lemma \ref{l:r-time3} and  Proposition \ref{p:boundn1}.
Applying  Proposition \ref{p:boundn1} and \eqref{compVhp} to \eqref{e:ss1new00}, we conclude that 
for all $R \le \wt R$ and all $x \in U_2$,
\begin{align*} \E_{x}[\tau_{U_1}] \le c_{8}   \Phi(R)\P_{x}\left(X_{ \tau_{U_1}} \in D \right)&\le c_{9}   \Phi(R)
\P_{x}\left(X_{ \tau_{U_1}} \in D_0(2\kappa\lambda_1^{-1}R, \kappa \lambda_1^{-1}R)  \right)
 \nn\\
&\le c_{10} \Phi(\delta_D(x))^{1/2} \Phi(R)^{-1/2}. \end{align*}
By taking $R_*=\wt R\lambda_1^{-1}\kappa$ we have proved \eqref{e:n_estimates2}.
\qed

\section{Upper bound estimates}\label{S:4}
In this section we discuss  the upper bound of the Dirichlet heat kernels on $C^{1,1}$ open sets. 
{\it Throughout the remainder of this paper, we always assume that   \eqref{e:jdouble} holds, that $\phi$  has no drift and that $H$ satisfies $ L_a(\gamma, C_L)$ and $ U_a(\delta, C_U)$ with $\delta<2$ and $\gamma >2^{-1}{\bf 1}_{\delta \ge 1}$
   for some $a > 0$.}

We first establish  sharp estimates on the survival probability. Lemma \ref{l:spH} is proved in \cite{BGR2} when 
weak scaling order of characteristic exponent is  strictly below $2$. We emphasize here that results in  \cite{BGR2} can not be used here.

\begin{lemma}
\label{l:spH}
 Suppose $D$ is a $C^{1,1}$ open set with the characteristic $(R_0, \Lambda)$. Then for every $T>0$ there exists 
 $C_1=C_1(T, R_0, \Lambda)>0$ such that  for $ t \in (0,  T]$,
\begin{align}
\label{e:ulbest0}
 \P_{x}(\tau_{D}>t) \,\le\, C_1\
\left(\sqrt{\frac{\Phi(\delta_{D}(x))}{t}} \wedge 1\right), \quad \text{ for all } x \in D,
\end{align}
and there exist $T_1 \in (0, \Phi(R_0)]$ and  $C_2>0$ such that   for $ t \in (0,  T_1]$,
\begin{align}
\label{e:ulbest00}
 \P_{x}(\tau_{D}>t) 
\ge C_2
\left(\sqrt{\frac{\Phi(\delta_{D}(x))}{t}} \wedge 1\right), \quad \text{ for all } x \in D.
\end{align}
\end{lemma}
\pf
Recall thet $R_*>0$ is the constant in Theorem \ref{t:n_estimates}.
Let $b:=\Phi(R_*/4)/T$ and $r_t:=\Phi^{-1}(bt)$ for $t \le T$ so that  $r_t \le R_*/4$.
First note that, if $\delta_D(x) \ge 2^{-4}r_t$ then, by Lemma \ref{L:3.3},
\begin{align}
\label{e:ulbest2}
\P_{x}(\tau_{D}>t) \ge \P_{x}(\tau_{B(x, \delta_D(x))}>t) \ge \P_{0}(\tau_{B(0, 2^{-4}r_t)}>t) =c_0>0.
\end{align}

We now assume that $\delta_D(x) <2^{-4} r_t$. Let $z_x\in \partial D$ with $|x-z_x| =\delta_D(x)$. Then by 
\cite[Lemma 2.4]{CK} and Theorem \ref{t:n_estimates}(b),
\begin{align}
\label{e:ulbest21}&\P_{x}\(\tau_{D}>t\) 
 =\P_{x}\( \tau_{D \cap B(z_x, r_t)} =\tau_{D}> t\)+ \P_{x}\(\tau_{D}>\tau_{D \cap B(z_x, r_t)} > t\)
 \nn\\&\le  \P_{x}\( \tau_{D \cap B(z_x, r_t)} > t\)+ \P_{x}\(X_{ \tau_{D \cap B(z_x, r_t)}}  \in D\)
 \nn\\ &\le  t^{-1}\E_{x}\left[\tau_{D \cap B(z_x, r_t)}\right] + \P_{x}\(X_{ \tau_{D \cap B(z_x, r_t)}}  \in D\) \le c_1 \Phi(\delta_D(x))^{1/2} t^{-1/2} .
\end{align}

Recall that $D_{z}(r, r)$ is defined in \eqref{e:d_com}.
Let $U(x,t):=D_{z_x}(r_t, r_t)$.
 For the lower bound, we use the strong Markov property and Theorem \ref{t:n_estimates}(b)
to get that for any $b \ge 1$ and $t \le T/b$, 
\begin{align}\label{e:ulbest22}
&\P_x\(\tau_D>bt\)\nn\\
 \ge &\P_x\Big(\tau_{U(x,t)}<bt, X_{\tau_{U(x,t)}} \in D_{z_x}(2r_t, r_t),  |X_{\tau_{U(x,t)}}-X_{\tau_{U(x,t)}+s}| \le \frac{r_t}{4} \text{ for all } 0<s<bt\Big)\nn\\
\ge &\P_x\(\tau_{U(x,t)}<bt, X_{\tau_{U(x,t)}} \in D_{z_x}(2r_t, r_t)\) \P_0\(\tau_{B_{r_t/4}}>bt\)\nn\\
\ge &\P_0\(\tau_{B_{r_t/4}}>bt\)\(\P_x\big(X_{\tau_{U(x,t)}} \in D_{z_x}(2r_t, r_t)\big)-\P_x\big(\tau_{U(x,t)}\ge bt\big)\)\nn\\
\ge& \P_0\(\tau_{B_{r_t/4}}>bt\)\(c_2t^{-1}\E_x[\tau_{U(x,t)}]-b^{-1}t^{-1}\E_x\left[\tau_{U(x,t)}\right]\).
\end{align}
Take $b=\frac{2}{c_2} \vee 1$. Then, by Lemma \ref{L:3.3} and Theorem \ref{t:n_estimates}(a) we have
from \eqref{e:ulbest22}
 that for 
$t \le T_0:=T/b$
\begin{align}\label{e:ulbest23}
\P_x\(\tau_D>t\)\ge \P_x\(\tau_D>bt\) \ge c_3t^{-1}\E_x\left[\tau_{U(x,t)}\right] \ge \Phi(\delta_D(x))^{1/2} t^{-1/2}.
\end{align}
Combining \eqref{e:ulbest2}, \eqref{e:ulbest21} and \eqref{e:ulbest23}, we have proved the lemma.
\qed

Using 
\cite[Lemmas 2.5 and 2.8]{CK}, Lemma \ref{l:spH} and  Theorem \ref{t:n_estimates}(b)
 we  obtain the following upper bound of $p_{D}(t,x,y)$.

\begin{lemma}\label{l:u_off1p}
Suppose that   $D$ is a $C^{1,1}$ open set with characteristics $(R_0, \Lambda)$.
 For each $T>0$, there exist  constants
 $c=c(a, \phi, R_0, \Lambda, T)>0$  and $a_0=a_0(\phi, R_0, T)>0$ such that  for every $(t,x,y) \in (0, T]
\times D \times D$ with $  a_0\Phi^{-1}(t) \le |x-y|$,
\begin{align}
 p_{D}(t,x,y) \nn
 \le & c
\Big(\sqrt{\frac{\Phi(\delta_{D}(x))}{t}} \wedge 1\Big)
\Big(   \sup_{(s,z):s\le t,
 \frac{|x-y|}{2}  \leq |z-y| \le \frac{3 |x-y|}{2}}  p_{D}(s, z, y)
\\ &\qquad  \qquad      + \Big(\sqrt{t \Phi(\delta_{D}(y))} \wedge  t\Big)
j({|x-y|}/{3})\Big). \label{e:dfnewp}
\end{align}
\end{lemma}

\pf 
Throughout the proof, we assume $t \in (0,  T]$ and  let $
 a_0:= {6 R_*} / {\Phi^{-1}(T)}.$ Note that $a_0\Phi^{-1}(t)/6 \le R_*$.

We first assume  $\delta_{D}(x) \le  2^{-7}a_0\Phi^{-1}(t)/3 \le 2^{-7}|x-y|/3$ and 
let 
$x_0$ be a point on $\partial D$ such that  $\delta_D(x)=|x-x_0|$ and let
$U_1:=B( x_0,
 a_0\Phi^{-1}(t)/(12)) \cap D$, $U_3:= \{z\in D: |z-x|>|x-y|/2\}$
and $U_2:=D\setminus (U_1\cup U_3)$.
Using Theorem \ref{t:n_estimates}(b) we have 
\begin{equation}\label{e:upae1p}
\E_{x}\left[ \tau_{U_1} \right]
 \le  \E_{x}\left[ \tau_{D_{x_0}(a_0\Phi^{-1}(t)/(12),a_0\Phi^{-1}(t)/(12))} \right] \le c_1 \sqrt{t \Phi(\delta_{D}(x))}.
\end{equation}
Since
$
|z-x| > 2^{-1}|x-y|  \ge a_0 2^{-1} \Phi^{-1}(t)$ for
$z\in U_3$,
we have for  $u\in U_1$ and $z\in U_3$,
$$
|u-z| \ge |z-x|-|x_0-x|-|x_0-u| \ge  \frac{1}{2}|x-y|- \frac16 a_0\Phi^{-1}(t) \ge \frac{1}{3}|x-y|.
$$
Thus, by the fact $U_1 \cap U_3 = \emptyset$ and the monotonicity of $j$,
\begin{eqnarray}
\sup_{u\in U_1,\, z\in U_3}j(|u-z|) \le \sup_{(u,z):|u-z| \ge
\frac{1}{3}|x-y|}j(|u-z|)
= j(|x-y|/3).
\label{e:n01p}
\end{eqnarray}

On the other hand,  for $z \in U_2$,
$$
\frac32 |x-y| \ge |x-y| +|x-z| \ge  |z-y| \ge |x-y| -|x-z| \ge
\frac{|x-y|}2 \ge a_02^{-1} \Phi^{-1}(t),
$$
so
\begin{align}
&\sup_{s\le t,\, z\in U_2} p_{D}(s, z, y)
\le \sup_{ s\le t,  \frac{|x-y|}{2}  \leq |z-y| \le \frac{3 |x-y|}{2}}  p_{D}(s, z, y).\label{e:n02p}
\end{align}
Furtherover, by Lemma \ref{l:spH},
\begin{align}
&\int_0^t \mathbb{P}_x\left(\tau_{U_1}>s\right) \mathbb{P}_y \left(\tau_{D}>t-s \right)ds \le
\int_0^t \mathbb{P}_x\left(\tau_{D}>s\right) \mathbb{P}_y \left(\tau_{D}>t-s \right)ds\nn \\
&\le c_3 \sqrt{\Phi(\delta_{D}(x))} \int_0^t s^{-1/2}\left(\sqrt{\frac{\Phi(\delta_{D}(y))}{t-s}} \wedge 1\right) ds\nn\\
&\le c_4 \sqrt{\Phi(\delta_{D}(x))} \left(\sqrt{\Phi(\delta_{D}(y))}  \wedge \sqrt{t}\right). \label{e:pdff} 
\end{align}
Finally,
applying \cite[Lemma 2.5]{CK}
 and then \eqref{e:upae1p}, we have
$$
\P_x \Big(X_{\tau_{U_1}}\in U_2 \Big) \le \P_x
\Big(X_{\tau_{U_1}}\in B( x_0, a\Phi^{-1}(t)/(12))^c\Big)
\,\le\,  \frac{c_5}{t}\, \E_{x}[\tau_{U_1}]
\le c_6 t^{-1/2} \,
\sqrt{\Phi(\delta_{D}(x))} .
$$

Applying 
 this and  \eqref{e:upae1p}--\eqref{e:pdff} to
\cite[Lemma 2.8]{CK}
 we conclude that 
\begin{eqnarray*}
p_{D}(t, x, y)
&\le& \left(\int_0^t \mathbb{P}_x\left(\tau_{U_1}>s\right) \mathbb{P}_y \left(\tau_{D}>t-s \right)ds \right)
\sup_{u\in U_1,\, z\in U_3}j(|u-z|) \\
&&\quad +
 \P_x\Big (X_{\tau_{U_1}}\in U_2 \Big) \sup_{s\le t,\, z\in U_2} p_D(s, z, y) \\
&\le& c_{4} \sqrt{\Phi(\delta_{D}(x))} \left(\sqrt{\Phi(\delta_{D}(y))}  \wedge \sqrt{t}\right) j(|x-y|/3)\\
&& \quad  +
c_6 t^{-1/2} \,
\sqrt{\Phi(\delta_{D}(x))} \sup_{ s\le t, \frac{|x-y|}{2}  \leq |z-y| \le \frac{3 |x-y|}{2}}  p_D(s, z, y).
\end{eqnarray*}

If $\delta_{D}(x) >   2^{-7}a_0\Phi^{-1}(t)/3$,
by Lemma \ref{lem:bf}(a),
$$
\sqrt{\frac{\Phi(\delta_{D}(x))}{t}} \ge \sqrt{\frac{\Phi( a_0\Phi^{-1}(t)/(24))}{\Phi(\Phi^{-1}(t))}}
\ge c_7>0.
$$
Thus \eqref{e:dfnewp} is clear.
Therefore we  have proved
\eqref{e:dfnewp}.
\qed

We now apply Lemma \ref{l:u_off1p} to get the upper bound of the Dirichlet heat kernel.

\noindent
\textbf{Proof of Theorem \ref{t:main0}(a)}:
We will closely follow the argument in \cite{CK}. We fix $T>0$.

By 
 \cite[Lemma 2.7]{CK}
  and Proposition  \ref{prop:on-up}, for every $(t,x,y) \in (0, T]
\times D \times D$,
$$
 p_{D}(t,x,y) \,\le\, c_1 (\Phi^{-1}(t))^{-d}
\left(\sqrt{\frac{\Phi(\delta_{D}(x))}{t}} \wedge 1\right)\left(\sqrt{\frac{\Phi(\delta_{D}(y))}{t}} \wedge 1\right)  .
$$

Recall that $a_0$ is the constant in Lemma \ref{l:u_off1p}. If $a_0\Phi^{-1}(t) \ge |x-y|$, by Proposition \ref{step1},
$ p(t,x-y) \ge  c_2  (\Phi^{-1}(t))^{-d}.$
Thus   for every $(t,x,y) \in (0, T]
\times D \times D$ with $ a_0 \Phi^{-1}(t) \ge |x-y|$,
\begin{align}
 p_{D}(t,x,y) \,\le\,c_3 \left(\sqrt{\frac{\Phi(\delta_{D}(x))}{t}} \wedge 1\right)\left(\sqrt{\frac{\Phi(\delta_{D}(y))}{t}} \wedge 1\right)  p(t,x-y). \label{e:dfghj1}
\end{align}

We extend the definition of
$p(t, w)$  by setting
 $p(t, w)=0$ for $t<0$
 and $w \in \R^d$.
 For each fixed $x, y\in \R^d$ and $t>0$ with $|x-y|> 8r$,
 one can easily check that $(s, w)\mapsto  p(s, w- y)$
is a parabolic function in $(-\infty, \infty)\times B(x, 2r)$.
Suppose  $  \Phi^{-1}(t) \le |x-y|$ and
let $(s,z)$ with $s\le t$ and  $\frac{|x-y|}{2}  \leq
 |z-y| \le \frac{3 |x-y|}{2}$. Then 
 by Theorem \ref{t:PHI},
 there is a constant
$c_4 \geq 1$
 so that for every $t \in (0,  T]$,
$$
\sup_{s \le t} p(s, z-y )\leq  c_4 p(t , z-y).
$$
Using this and the monotonicity of $ r \to p(t, r)$ we have
\begin{equation}\label{eq:ppun31}
 \sup_{ s\le t,
 \frac{|x-y|}{2}  \leq |z-y| \le \frac{3 |x-y|}{2}}
p(s, z-y)  \leq c_4\sup_{
 \frac{|x-y|}{2}  \leq |z-y| \le \frac{3 |x-y|}{2}}
 p(t, z-y)=c_4 p(t, |x-y|/2).
 \end{equation}
Combining  \eqref{eq:ppun31} and Lemma \ref{l:u_off1p} and Proposition \ref{step3} and using  the monotonicity of $ r \to p(t, r)$, we have
 for every $(t,x,y) \in (0, T]
\times D \times D$ with $ a_0 \Phi^{-1}(t) \le |x-y|$,
\begin{align*}
& p_{D}(t,x,y) \\
\le & c_5
\Big(\sqrt{\frac{\Phi(\delta_{D}(x))}{t}} \wedge 1\Big)
\left(   p(t, |x-y|/2)   +
 \Big(\sqrt{t \Phi(\delta_{D}(y))} \wedge  t\Big)
    j(|x-y|/3)\right)\\
\le & c_6
\Big(\sqrt{\frac{\Phi(\delta_{D}(x))}{t}} \wedge 1\Big)
\left( p(t, |x-y|/2) + p(t, |x-y|/3) \right)\\
\le & 2c_6
\Big(\sqrt{\frac{\Phi(\delta_{D}(x))}{t}} \wedge 1\Big)
p(t, |x-y|/3).
\end{align*}
In view of \eqref{e:dfghj1}, using the monotonicity of $ r \to p(t, r)$ again, 
the last inequality in fact holds for all $(t,x,y) \in (0, T]
\times D \times D$.  

Thus by semigroup properties of $p$ and $p_D$ and the symmetry of $(x,y) \to p_D(t,x,y)$, 
\begin{align*}
& p_{D}(t,x,y)=\int_D p_D(t/2, x,z) p_D(t/2, y,z)dz \\
\le & c_7
\Big(\sqrt{\frac{\Phi(\delta_{D}(x))}{t}} \wedge 1\Big)\Big(\sqrt{\frac{\Phi(\delta_{D}(y))}{t}} \wedge 1\Big)
 \int_D
p(t/2, |x-z|/3)p(t/2, |z-y|/3)dz\\
\le & c_8
\Big(\sqrt{\frac{\Phi(\delta_{D}(x))}{t}} \wedge 1\Big)\Big(\sqrt{\frac{\Phi(\delta_{D}(y))}{t}} \wedge 1\Big)
 \int_{\R^d}
p(t/2, x/3, z)p(t/2, z, y/3)dz\\
=&c_8
\Big(\sqrt{\frac{\Phi(\delta_{D}(x))}{t}} \wedge 1\Big)\Big(\sqrt{\frac{\Phi(\delta_{D}(y))}{t}} \wedge 1\Big)
 p(t, |x- y|/3).
\end{align*}
We have proved \eqref{e:tu1}.

\eqref{e:tu2} follows from \eqref{e:tu1}, Lemma \ref{lem:bf} and Theorem \ref{tm:main} (applying to  $p(t, |x-y|/3)$). \qed

\section{Lower bound estimates}\label{S:5}

Recall that  we always assume that   \eqref{e:jdouble} holds, that $\phi$  has no drift and that $H$ satisfies $ L_a(\gamma, C_L)$ and $ U_a(\delta, C_U)$ with $\delta<2$ and $\gamma >2^{-1}{\bf 1}_{\delta \ge 1}$
   for some $a > 0$.

Using
Lemma \ref{l:spH} from Section \ref{S:4}, 
 in this section we will prove Theorem \ref{t:main0}(b). The main ideas in this section come from \cite{CK}.
We first observe the following simple lemma.
\begin{lemma}\label{l:ne12} 
The function
  $ \H(\lambda):=
 \sup_{t \in (0,1]} \P_0 \left(|X_t|> \lambda\Phi^{-1}(t) \right)
 $
vanishes  at $\infty$, that is, $\lim_{\lambda \to  \infty} \H(\lambda)=0$.
\end{lemma}

\pf  
By \cite[Theorem 2.2]{CK} 
there exists a constant $c_1=c_1(d)>0$ such that
$$
\P_0 (|X_t|>r) \,\le\,  c_1\,  t/ \Phi( r)  \quad
\hbox{for } (t, r) \in (0, \infty)\times (0, \infty).
$$
Noting $\phi^{-1}(t^{-1})^{1/2}=\Phi^{-1}(t)^{-1}$, the above inequality implies that 
$$
\sup_{t \in (0,1]} \P_0 \left(|X_t|>  \lambda\Phi^{-1}(t) \right)
 \le  c_1\,   \sup_{t \in (0,1]} \frac{t}{ \Phi( \lambda\Phi^{-1}(t))}=  
  c_1\,   \sup_{t \in (0,1]} {t}{ \phi( \lambda^{-2} \phi^{-1}(t^{-1}))}.
$$
The condition $ L_a(\gamma, C_L)$ and Remark \ref{r:UL} imply that  for all $\lambda \ge 1$, 
$$
 \sup_{t \in (0,1]} {t}{ \phi(  \lambda^{-2} \phi^{-1}(t^{-1}))} \le  
 \sup_{t \in (0,1]} \frac{ \phi(  \lambda^{-2} \phi^{-1}(t^{-1}))}{ \phi(  \phi^{-1}(t^{-1}))} \le  c_2 \lambda^{-2\gamma}, 
 $$
which goes to zero as $\lambda\to \infty$.
 \qed

We now discuss some  lower bound estimates
of $p_{D}(t,x, y)$.
We first 
note that by Lemma \ref{l:spH}, 
 there exist $C_{3} \ge 1$ and $T_1 \in (0, 1 \wedge \Phi(R_0)]$ such that
\begin{align}\label{e:spH}
C_{3}^{-1}
\left(\sqrt{\frac{\Phi(\delta_{D}(x))}{t}} \wedge 1\right) \le \P_{x}(\tau_{D}>t) \,\le\, C_{3}\
\left(\sqrt{\frac{\Phi(\delta_{D}(x))}{t}} \wedge 1\right), \quad \text{ for all }x  \in D, t \in (0, T_1]\ .
\end{align}

For $x\in D$ we use $z_x$ to denote a point on $\partial D$ such that  $|z_x-x|=\delta_D(x)$ and ${\bf n}(z_x):=(x-z_x)/|z_x-x|$.
By a 
simple  geometric argument, one can easily see that  
\begin{align}
\label{e:geom}
x+r{\bf n}(z_x) \in D \quad \text{ for all } x \in D \text{ and } r \in [0, R_0/2]. 
\end{align}

\begin{lemma}\label{l:loww_01}
There exist $a_1>0$ and $M_1 >1 \vee 4a_1$ such that 
for all $a \in (0, a_1]$, $x \in D$  and $t \in (0, T_1]$,  we have that  
 $$ \P_x\left(X_t \in {D} \cap B(\xi^a_x(t), M_1\Phi^{-1}(t)) \text{ and } \Phi(\delta_{D}(X_t)) >  at \right)
   \ge   (2C_{3})^{-1} \left(\sqrt{\frac{\Phi(\delta_{D}(x))}{t}} \wedge 1\right)
   $$
   where $\xi^a_x(t):=x+a\Phi^{-1}(t){\bf n}(z_x)$ and 
   $C_{3}$ and $T_1$ are  the constants in 
   \eqref{e:spH}.
\end{lemma}

\pf
By  \eqref{e:tu1} and a change of variable, for every $ a>0$, $t \in (0, T_1]$ and $x \in D$,
\begin{align}\label{e:lbsp1}
&  \int_{
 \{ u \in D : \Phi(\delta_{D}(u)) \le  at\}
} p_{{D}}(t,x,u)du \nn\\
  \le & C_0\left(\sqrt{\frac{\Phi(\delta_{D}(x))}{t}} \wedge 1\right)
  \int_{
  \{ u \in D : \Phi(\delta_{D}(u)) \le  at\}} \left(\sqrt{\frac{\Phi(\delta_{D}(u))}{t}} \wedge 1\right) p(  t,  |x-u|/3)
   du  \nn\\
     \le & C_0\sqrt{a}   \left(\sqrt{\frac{\Phi(\delta_{D}(x))}{t}} \wedge 1\right)
  \int_{
  \{ u \in D : \Phi(\delta_{D}(u)) \le  at\}} p(  t,  |x-u|/3)  du  \nn\\
  \le &  C_0\sqrt{a}   \left(\sqrt{\frac{\Phi(\delta_{D}(x))}{t}} \wedge 1\right)  \int_{\R^d} p(  t, |x-u|/3)  du \nn\\
   =&  C_0 3^{d}\sqrt{a}   \left(\sqrt{\frac{\Phi(\delta_{D}(x))}{t}} \wedge 1\right)  \int_{\R^d} p(  t,  w)  dw
  =C_0 3^{d}\sqrt{a}   \left(\sqrt{\frac{\Phi(\delta_{D}(x))}{t}} \wedge 1\right).
 \end{align}
Choose $a_1>0$ small so that
$C_03^{d}\sqrt{a_1} \le (4 C_{3})^{-1}$
where $C_{3}$ is the constant in 
   \eqref{e:spH}.

For the rest of the proof, we assume that $x \in D$, $a \in (0, a_1]$ and $t \in (0, T_1]$. 
Since $\xi^a_x(t)=x+a \Phi^{-1}(t){\bf n}(z_x)$, 
for every ${\lambda} \ge 2a_1 $ and $ u \in {D} \cap B(\xi^a_x(t), \lambda\Phi^{-1}(t))^c$,  
we have
$$|x-u| \ge |\xi^a_x(t)-u|-|x-\xi^a_x(t)| \ge |\xi^a_x(t)-u|  -a_1 \Phi^{-1}(t) \ge (1-\frac{a_1}{\lambda}) |\xi^a_x(t)-u|\ge \frac12 |\xi^a_x(t)-u|.$$
Thus using this, \eqref{e:tu1} and  the monotonicity of $ r \to p(t, r)$,
 we have that for every  ${\lambda} \ge 2a_1 $,
\begin{align}
& \int_{{D} \cap B(\xi^a_x(t), \lambda\Phi^{-1}(t))^c} p_{{D}}(t,x,u)du
\nn \\
\le& C_0\left(\sqrt{\frac{\Phi(\delta_{D}(x))}{t}} \wedge 1\right)  \int_{{D} \cap B(\xi^a_x(t), \lambda\Phi^{-1}(t))^c}
p(  t,   |x-u|/3) du  \nn\\
 \le& C_0\left(\sqrt{\frac{\Phi(\delta_{D}(x))}{t}} \wedge 1\right) \int_{{D} \cap B(\xi^a_x(t), \lambda\Phi^{-1}(t))^c}p( t,  |\xi^a_x(t)-u|/6)  du  \nn\\
   \le& C_0\left(\sqrt{\frac{\Phi(\delta_{D}(x))}{t}} \wedge 1\right) \int_{B(0, \lambda\Phi^{-1}(t))^c}p( t,  6^{-1} y)
  dy   \nn\\
    \le&
 C_0 6^{d}\H(6^{-1} \lambda)\left(\sqrt{\frac{\Phi(\delta_{D}(x))}{t}} \wedge 1\right).\label{n:neq11} \end{align}
By Lemma \ref{l:ne12}, we can
choose $M_1>1 \vee 4a_1$ large so that
$
C_0 6^{d}\H(6^{-1} M_1) < (4 C_{3})^{-1}.
$
Then by 
   \eqref{e:spH}--\eqref{n:neq11} and our choice of $a_1$ and $M_1$, we conclude that
\begin{align}
 &\int_{
 \{u \in {D} \cap B(\xi^a_x(t), M_1\Phi^{-1}(t)):\Phi(\delta_{D}(u)) >  at\}} p_{{D}}(t,x,u)du \nn\\
 =&
  \int_{{D}} p_{{D}}(t,x,u)du-\int_{{D} \cap B(\xi^a_x(t), M_1\Phi^{-1}(t))^c} p_{{D}}(t,x,u)du -\int_{ \{u \in {D}: \Phi(\delta_{D}(u)) \le   at\}} p_{{D}}(t,x,u)du\nn \\
  \ge & (2 C_{3})^{-1} \left(\sqrt{\frac{\Phi(\delta_{D}(x))}{t}} \wedge 1\right).\nn
 \end{align}
\qed

The next result is easy to check (see the proof of  \cite[Lemma 2.5]{CKS6} for a similar computation). We skip the proof. 
\begin{lemma}\label{L:2.5}
For any given positive constants $c_1, r_1, T$
and $r_2>r_1$, there is a positive constant $c_2=c_2 (r_1, r_2,  T, c_1, \phi)$
so that
$$
\phi^{-1}(t^{-1})^{d/2}e^{-c_1 |x-y|^2\phi^{-1}(t^{-1})} \leq  c_2 t r^{-d}H(r^{-2})
\qquad \hbox{for every } r_1 \le r < r_2 (a\wedge 1)^{-1} \hbox{ and } t\in (0, T].
$$
\end{lemma}

\noindent
\textbf{Proof of Theorem \ref{t:main0}(b)}:
It is clear that 
any  bounded $C^{1,1}$ open set has the property that the path distance in any connected component of $D$
is comparable to the Euclidean distance.

By \eqref{e:jlow} and \cite[Proposition 3.6]{M}, we have 
  \begin{align}
  \label{e:jlow1}
    j(|x-y|) \geq c_0 |x-y|^{-d}H(|x-y|^{-2}),\, \quad \text{for all } x, y \in D
  \end{align}

Recall that $a_1>0$ and $M_1 >1 \vee 4a_1$ are the constants in Lemma \ref{l:loww_01} and   $C_{3}$ and $T_1$ are  the constants in 
   \eqref{e:spH}. We also recall that for $x\in D$, $z_x\in \partial D$ such that  $|z_x-x|=\delta_D(x)$ and ${\bf n}(z_x)=(x-z_x)/|z_x-x|$.
Without loss of the generality we assume that $T>3T_1$. 

Let $a_2 := a_1 \wedge (2^{-1}R_0/\Phi^{-1}(T))$.
For $x \in D$ and $t \in (0, T]$,  let $\xi_x(t):=x+a_2\Phi^{-1}(t){\bf n}(z_x)$. 
Note that $\xi_x(t) \in D$ by \eqref{e:geom}.
Define
\begin{align}\label{e:sB(x,t)}
\sB (x, t):= \left\{ z \in {D} \cap B(\xi_x(t), M_1\Phi^{-1}(t)):\delta_{D}(z) >  a_2\Phi^{-1}(t) \right\}.
\end{align}
Observe that, we have 
\begin{align}\label{e:deltage0}
\delta_{D}(u) \wedge \delta_{D}(v) \ge
 a_2  \Phi^{-1}(t), \quad  \text{for every } (u,v) \in \sB(x,t) \times \sB(y,t), 
 \end{align}
and 
\begin{align}\label{e:deltage}
 |x-y| - 2a_2 \Phi^{-1}(t)  \le  |\xi_x(t)-\xi_y(t)|  \le |x-y| + 2a_2 \Phi^{-1}(t), \qquad \end{align}
 Using \eqref{e:deltage}  we also have  that for every $(u,v) \in \sB(x,t) \times \sB(y,t)$,
 \begin{align}\label{e:7.2} 
 &|x-y| -\frac{5}{2} M_1 \Phi^{-1}(t) \le |x-y| -2(M_1+a_2) \Phi^{-1}(t) 
 \le |u-v|\nn\\
 & \le |x-y| +|u-\xi_x(t)|+|v-\xi_y(t)| +2a_2 \Phi^{-1}(t)
 \nn\\
 &
  \le |x-y| +2(M_1+ a_2)\Phi^{-1}(t)\le |x-y| +3M_1 \Phi^{-1}(t).
 \end{align}

\noindent
{\it Step1:} Suppose $t \in (0, 3T_1]$ and $x$ and $y$ are in the same connected component.  
By the semigroup property of $p_D$,
\begin{align}
&p_{D}(t,x,y)\,
\geq\, \int_{\sB(y,t)}\int_{\sB(x,t)}
p_{D}(t/3,x,u) p_{D}(t/3,u,v)p_{D}
(t/3,v,y)dudv \nonumber\\
\geq&  \left(\inf_{(u,v) \in \sB(x,t) \times
\sB(y,t)} p_{D}(t/3,u,v)\right)
\int_{\sB(y,t)}p_{{D}}(t/3,x,u)du
\int_{\sB(x,t)}
p_{D}(t/3,v,y)dv. \label{e:deltage00}
\end{align}

 When $|x-y|  \le 3M_1 \Phi^{-1}(t)$, by \eqref{e:deltage0} and  \eqref{e:7.2} $|u-v|\le  6M_1 \Phi^{-1}(t)$ and $\delta_{D}(u) \wedge \delta_{D}(v) \ge
 a_2  \Phi^{-1}(t)$ for $(u,v) \in \sB(x,t) \times \sB(y,t)$. Thus using  Theorem \ref{t:PHI} and Lemma \ref{lem:bf}(a) and Proposition \ref{step1}, we get
\begin{align}\label{e:fl11}
p_{D}(t/3,u,v) \ge c_0 p_{D}(c_1 t,u,u)  \ge c_{2} \Phi^{-1}(t)^{-d} \quad \text{ for every } (u,v) \in \sB(x,t) \times \sB(y,t).
\end{align}

When $|x-y|  > 3M_1 \Phi^{-1}(t)$, we have by \eqref{e:7.2} that for $(u,v) \in \sB(x,t) \times \sB(y,t)$,
$$
\frac{a_2}{4} \Phi^{-1}(t/3)  \le \frac{1}{2} M_1 \Phi^{-1}(t) \le  |u-v| \le (2|x-y|) \wedge (|x-y| + 3M_1 \Phi^{-1}(T)).
$$
Thus, by Lemma \ref{lem:bf}(a),  Propositions \ref{step3} and \ref{prop:on-D-low}(a) we have that for
$|x-y|  > 3M_1 \Phi^{-1}(t)$ and $t \le 3 T_1$,
\begin{align}
&\inf_{(u,v) \in \sB(x,t) \times \sB(y,t)} p_{D}(t/3,u,v) \ge
\inf_{(u,v) : 2^{-2} a_2 \Phi^{-1}(t/3) \le  |u-v|  \le (2|x-y|) \wedge (|x-y| + 3M_1 \Phi^{-1}(T))\atop
\delta_{D}(u) \wedge \delta_{D}(v) >  a_2 \Phi^{-1}(t/3)} p_{D}(t/3,u,v) \nn \\
 \ge &c_3\inf_{(u,v) :   |u-v|  \le (2|x-y|) \wedge (|x-y| + 3M_1 \Phi^{-1}(T))}   \left( t j(|u-v|)+  \phi^{-1}((t/3)^{-1})^{d/2}e^{-c_4 |u-v|^2\phi^{-1}((t/3)^{-1})}\right)   \nn  \\
\ge & c_5\left( t j((2|x-y|) \wedge (|x-y| + 3M_1 \Phi^{-1}(T)))+  \phi^{-1}(t^{-1})^{d/2}e^{-c_6|x-y|^2\phi^{-1}(t^{-1})}\right) 
.
 \label{e:flow11}
   \end{align}
We now apply Lemma \ref{l:loww_01}, \eqref{e:flow11} and \eqref{e:fl11} to 
\eqref{e:deltage00} and use \eqref{e:jlow1} to obtain \eqref{e:pDl1} for $t \le 3 T_1$ and $x$ and $y$ in the same connected component.  

\noindent
{\it Step2:} Suppose $t \in (3T_1, T]$ and $x$ and $y$ are in the same connected component.    By semigroup property of $p_D$ and Lemma \ref{l:loww_01},
\begin{align}
&p_{D}(t,x,y)\,
\geq\, \int_{\sB(y,T_1)}\int_{\sB(x,T_1)}
p_{D}(T_1,x,u) p_{D}(t-2T_1,u,v)p_{D}
(T_1,v,y)dudv \nonumber\\
\geq&  \left(\inf_{(u,v) \in \sB(x,T_1) \times
\sB(y,T_1)} p_{D}(t-2T_1,u,v)\right)\int_{\sB(y,T_1)}
\int_{\sB(x,T_1)}
p_{{D}}(T_1,x,u)p_{D}(T_1,v,y)dudv\nn\\
\geq& (2C_{3})^{-2}
\left(\inf_{(u,v) \in \sB(x,T_1) \times
\sB(y,T_1)} p_{D}(t-2T_1,u,v)\right) \left(\sqrt{\frac{\Phi(\delta_{D}(x))}{T_1}} \wedge 1\right)
 \left(\sqrt{\frac{\Phi(\delta_{D}(y))}{T_1}} \wedge 1\right)
 \nn\\
\geq& (2C_{3})^{-2}
\left(\inf_{(u,v) \in \sB(x,T_1) \times
\sB(y,T_1)} p_{D}(t-2T_1,u,v)\right) \left(\sqrt{\frac{\Phi(\delta_{D}(x))}{t}} \wedge 1\right)
 \left(\sqrt{\frac{\Phi(\delta_{D}(y))}{t}} \wedge 1\right)
     \label{e:deltage00n}
\end{align}

 When $|x-y|  \le 3M_1 \Phi^{-1}(t)$, by \eqref{e:deltage0} and  \eqref{e:7.2} $|u-v|\le  c_7 \Phi^{-1}(T_1)$ and $\delta_{D}(u) \wedge \delta_{D}(v) \ge
 a_2  \Phi^{-1}(T_1)$ for $(u,v) \in \sB(x,T_1) \times \sB(y,T_1)$. Thus using  Theorem \ref{t:PHI}  and Lemma \ref{lem:bf}(a) and Proposition \ref{step1}, we get
that for every $(u,v) \in \sB(x,T_1) \times \sB(y,T_1)$, \begin{align}\label{e:fl11n}
p_{D}(t-2T_1,u,v) \ge c_8 p_{D}(c_9 T_1,u,u)  \ge c_{10} \Phi^{-1}(c_9T_1)^{-d} \ge c_{11} \Phi^{-1}(t)^{-d}.
\end{align}

When $|x-y|  > 3M_1 \Phi^{-1}(t)$, we have by \eqref{e:7.2} that for $(u,v) \in \sB(x,T_1) \times \sB(y,T_1)$,
$$
\frac{a_2}{4} \Phi^{-1}(t-2T_1)  \le \frac{1}{2} M_1 \Phi^{-1}(t) \le  |u-v| \le (2|x-y|) \wedge (|x-y| + 3M_1 \Phi^{-1}(T)).
$$
Thus, by Lemma \ref{lem:bf}(a), Propositions \ref{step3} and \ref{prop:on-D-low}(1) 
we have that for
$|x-y|  > 3M_1 \Phi^{-1}(t)$ and $3 T_1 < t \le T$,
\begin{align}
&\inf_{(u,v) \in \sB(x,T_1) \times \sB(y,T_1)} p_{D}(t-2T_1,u,v)\nn \\
&\ge
\inf_{(u,v) : 2^{-2}a_2 \Phi^{-1}(t-2T_1) \le  |u-v|  \le (2|x-y|) \wedge (|x-y| + 3M_1 \Phi^{-1}(T_1))\atop
\delta_{D}(u) \wedge \delta_{D}(v) >  a_2 \Phi^{-1}(T_1)} p_{D}(t-2T_1,u,v) \nn \\
&\ge
\inf_{(u,v) :  2^{-2}a_2\Phi^{-1}(t-2T_1) \le  |u-v|  \le (2|x-y|) \wedge (|x-y| + 3M_1 \Phi^{-1}(T))\atop
\delta_{D}(u) \wedge \delta_{D}(v) >  a_2 \Phi^{-1}(  (a_2T_1/T) (t-2T_1))} p_{D}(t-2T_1,u,v) \nn \\
 \ge &c_{12} \inf_{(u,v) :   |u-v|  \le (2|x-y|) \wedge (|x-y| + 3M_1 \Phi^{-1}(T))}   \left( t j(|u-v|)+  \phi^{-1}((t/3)^{-1})^{d/2}e^{-c_{13} |u-v|^2\phi^{-1}((t/3)^{-1})}\right)   \nn  \\
\ge &c_{14} \left( t j((2|x-y|) \wedge (|x-y| + 3M_1 \Phi^{-1}(T)))+  \phi^{-1}(t^{-1})^{d/2}e^{-c_{15}|x-y|^2\phi^{-1}(t^{-1})}\right) 
.
 \label{e:flow11n}
   \end{align}
Combining \eqref{e:deltage00n} and  \eqref{e:flow11n} and using \eqref{e:jlow1} we obtain \eqref{e:pDl1} for $t \in (3T_1, T]$ and $x$ and $y$ are in the same connected component.  

\noindent
{\it Step3:} Suppose $t \in (0, T]$ and $x$ and $y$ are in different  connected components.  
We use Proposition \ref{p:loww21n}.
Then, thanks to \eqref{e:jlow1} and Lemma \ref{L:2.5}, we see that \eqref{e:pDl1} still holds.
\qed

Note that, in the proof of Theorem \ref{t:main0}(b),
the assumptions that  $D$ is connected and  the path distance in $D$
is comparable to the Euclidean distance, are only used to apply Proposition \ref {prop:on-D-low}.
Thus, following the proof of Theorem \ref{t:main0}(b) without applying Proposition \ref {prop:on-D-low}, we  have the following. 
\begin{prop} \label{p:loww21n} 
For every $C^{1,1}$ open set $D$ and $T>0$, 
there exist constants $c>0, M_1>1$
 such that  for every $(t,x,y) \in (0, T]
\times D \times D$, 
\begin{align*}
&p_{D}(t,x, y)\nn\\
\ge& 
 c 
  \left(\sqrt{\frac{\Phi(\delta_{D}(x))}{t}} \wedge 1\right)
\left(\sqrt{\frac{\Phi(\delta_{D}(y))}{t}} \wedge 1\right)  \begin{cases}
t \displaystyle \frac{H(|x-y|^{-2})}{|x-y|^{d}} 
 &
\hbox{if }|x-y|  > 3M_1 \Phi^{-1}(t),\\
\Phi^{-1}(t)^{-d}& \hbox{if }|x-y|   \le  3M_1 \Phi^{-1}(t).
\end{cases}
\end{align*}
\end{prop}

\noindent
\textbf{Proof of Theorem \ref{t:main0}(c)}:
Since $D$ is bounded and $j$ is non-increasing, Theorem \ref{t:main0}(a) and Proposition \ref{p:loww21n} imply that 
for every $(x,y) \in  D \times D$, 
$$
c^{-1} \Phi(\delta_{D}(x))^{1/2}\Phi(\delta_{D}(y))^{1/2} \le p_{D}(1,x, y) \le c \Phi(\delta_{D}(x))^{1/2}\Phi(\delta_{D}(y))^{1/2}.
$$
Using this, the proof of Theorem \ref{t:main0}(c) is almost identical to that of \cite[Theorems 1.3(iii)]{CKS9}, so we omit the proofs.
\qed

\noindent
\textbf{Proof of Theorem \ref{t:main1}.}
Either by the proof of Theorem \ref{t:main0} or by applying the main result in \cite{M} and our Propositions \ref{step3} and \ref{prop:on-D-low}(1)
  to \cite[Theorem 4.1 and 4.5]{CK}, 
the theorem holds true when $D$ is an upper half space $\{x=(\wt x,x_d)\in\bR^d:x_d>0\}$.
Then using the
``push inward" method of
\cite{CT} (see also  \cite[Theorem 5.8]{BGR3})
and our short time
heat kernel estimates in Theorem \ref{t:main0},   one can obtain
global sharp two-sided Dirichlet heat kernel estimates when 
$D$ is a domain consisting of all the points above the graph of a 
bounded globally $C^{1,1}$ function. 
We skip the proof since it would be almost identical to the one of 
\cite[Theorem 5.8]{BGR3}.
\qed

\section{Green function estimates}
In next two sections we use the notation $f(x)\asymp g(x), x\in I$, which means that there exist constants $c_1,c_2>0$ such that $c_1f(x)\leq g(x)\leq c_2g(x)$ for $x\in I$.

Recall that $\Phi(r)=({\phi(1/r^2)})^{-1}$ where $\phi$ is the Laplace exponent $\phi$ of the subordinator $S$. 
When $\phi$ satisfies $ L_a(\gamma, C_L)$ and $ U_a(\delta, C_U)$ with $\delta<2$ for some $a>0$, 
Green function estimates for   the corresponding  subordinate Brownian motion were already discussed in 
\cite{CKS9}. In this section we discuss Green function estimates when $\phi$  has no drift and that $H$ satisfies $ L_a(\gamma, C_L)$ and $ U_a(\delta, C_U)$ with $\delta<2$ and $\gamma >2^{-1}{\bf 1}_{\delta \ge 1}$
   for some $a > 0$.

By the exactly same proof as the one of \cite[Lemma 7.1]{CKS9}, we have the following.
\begin{lemma}\label{l:phi_a}
For every $r\in (0, 1]$ and every open subset $U$ of $\R^d$,
\begin{align}
\frac12\left(1\wedge \frac{r^2\Phi (\delta_U(x))^{1/2}\Phi(\delta_U(y))^{1/2}}{\Phi(|x-y|)}\right)
&\leq
\left(1\wedge \frac{r \Phi (\delta_U(x))^{1/2}}{\Phi(|x-y|)^{1/2}}\right)
\left(1\wedge \frac{r \Phi (\delta_U(y))^{1/2}}{\Phi(|x-y|)^{1/2}}\right)\nonumber\\
&\leq    1\wedge \frac{r^2\Phi (\delta_U(x))^{1/2}\Phi(\delta_U(y))^{1/2}}{\Phi(|x-y|)} .\label{e:12}
\end{align}
\end{lemma}

Since $\phi$
has no drift and  satisfies $ L_a(\gamma, C_L)$, by \cite[Lemma 1.3]{KM} for every $M>0$, we have
\bee\label{e:6.20}
 r\Phi '(r) \asymp \Phi (r) \quad \hbox{for } r\in (0, M].
 \eee
Note that, 
by Lemma
\ref{lem:inverse}, 
 for every $T>0$,
 there exists $C_T>1$ such that
\begin{equation} \label{e:77}
\frac{ \Phi^{-1} ( r)}{\Phi^{-1} ( R)} \ge C_T^{-1}\left(\frac{r}{R} \right)^{1/(2 \gamma)}  \quad \hbox{for }
0<r \le R\le T.
\end{equation}
Moreover, by Lemma \ref{lem:bf},\begin{equation} \label{e:78}
 \frac{ \Phi^{-1} ( r)}{\Phi^{-1} ( R)} \le \left(\frac{r}{R} \right)^{1/2}
 \quad \hbox{for }
0<r \le R< \infty.
\end{equation}

Recall $x_+=x\vee 0$.
\begin{lemma}\label{l:new1dim}
For $T,b,r>0$ and $d=1, 2$, set
\bee\label{e:7.4}
h_{T, d}(b,r)=b+ \Phi(r) \int_{\Phi(r)/T}^1 \left(1 \wedge \frac{ub}{\Phi(r)}\right) \frac{1}{u^2(\Phi^{-1}(u^{-1} \Phi(r)))^d}\, du+\frac{\Phi(r)}{r^d}\left(1 \wedge \frac{b}{\Phi(r)}\right).
\eee
Then,  for  $0<r\leq \Phi^{-1}(T/2)$  and $0<b \le T/2$,
$$ h_{T,d}(b, r) \asymp \frac{b}{r^d} \wedge \left( \frac{b}{\Phi^{-1}(b)^d}
+ \left( \int_{r}^{\Phi^{-1}(b)} \frac{\Phi(s)}{s^{d+1}}ds \right)_+ \, \right).
$$
\end{lemma}
\pf 
(a) The lemma for $d=1$ is given in \cite[Lemma 7.2]{CKS9} under the assumption that $\phi$ satisfies $ L_a(\gamma, C_L)$ and $ U_a(\delta, C_U)$ with for some $a>0$ $\delta<2$. Using \eqref{e:78} instead of the assumption $U_a(\delta, C_U)$ with $\delta<2$, the proof of (a) is the same as the that of  \cite[Lemma 7.2]{CKS9}.

\noindent
(b) We now assume that $d=2$. Using \eqref{e:6.20}--\eqref{e:78}, the proof is a simple modification of the one of \cite[Lemma 7.2]{CKS9}. We provide the proof  in details for the readers' convenience.  

For $(b, r)$ with
 $0<b <  \Phi(r)\le   T/2 $,
\begin{align*}
h_{T, 2} (b,r)& \asymp
b+ b\int_{\Phi(r)/T}^1  \frac{du}{u(\Phi^{-1}(u^{-1} \Phi(r)))^2}
  + \frac{b}{r^2}\\
  &=b+ \frac{b}{r^2}\int_{\Phi(r)/T}^1  \left(\frac{\Phi^{-1}(\Phi(r))}{\Phi^{-1}(u^{-1} \Phi(r))} \right)^2u^{-1}du + \frac{b}{r^2}.
\end{align*}
Since $\Phi(r) \le T/2$, by \eqref{e:77}--\eqref{e:78} we have
\begin{align*}
0<c_2= c_1^{-1}\int^{1}_{1/2}
 u^{\frac{1}{\gamma}-1} du  \le \int_{\Phi(r)/T}^1  \left(\frac{\Phi^{-1}(\Phi(r))}{\Phi^{-1}(u^{-1} \Phi(r))} \right)^2u^{-1}du \le c_1 \int^{1}_{0} du=c_1<\infty.
\end{align*}
Thus,
for  $0<b <  \Phi(r)\le   T/2 $, we have
\bee \label{e:7.5}
h_{T, 2} (b,r) \asymp
 \frac{b}{r^2}.\eee

On the other hand,
using the change of variable
$u=\Phi (r)/\Phi (s)$ and  integration by parts,  we have that  for $(b, r)$ with
$\Phi(r)\leq b \le T/2$,
 \begin{align}
& h_{T, 2}(b,r) \nn\\ = &
b+ \Phi(r) \int_{ \Phi(r)/b}^1 \frac{du}{u^2(\Phi^{-1}(u^{-1} \Phi(r)))^2}+b\int_{\Phi(r)/T}^{\Phi(r)/b} \frac{du}{u(\Phi^{-1}(u^{-1} \Phi(r)))^2}
+\frac{\Phi (r)}{r^2}
\nn\\
=& b+ \int_r^{\Phi^{-1}(b)} \frac{\Phi'(s)}{s^2} ds+ b\int_{\Phi^{-1}(b)}^{\Phi^{-1}(T)}
\frac{\Phi '(s)}{s^2\Phi (s)} ds + \frac{\Phi (r)}{r^2} \nn \\
=& b+ \left(\frac{b}{\Phi^{-1}(b)^2}-\frac{\Phi(r)}{r^2}\right)
+2 \int_{r}^{\Phi^{-1}(b)} \frac{\Phi(s)}{s^3}ds +
  b \int_{\Phi^{-1}(b)}^{\Phi^{-1}(T)}
\frac{\Phi '(s)}{s^2\Phi (s)} ds + \frac{\Phi (r)}{r^2} \nn \\
=&  b+ \frac{b}{\Phi^{-1}(b)^2} +2 \int_{r}^{\Phi^{-1}(b)} \frac{\Phi(s)}{s^3}ds
 +  b \int_{\Phi^{-1}(b)}^{\Phi^{-1}(T)}
\frac{\Phi '(s)}{s^2\Phi (s)} ds  .
 \label{e:gt1}
\end{align}
Since
$b\le T/2$,
by \eqref{e:78} and the fact that $\Phi^{-1}$ is increasing,
\begin{align}
 \frac{1}{\Phi^{-1}(b)^2}-\frac{1}{\Phi^{-1}(T)^2}\asymp \frac{1}{\Phi^{-1}(b)^2}\ge c_4\label{e:gt2}
\end{align}
for some $c_4>0$.
Using  \eqref{e:6.20} and  \eqref{e:gt2} in the second integral
   in \eqref{e:gt1}, we get that for $(b, r)$ with
$\Phi(r)\leq b \le  T/2$,
\begin{eqnarray}
  h_{T, 2} (b,r)&\asymp&  b+  \frac{b}{\Phi^{-1}(b)^2}
+ \int_{r}^{\Phi^{-1}(b)} \frac{\Phi(s)}{s^3}ds+
b\int_{\Phi^{-1}(b)}^{\Phi^{-1}(T)}
\frac{1}{s^3} ds  \nn \\
&=& b+  \frac{b}{\Phi^{-1}(b)^2}  +\int_{r}^{\Phi^{-1}(b)} \frac{\Phi(s)}{s^3}ds+\frac{b}{2}
 \left( \frac{1}{\Phi^{-1}(b)^2}-\frac{1}{\Phi^{-1}(T)^2} \right) \nn \\
& \asymp&  \frac{ b}{\Phi^{-1}(b)^2}
+2\int_{r}^{\Phi^{-1}(b)} \frac{\Phi(s)}{s^3}ds.  \label{e:7.8}
\end{eqnarray}

Since $\Phi (s)$ is an increasing function,
when $0<\Phi (r) \leq b$, we have
$$\frac{ b}{\Phi^{-1}(b)^2}
+2\int_{r}^{\Phi^{-1}(b)} \frac{\Phi(s)}{s^3}ds
\leq \frac{ b}{\Phi^{-1}(b)^2}
+2 b \int_{r}^{\Phi^{-1}(b)} \frac{1}{s^3}ds = \frac{b}{r^2},
$$
while when $\Phi (r)\geq b>0$,
$$\frac{ b}{\Phi^{-1}(b)^2}
+ \left( \int_{r}^{\Phi^{-1}(b)} \frac{\Phi(s)}{s^3}ds \right)_+
= \frac{ b}{\Phi^{-1}(b)^2} \geq   \frac{b}{r^2} .
$$
Thus  combining this  with \eqref{e:7.5} and \eqref{e:7.8} we establishes the lemma.
\qed

Recall that the Green function $G_D(x, y)$ of $X$ on $D$ is defined as  
$G_D(x, y)
=\int_0^\infty p_D(t, x, y)dt$.
As an application of Theorems \ref{t:main0} and  \ref{t:main1}, we derive the sharp two sided estimates on the Green functions  of $X$ on 
 bounded $C^{1, 1}$ open sets.
For notational convenience, let
\begin{align}
a(x, y):=\sqrt{\Phi(\delta_D(x))}\sqrt{\Phi(\delta_D(y))} \label{e:axy}
\end{align}
and 
\begin{align}
g(x,y):=
\begin{cases}
\displaystyle \frac{\Phi(|x-y|)}{|x-y|^{d}}\left(1\wedge \frac{\Phi(\delta_D(x))}{\Phi(|x-y|)}\right)^{1/2}  \left(1\wedge \frac{\Phi(\delta_D(y))}{\Phi(|x-y|)}\right)^{1/2},&\mbox{ when } d> 2\\
\displaystyle \frac{a(x, y)}{|x-y|^d}\wedge 
\left(\frac{a(x, y)}{\Phi^{-1}(a(x, y))^d}
+\left(\int_{|x-y|}^{\Phi^{-1}(a(x, y))} \frac{\Phi(s)}{s^{d+1}}ds\right)^{+}\right), &\mbox{ when } d \le 2.
\end{cases}\label{e:gxy}
\end{align}

\begin{thm}\label{t:green} Assume that   \eqref{e:jdouble} holds, that $\phi$  has no drift and that $H$ satisfies $ L_a(\gamma, C_L)$ and $ U_a(\delta, C_U)$ with $\delta<2$ and $\gamma >2^{-1}{\bf 1}_{\delta \ge 1}$
   for some $a > 0$.
Suppose that  $D$ is a bounded $C^{1,1}$ open set in $\R^d$, $d \ge 1$,  with characteristics $( R_0, \Lambda)$. 

\begin{itemize}
\item[\rm (i)]
There exists
$c_1 >0$ depending only on $\text{diam}(D),   R_0, \Lambda, d$ and  $\phi$
such that 
$$
G_D(x, y)\ge c_1 \frac{\Phi(|x-y|)} {|x-y|^{d}}
 \left(1\wedge \frac{  \Phi (\delta_D(x))}{ \Phi(|x-y|)}\right)^{1/2}
 \left(1\wedge \frac{  \Phi (\delta_D(y))}{ \Phi(|x-y|)}\right)^{1/2}, \quad x,y \in D .
$$

\item[\rm (ii)]
There exists
$c_2 >0$ depending only on $\text{diam}(D), R_0, \Lambda,  d$ and  $\phi$
such that
$$
G_D(x, y)\le c_2 \frac{ a(x,y)}{ |x-y|^d}, \quad x,y \in D .
$$
\item[\rm (iii)]
 If  $D$ is connected,  then 
$$
G_D(x, y)\asymp g(x,y), \quad x,y \in D.
$$

\end{itemize}
\end{thm}

\pf 
Put $T=
      2\Phi({\rm diam} (D))$.
  
  \noindent  
      (i) Let $M_1>0$ be the constant in Proposition \ref{p:loww21n} with our $T$. 
By Proposition \ref{p:loww21n}
 for every $(t,x,y) \in (0, T]
\times D \times D$ with  $|x-y|   \le  3M_1 \Phi^{-1}(t)$, 
\begin{align*}
p_{D}(t,x, y)
\ge
 c_1\left(1\wedge \frac{\Phi (\delta_D(x))}t \right)^{1/2} \left(1\wedge \frac{\Phi (\delta_D(y))}t \right)^{1/2}
(\Phi^{-1}(t))^{-d}.
\end{align*}
Thus, noting that $2\Phi(|x-y|) \le T$, we have 
\begin{align*}
&G_D(x, y)  \ge c_1  \int^{2\Phi(|x-y|)}_{\Phi(|x-y|/(3M_1))}
\left(1\wedge \frac{\Phi (\delta_D(x))}t \right)^{1/2} \left(1\wedge \frac{\Phi (\delta_D(y))}t \right)^{1/2}
(\Phi^{-1}(t))^{-d}dt\\
& \ge \frac{c_1 2^{-1}}{\Phi^{-1}( 2\Phi(|x-y|))^d  } \left(1\wedge \frac{ \Phi (\delta_D(x))^{1/2}}{ \Phi(|x-y|)^{1/2} }\right) 
\left(1\wedge \frac{ \Phi (\delta_D(y))^{1/2}}{ \Phi(|x-y|)^{1/2} }\right) \int^{2\Phi(|x-y|)}_{\Phi(|x-y|/(3M_1))} dt\\
& \ge c_2 \frac{\Phi(|x-y|)} {|x-y|^{d}}  \left(1\wedge \frac{ \Phi (\delta_D(x))^{1/2}}{ \Phi(|x-y|)^{1/2} }\right) \left(1\wedge \frac{ \Phi (\delta_D(y))^{1/2}}{ \Phi(|x-y|)^{1/2} }\right).
\end{align*}
We have proved part (i) of the theorem.

\medskip

 \noindent     
   (ii)    
It follows from Theorem \ref{t:main0}(c) that
\begin{equation}\label{e:newgfe}
\int^\infty_Tp_D(t, x, y)dt \asymp \Phi (\delta_D(x))^{1/2}
  \Phi (\delta_D(y))^{1/2}, \quad x,y \in D.
\end{equation}

By Theorem \ref{t:main0}(a) and \eqref{e:uppcom},  there exists $c_3>0$ such that for $(t, x, y)\in (0, T]\times D\times D$,
\begin{align}
&p_D(t, x, y)\nn \\
\leq &c_3\left(1\wedge \frac{\Phi (\delta_D(x))}t \right)^{1/2} \left(1\wedge \frac{\Phi (\delta_D(y))}t \right)^{1/2}
\left( (\Phi^{-1}(t))^{-d}\wedge \frac{t}{|x-y|^d\Phi(|x-y|)}\right)\label{e:nup} .
\end{align}

By the change of variable
$u= \frac{\Phi(|x-y|)}{t}$ and the fact that $t\to \Phi^{-1}(t)$ is increasing, we have
\begin{eqnarray}
&& \int_0^{T} \left(1\wedge \frac{\Phi (\delta_D(x))}t \right)^{1/2} \left(1\wedge \frac{\Phi (\delta_D(y))}t \right)^{1/2}
\left( (\Phi^{-1}(t))^{-d}\wedge \frac{t}{|x-y|^d\Phi(|x-y|)}\right)  dt
\nonumber\\
&=&\frac{\Phi(|x-y|)} {|x-y|^{d}}  \left(\int_{\Phi(|x-y|)/T}^1+\int_1^\infty\right)
 u^{-2}\left(   \left( \frac{ \Phi^{-1} (  ut)}{   \Phi^{-1} (  t)}  \right)^d
\wedge u^{-1} \right)   \nonumber\\
&&\qquad \qquad \qquad \qquad  \times\left(1\wedge \frac{ {\sqrt u} \Phi (\delta_D(x))^{1/2} }{ \Phi(|x-y|)^{1/2} }\right)\left(1\wedge \frac{
{\sqrt u} \Phi (\delta_D(y))^{1/2}} {\Phi(|x-y|)^{1/2} }\right) du
\nonumber\\
&\asymp&
 \frac{\Phi(|x-y|)} {|x-y|^{d}}   \int_{\Phi(|x-y|)/T}^1  u^{-2}\left(  \frac{ |x-y|}{   \Phi^{-1} (  u^{-1} \Phi(|x-y|))} \right)^d
  \left(1\wedge \frac{  u a(x,y) }{ \Phi(|x-y|) }\right)  du
 \nonumber\\
&& + \frac{\Phi(|x-y|)} {|x-y|^{d}} \int_1^\infty
 u^{-3}  \left(1\wedge \frac{ {\sqrt u} \Phi (\delta_D(x))^{1/2} }{ \Phi(|x-y|)^{1/2} }\right) \left(1\wedge \frac{
{\sqrt u} \Phi (\delta_D(y))^{1/2}} {\Phi(|x-y|)^{1/2} }\right) du
\nonumber\\
&=:&I+II.        \label{e:2}
\end{eqnarray}
In the fourth line of the display above,
we used Lemma \ref{l:phi_a}.

Since ${ \Phi(|x-y|)}/ { a(x,y)}\geq \Phi (|x-y|)/\Phi ({\rm diam}(D))\ge 2\Phi (|x-y|)/T$,  by \eqref{e:78},
\begin{align}
 I \le & \frac{ a(x,y)}{|x-y|^d}\int^{1}_{\Phi(|x-y|)/T}
 \frac{ |x-y|^d }{   \Phi^{-1} (  u^{-1} \Phi(|x-y|))^d}\, u^{-1} du \nn\\
 = & \frac{ a(x,y)}{|x-y|^d}\int^{1}_{\Phi(|x-y|)/T}
 \left(\frac{ \Phi^{-1} (  \Phi(|x-y|))}{   \Phi^{-1} (  u^{-1} \Phi(|x-y|))}\right)^d\, u^{-1} du \nn\\
 \le& c_4 \frac{ a(x,y)}{|x-y|^d}
 \int^{1}_{0}
 u^{\frac{d}{2}-1} du  \,=\,  {2c_4}{d}^{-1} \frac{ a(x,y)}{|x-y|^d}. \label{e:555}
\end{align}

On the other hand, by Lemma \ref{l:phi_a}
\begin{align}
 & II  \,\le \,
 \frac{\Phi(|x-y|)} {|x-y|^{d}}  \int_{1}^\infty
  u^{-2} \, \left(u^{-1/2} \wedge
\frac{   \Phi (\delta_D(x))^{1/2}}{ \Phi(|x-y|)^{1/2} }\right)
\left(u^{-1/2} \wedge \frac{  \Phi (\delta_D(y))^{1/2}}{
\Phi(|x-y|)^{1/2} }\right) du \nonumber\\
&\leq
 \frac{\Phi(|x-y|)} {|x-y|^{d}}
\int_{1}^\infty  u^{-2} \,
  \left(1 \wedge
\frac{   \Phi (\delta_D(x))^{1/2}}{ \Phi(|x-y|)^{1/2} }\right) \left(1
\wedge \frac{  \Phi (\delta_D(y))^{1/2}}{
\Phi(|x-y|)^{1/2} }\right) du \nonumber\\
& = \frac{\Phi(|x-y|)} {|x-y|^{d}}    \left(1\wedge \frac{
\Phi (\delta_D(x))^{1/2}}{ \Phi(|x-y|)^{1/2} }\right) \left(1\wedge \frac{
\Phi (\delta_D(y))^{1/2}}{ \Phi(|x-y|)^{1/2} }\right)\le \frac{ a(x,y)}{|x-y|^d} . \label{e:4}
\end{align}
Part (ii) of the theorem now follows from \eqref{e:newgfe},  \eqref{e:nup},  \eqref{e:555} and  \eqref{e:4}.

\medskip
\noindent
(iii)
For the remainder of the proof we assume either that  $D$ is connected or that $H$ satisfies $ L_a(\gamma, C_L)$ and $ U_a(\delta, C_U)$ with $\delta<2$ for some  $a>0$.  Then   by Theorems \ref{t:main0}(b) and  \ref{t:main0}
\begin{eqnarray}
p_D(t, x, y) \ge c_5\left(1\wedge \frac{\Phi (\delta_D(x))}t \right)^{1/2} \left(1\wedge \frac{\Phi (\delta_D(y))}t \right)^{1/2}
(\Phi^{-1}(t))^{-d}e^{-c_0 \frac{|x-y|^2}{\Phi^{-1}(t)^2}}   \label{e:nlo}.
\end{eqnarray}

Since by \eqref{e:77}
$$
\frac{|x-y|}{\Phi^{-1}( \Phi(|x-y|)/u)} =\frac{\Phi^{-1}( \Phi( |x-y|))}{\Phi^{-1}( \Phi(|x-y|)/u)}  \le  
c_6 u^{1/(2\gamma)} \quad  \text {if } { u >1},
$$
using this, by the change of variable
$u= \frac{\Phi(|x-y|)}{t}$ and the fact that $t\to \Phi^{-1}(t)$ is increasing, we have
\begin{eqnarray}
&&\int_0^{T} p_D(t, x, y) dt \nn\\&\ge& c_5 \int_0^{T} \left(1\wedge \frac{\Phi (\delta_D(x))}t \right)^{1/2} \left(1\wedge \frac{\Phi (\delta_D(y))}t \right)^{1/2}
 (\Phi^{-1}(t))^{-d}e^{-c_0 \frac{|x-y|^2}{\Phi^{-1}(t)^2}}  dt
\nonumber\\
&=&c_5\frac{\Phi(|x-y|)} {|x-y|^{d}}  \left(\int_{\Phi(|x-y|)/T}^1+\int_1^\infty\right)
 u^{-2}  \left( \frac{ \Phi^{-1} (  ut)}{   \Phi^{-1} (  t)}  \right)^d
e^{-c_0 \frac{|x-y|^2}{\Phi^{-1}( \Phi(|x-y|)/u)^2}}  \nonumber\\
&&\qquad \qquad \qquad  \times \left(1\wedge \frac{ {\sqrt u} \Phi (\delta_D(x))^{1/2} }{ \Phi(|x-y|)^{1/2} }\right)\left(1\wedge \frac{
{\sqrt u} \Phi (\delta_D(y))^{1/2}} {\Phi(|x-y|)^{1/2} }\right) du
\nonumber\\
&\ge& c_5 e^{-c_0}
 \frac{\Phi(|x-y|)} {|x-y|^{d}}   \int_{\Phi(|x-y|)/T}^1  u^{-2}\left(  \frac{ |x-y|}{   \Phi^{-1} (  u^{-1} \Phi(|x-y|))} \right)^d
  \left(1\wedge \frac{  u a(x,y) }{ \Phi(|x-y|) }\right)  du
 \nonumber\\
&& + c_5\frac{\Phi(|x-y|)} {|x-y|^{d}} \int_1^\infty
 u^{-2} e^{-c_0 c_6^2 u^{1/\gamma}} \left(1\wedge \frac{ {\sqrt u} \Phi (\delta_D(x))^{1/2} }{ \Phi(|x-y|)^{1/2} }\right) \left(1\wedge \frac{
{\sqrt u} \Phi (\delta_D(y))^{1/2}} {\Phi(|x-y|)^{1/2} }\right) du
\nonumber\\
&=:&c_5(I+III).        \label{e:2-1}
\end{eqnarray}

Clearly, we have 
\begin{align}
&III \ge  \frac{\Phi(|x-y|)} {|x-y|^{d}}  \left(1\wedge \frac{ \Phi (\delta_D(x))^{1/2}}{ \Phi(|x-y|)^{1/2} }\right) \left(1\wedge \frac{ \Phi (\delta_D(y))^{1/2}}{ \Phi(|x-y|)^{1/2} }\right)
\int_1^\infty u^{-2}e^{-c_0 c_6^2 u^{1/\gamma}}du.  \label{e:4-1}
\end{align}

\medskip
Suppose that $d \le 2$. 
Let $h_{T, d}  (a, r)$ be defined as in \eqref{e:7.4}.
Since $a(x,y) \le \Phi({\rm diam} (D)) = T/2$,
we have by
\eqref{e:newgfe}--\eqref{e:2}, \eqref{e:4}, \eqref{e:2-1}, \eqref{e:4-1}  and Lemma  \ref{l:phi_a} that
$G_D (x, y)\asymp h_T (a(x, y), |x-y|)$. Now, part (iii) of the theorem for $d \le 2$  follows from
Lemmas \ref{l:new1dim}.

Suppose $2<d$, then we have  that
 \begin{eqnarray}
&&\frac{\Phi(|x-y|)} {|x-y|^{d}}   \int_{\Phi(|x-y|)/T}^1  u^{-2}\left(  \frac{ |x-y|}{   \Phi^{-1} (  u^{-1} \Phi(|x-y|))} \right)^d
  \left(1\wedge \frac{  u \Phi (\delta_D(x))^{1/2}
  \Phi (\delta_D(y))^{1/2} }{ \Phi(|x-y|) }\right)  du \nn \\&=&
\frac{\Phi(|x-y|)} {|x-y|^{d}}   \int_{\Phi(|x-y|)/T}^1  u^{-2}\left(  \frac{ |x-y|}{   \Phi^{-1} (  u^{-1} \Phi(|x-y|))} \right)^d
  \left(1\wedge \frac{  u \Phi (\delta_D(x))^{1/2}
  \Phi (\delta_D(y))^{1/2} }{ \Phi(|x-y|) }\right)  du\nn\\
&\leq &c_7 \frac{\Phi(|x-y|)} {|x-y|^{d}}
 \left(1\wedge \frac{   \Phi (\delta_D(x))^{1/2}
  \Phi (\delta_D(y))^{1/2} }{ \Phi(|x-y|) }\right) \int_0^1 u^{d/2-2} \, du\nonumber \\
 &= &\frac{2c_7}{d-2} \frac{\Phi(|x-y|)} {|x-y|^{d}}
  \left(1\wedge \frac{   \Phi (\delta_D(x))^{1/2}
  \Phi (\delta_D(y))^{1/2} }{ \Phi(|x-y|) }\right). \label{e:5}
\end{eqnarray}
The case $d>2$ of (iii) now follows from
part (i) of  the theorem, Lemma \ref{l:phi_a},
\eqref{e:newgfe}, \eqref{e:nup}, \eqref{e:4} and \eqref{e:5}.
\qed

We now consider the Green function  estimates for half space-like domains. Here we will give a sketch of the proofs only.

The proof of the next lemma is very similar (and simpler) to the one of Lemma \ref{l:new1dim} so we skip the proof. 

\begin{lemma}\label{l:new1dim0}
Suppose that   \eqref{e:jdouble} holds, that $\phi$  has no drift and that $H$ satisfies $ L_0(\gamma, C_L)$ and $ U_0(\delta, C_U)$ with $\delta<2$ and $\gamma >2^{-1}{\bf 1}_{\delta \ge 1}$.
For $b,r>0$ and $d=1, 2$, set
$$
h_{d}(b,r)= \Phi(r) \int_{0}^1 \left(1 \wedge \frac{ub}{\Phi(r)}\right) \frac{1}{u^2(\Phi^{-1}(u^{-1} \Phi(r)))^d}\, du+\frac{\Phi(r)}{r^d}\left(1 \wedge \frac{b}{\Phi(r)}\right).
$$
Then,   for  $r, b>0$,
$$
h_{d}(b, r) 
\asymp \frac{b}{r^d} \wedge \left( \frac{b}{\Phi^{-1}(b)^d}
+ \left( \int_{r}^{\Phi^{-1}(b)} \frac{\Phi(s)}{s^{d+1}}ds \right)_+ \, \right).$$
\end{lemma}

Recall that  $g(x, y)$ is defined in \eqref{e:gxy}.

\begin{thm}\label{t:gmain1}
  Let $S=(S_t)_{t\geq 0}$ be a subordinator with zero drift whose Laplace  exponent  is {$\phi$}  
  and let $X=(X_t)_{t\geq 0}$ be the corresponding subordinate Brownian motion in $\R^d$.  
  Suppose that $D$ is a domain consisting of all the points above the graph of a 
bounded globally $C^{1,1}$ function and $H$ satisfies $ L_0(\gamma, C_L)$ and $ U_0(\delta, C_U)$ with $\delta<2$.  Then 
$$
G_D(x,y) \asymp g(x,y), \quad \text{for } x,y \in D.$$
 \end{thm}
\pf 
By Theorem \ref{t:main1} and \eqref{e:uppcom},
$$
G_D(x,y) \le c_1 \int_0^\infty 
\left(1\wedge \frac{\Phi (\delta_D(x))}t \right)^{1/2} \left(1\wedge \frac{\Phi (\delta_D(y))}t \right)^{1/2}
\left( (\Phi^{-1}(t))^{-d}\wedge \frac{t}{|x-y|^d\Phi(|x-y|)}\right)dt
$$
and
$$
G_D(x,y) \ge c_2 \int_0^\infty 
\left(1\wedge \frac{\Phi (\delta_D(x))}t \right)^{1/2} \left(1\wedge \frac{\Phi (\delta_D(y))}t \right)^{1/2}
(\Phi^{-1}(t))^{-d}e^{-c_0 \frac{|x-y|^2}{\Phi^{-1}(t)^2}}dt.
$$
Thus by following the argument in Theorem \ref{t:green} one can easily see that 
for $d>2$,
$$
G_D(x,y) \asymp 
\frac{\Phi(|x-y|)}{|x-y|^{d}}\left(1\wedge \frac{\Phi(\delta_D(x))}{\Phi(|x-y|)}\right)^{1/2}  \left(1\wedge \frac{\Phi(\delta_D(y))}{\Phi(|x-y|)}\right)^{1/2}.
$$
and, for $d \le 2$, 
$G_D(x,y) \asymp  h_{d}(a(x,y),|x-y|)$. Thus the theorem follows by this and Lemmas \ref{lem:bf}(b) and  \ref{l:new1dim0}.
\qed

\section{Examples}\label{s:ex}
  Suppose that  $D$ is a { bounded $C^{1,1}$} open set with diam$(D) <1/2$ and $\phi$ is either 
  $$\text{ (i) }\phi(\lambda)=\frac{\lambda}{\log(1+\lambda^{\beta/2})},  \quad \text{ where } \beta\in (0,2), \quad \text{or} \quad \text{ (ii) }\phi(\lambda)=\frac{\lambda}{\log(1+\lambda)}-1.$$ 
Then 
$\phi(\lambda)-\lambda\phi'(\lambda)$ satisfies $ L_a(\gamma, C_L)$ and $ U_a(\delta, C_U)$ with $2^{-1}<\gamma<\delta<2$
where $a=0$ for the case (i)  and  $a>0$ for the case (ii).
 It is easy to check that  we have 
$$ \phi^{-1}(\lambda)\asymp                            \lambda\log\lambda \quad \text { and  } \quad      
 H(\lambda) \asymp 
                   \frac{\lambda}{(\log\lambda)^2}, \quad \lambda\geq 2.
$$
Moreover, 
\bee \label{e:7.16}
\Phi(r)=1/ \phi(1/r^{2}) 
\asymp r^2 \log (1/r )
\quad \text{for }  0<r\leq 1/2, 
\eee
and 
$$\frac{1}{\sqrt{\phi(1/\lambda^2)}}\asymp 
                  {\lambda}\sqrt{\log(\lambda^{-1})}, \qquad  \lambda \le  1/2\,.$$
Thus by Theorem \ref{t:main0},  for $0<t<1/2$, 
\begin{align}\label{e:example1}
& p_D(t,x,y)\ge c_1
 \left(1\wedge\frac{\delta_D(x)}{\sqrt{  t}}   \sqrt{\log(1/\delta_D(x))}    \right) \left(1\wedge\frac{\delta_D(y)}{\sqrt{  t}}   \sqrt{\log(1/\delta_D(y))}    \right) \nn\\&  \times
\left[\left(t^{-d/2}\Big(\log\frac{1}{t}\Big)^{d/2}  \right)\wedge  \left( \frac{t(\log\frac{1}{|x-y|})^{-2}}{|x-y|^{d+2}}+  t^{-d/2}\Big(\log\frac{1}{t}\Big)^{-d/2}e^{-a_L\frac{|x-y|^2}{t}\log\frac{1}{t}}\right) \right]\,,
\end{align}
and 
\begin{align}\label{e:example2}
& p_D(t,x,y)\le c_2
 \left(1\wedge\frac{\delta_D(x)}{\sqrt{  t}}   \sqrt{\log(1/\delta_D(x))}    \right) \left(1\wedge\frac{\delta_D(y)}{\sqrt{  t}}   \sqrt{\log(1/\delta_D(y))}    \right) \nn\\&  \times
\left[ \left(t^{-d/2}\Big(\log\frac{1}{t}\Big)^{d/2}  \right)\wedge  \left( \frac{t(\log\frac{1}{|x-y|})^{-2}}{|x-y|^{d+2}}+  t^{-d/2}\Big(\log\frac{1}{t}\Big)^{-d/2}e^{-a_U\frac{|x-y|^2}{t}\log\frac{1}{t}}\right)\right]\,.
\end{align}

We now assume that $d=2$ and $D$ is a { bounded $C^{1,1}$} open set in $\R^2$ with sufficiently small diameter.
 We will give the sharp estimates of  the Green function on $D$.

  There is a constant $c_0\in (0, 1)$ so that
\begin{equation}\label{e:LWf}
c_0  \left(\frac{s}{\log(1/s)} \right)^{1/2} \leq \Phi^{-1} (s) \leq c_0^{-1}
  \left(\frac{s}{\log(1/s)} \right)^{1/2}
\quad \hbox{ for } s \in (0, \Phi (1/2)].
\end{equation}
Suppose   $0<r \le \Phi^{-1}(b)\le 1/2$. Then
\begin{eqnarray} \label{e:7.18}
&& \int_{r}^{\Phi^{-1}(b)}  \frac{\Phi(s)}{s^3} ds
 \asymp
\int_{r}^{\Phi^{-1}(b)}  \frac{ \log (1/s )}{s} ds \nn \\
&=&
\frac1{2}\left(   (\log(1/r))^{2}  -(\log(1/\Phi^{-1}(b)))^{2} \right) =\frac1{2}
\log^+(\Phi^{-1}(b)/r)\log^+(1/(r\Phi^{-1}(b))).
\end{eqnarray}

Let 
\begin{align}
\label{e:7.89}
b(x,y):={\delta_D(x)\delta_D(y)} \sqrt{\log(1/\delta_D(x))\log(1/\delta_D(y))}\asymp \Phi (\delta_D(x))^{1/2} \Phi (\delta_D(y))^{1/2}=a(x,y).
\end{align}
Note that by \eqref{e:LWf}
\begin{align}
\label{e:7.90}
\Phi^{-1} (a(x,y)) \asymp  
 \left(\frac{b(x,y)}{\log(1/b(x,y))} \right)^{1/2}
\end{align}
and, so 
\begin{align}
\label{e:7.91}
\frac{a(x,y)}{\Phi^{-1}(a(x,y))^2} \asymp\log\left[b(x,y)^{-1}\right].
\end{align}
Applying expressions $\Phi(|x-y|) \asymp |x-y|^2 \log (1/|x-y| )$ and \eqref{e:7.18}--\eqref{e:7.91} 
 to  Theorem  \ref{t:green}(iii), we have the following explicit estimates:
\begin{align*}
G_D(x, y)\asymp
 \frac{ b(x,y)}{|x-y|^2} \wedge \left( \log^{+}  \left[  \frac{b(x,y)}{|x-y|^2\log(1/b(x,y))}  \right]
  \log^{+}  \left[      \frac{\log(1/b(x,y))}{|x-y|^2b(x,y)}  \right] +\log\left[b(x,y)^{-1}\right]  \right).\end{align*}

\vspace{.1in}

\noindent
{\bf Acknowledgements:} The first named author is grateful to 
Joohak Bae, Renming Song and Zoran Vondra\v{c}ek for
reading the earlier drafts of this paper and giving helpful comments.
The first named author is also grateful to  Tomasz Grzywny and referees for pointing out errors in the earlier draft of this paper.

\end{document}